\pdfoutput=1
\RequirePackage{ifpdf}
\ifpdf 
\documentclass[pdftex]{sigma}
\else
\documentclass{sigma}
\fi

\usepackage{mathrsfs,fouridx}
\usepackage{tikz,calligra}

\numberwithin{equation}{section}

\newtheorem{Theorem}{Theorem}[section]
\newtheorem{Corollary}[Theorem]{Corollary}
\newtheorem{Lemma}[Theorem]{Lemma}
\newtheorem{Proposition}[Theorem]{Proposition}
\newtheorem{Conjecture}[Theorem]{Conjecture}

\newcommand{\Int}{\int\limits}

\renewcommand{\leq}{\leqslant}
\renewcommand{\geq}{\geqslant}

\DeclareMathOperator*{\pf}{pf}
\DeclareMathOperator{\core}{2\text{-}core}
\newcommand{\twocore}[1]{\core(#1)}
\newcommand{\spec}[1]{\langle#1\rangle}
\renewcommand{\subseteq}{\subset}

\DeclareMathOperator{\symp}{sp}
\DeclareMathOperator{\ortho}{o}
\DeclareMathOperator{\so}{so}
\DeclareMathOperator{\sln}{sl}
\newcommand{\iup}{\hspace{0.5pt}\mathrm{i}}
\DeclareMathOperator{\dup}{d\hspace{-1.5pt}}
\DeclareMathOperator{\odd}{o}

\newcommand{\floor}[1]{\lfloor#1\rfloor}
\newcommand{\qhypc}[2]{\fourIdx{}{#1}{}{#2}\phi}

\newcommand{\abs}[1]{\vert#1\vert}
\newcommand{\la}{\lambda}
\newcommand{\La}{\Lambda}
\newcommand{\bla}{{\boldsymbol\la}}
\newcommand{\bmu}{{\boldsymbol\mu}}
\newcommand{\bnu}{{\boldsymbol\nu}}

\newcommand{\qbin}[2]{\genfrac{[}{]}{0pt}{}{#1}{#2}}
\newcommand{\obinomE}[2]{\genfrac{\langle}{\rangle}{0pt}{}{#1}{#2}}

\DeclareMathOperator{\mult}{mult}

\newcommand{\bart}{\underline{t}}

\DeclareMathAlphabet{\mathpzc}{OT1}{pzc}{m}{it}

\newcommand{\cc}{\mathpzc{c}}
\newcommand{\CC}{\mathpzc{C}}
\newcommand{\dd}{\mathpzc{d}}
\newcommand{\DD}{\mathpzc{D}}

\newcommand{\leeg}{\hspace{0.8pt}\text{--}\,}

\begin{document}

\newcommand{\arXivNumber}{2007.03174}

\renewcommand{\thefootnote}{}

\renewcommand{\PaperNumber}{142}

\FirstPageHeading

\ShortArticleName{An Elliptic Hypergeometric Function Approach to Branching Rules}

\ArticleName{An Elliptic Hypergeometric Function Approach\\ to Branching Rules\footnote{This paper is a~contribution to the Special Issue on Elliptic Integrable Systems, Special Functions and Quantum Field Theory. The full collection is available at \href{https://www.emis.de/journals/SIGMA/elliptic-integrable-systems.html}{https://www.emis.de/journals/SIGMA/elliptic-integrable-systems.html}}}

\Author{Chul-hee LEE~$^\dag$, Eric M.~RAINS~$^\ddag$ and S.~Ole WARNAAR~$^\S$}

\AuthorNameForHeading{C.-h.~Lee, E.M.~Rains and S.O.~Warnaar}

\Address{$^\dag$~School of Mathematics, Korea Institute for Advanced Study, Seoul 02455, Korea}
\EmailD{\href{mailto:chlee@kias.re.kr}{chlee@kias.re.kr}}

\Address{$^\ddag$~Department of Mathematics, California Institute of Technology, Pasadena, CA 91125, USA}
\EmailD{\href{mailto:rains@caltech.edu}{rains@caltech.edu}}

\Address{$^\S$~School of Mathematics and Physics, The University of Queensland,\\
\hphantom{$^\S$}~Brisbane, QLD 4072, Australia}
\EmailD{\href{mailto:o.warnaar@maths.uq.edu.au}{o.warnaar@maths.uq.edu.au}}

\ArticleDates{Received July 08, 2020, in final form December 09, 2020; Published online December 23, 2020}

\Abstract{We prove Macdonald-type deformations of a number of well-known classical branching rules by employing identities for elliptic hypergeometric integrals and series. We also propose some conjectural branching rules and allied conjectures exhibiting a novel type of vanishing behaviour involving partitions with empty 2-cores.}

\Keywords{branching formulas; elliptic hypergeometric series; elliptic Selberg integrals; interpolation functions; Koornwinder polynomials; Littlewood identities; Macdonald polynomials}

\Classification{05E05; 05E10; 20C33; 33D05; 33D52; 33D67}

\renewcommand{\thefootnote}{\arabic{footnote}}
\setcounter{footnote}{0}

\section{Branching rules}\label{Sec_branching}

Let $\mathfrak{g}$ be a semisimple complex Lie algebra of rank $r$,
with fundamental weights $\omega_1,\dots,\omega_r$.
Denote by $V(\lambda;\mathfrak{g})$ the irreducible
$\mathfrak{g}$-module of highest weight
$\lambda=\sum_{i=1}^r \la_i \omega_i$, $\la_i\in\mathbb{Z}_{\geq 0}$.
For~$\mathfrak{h}$ a subalgebra of $\mathfrak{g}$, let
$V(\lambda;\mathfrak{g})|_{\mathfrak{h}}$ be the restriction of
$V(\lambda;\mathfrak{g})$ to $\mathfrak{h}$.
Generally, the $\mathfrak{h}$-module
$V(\lambda;\mathfrak{g})|_{\mathfrak{h}}$ is not irreducible, and the
branching problem refers to the problem of determining the multiplicities
$\mult_{\la}(\mu)$ in
\begin{gather}\label{Eq_branching-rule}
V(\lambda;\mathfrak{g})|_{\mathfrak{h}}
=\bigoplus_{\mu} \mult_{\la}(\mu) V(\mu;\mathfrak{h}),
\end{gather}
where $\mu$ runs over the weights indexing the
irreducible $\mathfrak{h}$-modules,
see, e.g., \cite{King75,KT87, Littlewood50}.
When $\mult_{\la}(\mu)\leq 1$ for all $\mu$ we say that
the branching rule~\eqref{Eq_branching-rule} is multiplicity free.

\begin{Proposition}
Let $\mathfrak{g}=\sln(2n,\mathbb{C})$,
$\mathfrak{h}=\symp(2n,\mathbb{C})$ with
canonical embedding $\mathfrak{h}\hookrightarrow\mathfrak{g}$.
For $m$,~$r$,~$p$ integers such that
$1\leq r\leq n$ and $0\leq p\leq m$,
\begin{gather}
V\big(p\omega_{r-1}+(m-p)\omega_r;\mathfrak{g}\big)\big|_{\mathfrak{h}}
=\bigoplus_{\substack{m_0,\dots,m_r\geq 0 \\ m_0+m_1+\cdots+m_r=m \\
m_{r-1}+m_{r-3}+\cdots=p}}
V\big(m_1 \omega_1+\cdots+m_r \omega_r;\mathfrak{h}).
\notag
\end{gather}
\end{Proposition}

\looseness=-1 For $p=0$ this is Proctor's branching rule \cite[Lemma~4]{Proctor83}
(see also \cite[Theorem~2.6]{Okada98}),
and for general~$p$ it is equivalent to the following combinatorial
identity of Krattenthaler \cite[equation~(3.3)]{Krattenthaler98}.
Given a partition $\la$, let $\la'$ be its conjugate and $l^{\odd}(\la)$
the number of odd parts of~$\la$.
Moreover, for $\la\subseteq (m^r)$ (i.e., $\la=(\la_1,\dots,\la_r)$
such that $0\leq\la_r\leq\cdots\leq\la_1\leq m$), let
$(m^r)-\la$ be shorthand for the complement of $\la$ with respect
to $(m^r)$ (i.e., $(m^r)-\la=(m-\la_r,\dots,m-\la_1)$.
Then
\begin{gather}\label{Eq_Kratt33}
s_{(m^{r-1},m-p)}\big(x_1^{\pm},\dots,x_n^{\pm}\big)=
\sum_{\substack{\la\subseteq (m^r) \\[1pt] l^{\odd}(\la')=p}}
\symp_{2n,(m^r)-\la}(x_1,\dots,x_n).
\end{gather}
Here $s_{\la}$ and $\symp_{2n,\la}$ are the Schur function and
symplectic Schur function indexed by the partition~$\la$
respectively \cite{Littlewood50,Macdonald95}, and
$f\big(x_1^{\pm},\dots,x_n^{\pm}\big):=f\big(x_1,x_1^{-1},\dots,x_n,x_n^{-1}\big)$.
For a more general, not necessarily multiplicity-free
$V(\lambda;\sln(2n,\mathbb{C}))|_{\symp(2n,\mathbb{C})}$-branching
rule, which is consistent with~\eqref{Eq_Kratt33} and which is
expressed in terms of Littelmann paths, we refer to~\cite{NS05,ST18}.

Krattenthaler's Schur function identity \eqref{Eq_Kratt33} follows
from a more general identity of his for the universal symplectic
characters $\symp_{\la}=\symp_{\la}(x_1,x_2,\dots)$ introduced by
Koike and Terada \cite{KT87}.
The latter specialise to the symplectic Schur functions as
\begin{gather}\label{Eq_ss}
\symp_{\la}\big(x_1^{\pm},\dots,x_n^{\pm},0,0,\dots\big)
=\begin{cases}
\symp_{2n,\la}(x_1,x_2,\dots,x_n) & \text{if $l(\la)\leq n$}, \\
0 & \text{otherwise},
\end{cases}
\end{gather}
and the lift of \eqref{Eq_Kratt33} to universal characters
is given by \cite[equation~(3.1)]{Krattenthaler98}
\begin{gather}\label{Eq_Kratt31}
s_{(m^{r-1},m-p)}=
\sum_{\substack{\la\subseteq (m^r) \\ l^{\odd}(\la')=p}}
\symp_{(m^r)-\la}.
\end{gather}

The branching problem is an important problem in representation theory,
algebraic combinatorics and orthogonal polynomials on root systems, see e.g.,
\cite{vDE18b,vDE18a,HS18,HS20,King75,KT87,Krattenthaler98,
Kwon18,LLMS13,LW11,Littlewood50,Macdonald95,NS05,Okada98,Proctor83,Rains05,ST18}.
In this paper we are interested in $q,t$-analogues of multiplicity-free
branching rules such as \eqref{Eq_Kratt33} and \eqref{Eq_Kratt31}.
We prove non-trivial new examples of what appears to be a very general
phenomenon: the natural $q,t$-analogues of multiplicity-free formulas
arising in the representation theory of the classical groups,
be it branching formulas, tensor product decompositions or
other types of multiplicity formulas, are usually `nice'.
More precisely, the $q,t$-analogue of a multiplicity $1$ appears to
almost always factor as ratio of products of binomials of the
form $1-q^k t^{\ell}$.
For some representative examples of this phenomenon, see, e.g.,
\cite{vDE11,Macdonald95,Rains05,RV07,RW15}.

The natural $q,t$-analogues of the Schur functions $s_{\la}$ are
the Macdonald polyno\-mials $P_{\la}(q,t)$ \cite{Macdonald95}.
Similarly, the natural analogues of the
symplectic Schur function $\symp_{2n,\la}$
are the $\mathrm{C}_n$ Macdonald polynomials~\cite{Macdonald00}
\begin{gather}
P_{\la}^{(\mathrm{C}_n,\mathrm{B}_n)}(q,t,s)
\qquad\text{and}\qquad
P_{\la}^{(\mathrm{C}_n,\mathrm{C}_n)}(q,t,s).
\notag
\end{gather}
It will be convenient to view these two families of
($\mathrm{BC}_n$ symmetric Laurent)
polynomials as special instances of the Koornwinder polynomials
$K_{\la}(q,t;\bart)=K_{\la}(q,t;t_0,t_1,t_2,t_3)$,~\cite{Koornwinder92}.
Specifically, see, e.g., \cite{RW15},
\begin{subequations}\label{Eq_PCB-PCC}
\begin{gather}
P_{\la}^{(\mathrm{C}_n,\mathrm{B}_n)}(q,t,s) =
K_{\la}\big(q,t;s^{1/2},-s^{1/2},q^{1/2},-q^{1/2}\big), \\
P_{\la}^{(\mathrm{C}_n,\mathrm{C}_n)}(q,t,s) =
K_{\la}\big(q,t;s^{1/2},-s^{1/2},(qs)^{1/2},-(qs)^{1/2}\big).
\end{gather}
\end{subequations}
In the following we will also need the $\mathrm{B}_n$ Macdonald
polynomial{\samepage
\begin{gather}\label{Eq_PBC}
P_{\la}^{(\mathrm{B}_n,\mathrm{C}_n)}(q,t,s)=
K_{\la}\big(q,t;-1,-q^{1/2},s,s q^{1/2}\big),
\end{gather}
where we restrict ourselves to partitions
$\la$.\footnote{The Macdonald polynomials
$P_{\la}^{(\mathrm{B}_n,\mathrm{C}_n)}(q,t,t_2)$
may also be defined for half-partitions $\la=(\la_1,\dots,\la_n)$
where $\la_1\geq\cdots\geq\la_n\geq 0$ and
$\la_i\in\frac{1}{2}+\mathbb{Z}$, in which case the right-hand side
of \eqref{Eq_PBC} needs to be slightly modified \cite{RW15}.}}

Finally, in view of \eqref{Eq_KK}, the $q,t$-analogues of the universal
symplectic characters are specialisations of the lifted Koornwinder
polynomials $\tilde{K}_{\la}(q,t,T;\bart)$, \cite{Rains05}.
These symmetric functions, which contain the extra parameter $T$,
specialise to the Koornwinder polynomials as
\begin{gather}
\tilde{K}_{\la}\big(x_1^{\pm},\dots,x_n^{\pm},0,0,\dots;q,t,t^n;\bart\big)
=\begin{cases}
K_{\la}(x_1,\dots,x_n;q,t;\bart) & \text{if $l(\la)\leq n$}, \\
0 & \text{otherwise}.
\end{cases}\label{Eq_KK}
\end{gather}
In fact, the lifted Koornwinder polynomials (which are \emph{not} polynomials)
are the unique symmetric functions such that~\eqref{Eq_KK} holds, so that the
above may serve as the definition of the~$\tilde{K}_{\la}$.

The first main result of this paper is the following triple of $q,t$-branching
rules.
Let $C^0_{\la}(z;q,t)$, $C^{-}_{\la}(z;q,t)$ and $C^{+}_{\la}(z;q,t)$ be the
standard three families of $q,t$-shifted factorials indexed by partitions (see~\eqref{Eq_qshiftE} and~\eqref{Eq_qshift} below), and, for $\la$ a partition,
let $2\la:=(2\la_1,2\la_2,\dots)$ and $\la^2:=(\la_1,\la_1,\la_2,\la_2,\dots)$.

\begin{Theorem}\label{Thm_universal}
For $m$, $r$ nonnegative integers,
\begin{subequations}\label{Eq_universal}
\begin{gather}\label{Eq_qt-sp}
P_{(m^r)}(q,t) =\sum_{\substack{\la\subset (m^r) \\ \la' \text{ even}}}
c_{\la}\big(q^{-m},t^r/T;q,t\big) \tilde{K}_{(m^r)-\la}
\big(q,t,T;t^{1/2},-t^{1/2},(qt)^{1/2},-(qt)^{1/2}\big), \\
P_{(m^r)}(q,t) =\sum_{\substack{\la\subset (m^r) \\ \la \text{ even}}}
d_{\la}\big(q^{-m},t^r/T;q,t\big)\tilde{K}_{(m^r)-\la}
\big(q,t,T;1,-1,t^{1/2},-t^{1/2}\big), \label{Eq_qt-o} \\
P_{(m^r)}\big(q^2,t^2\big) =\sum_{\la\subset (m^r)}
e_{\la}\big(q^{-m},t^r/T;q,t\big)\tilde{K}_{(m^r)-\la}
\big(q^2,t^2,T^2;-1,-q,-t,-qt\big), \label{Eq_qt-unknown}
\end{gather}
\end{subequations}
where
\begin{subequations}\label{Eq_cde-def}
\begin{gather}
c_{\la^2}(w,z;q,t):=\left(\frac{q}{t}\right)^{\abs{\la}}
\frac{C^0_{\la^2}(w;q,t)}{C^0_{\la^2}(qw/t;q,t)}
 \frac{C^{-}_{\la}\big(t;q,t^2\big)}{C_{\la}^{-}\big(q;q,t^2\big)}
 \frac{C_{\la}^{+}\big(qw^2z^2/t^4;q,t^2\big)}
{C_{\la}^{+}\big(w^2z^2/t^3;q,t^2\big)}, \label{Eq_c-def} \\
d_{2\la}(w,z;q,t) :=\left(\frac{q}{t}\right)^{\abs{\la}}
\frac{C_{2\la}^0(w;q,t)}{C_{2\la}^0(qw/t;q,t)}
\frac{C^{-}_{\la}\big(qt;q^2,t\big)}{C_{\la}^{-}\big(q^2;q^2,t\big)}
\frac{C_{\la}^{+}\big(q^2w^2z^2/t^2;q^2,t\big)}
{C_{\la}^{+}\big(qw^2z^2/t;q^2,t\big)}, \\
e_{\la}(w,z;q,t) :=\left(\frac{q}{t}\right)^{\abs{\la}}
\frac{C_{\la}^0\big(w^2;q^2,t^2\big)}{C_{\la}^0\big(q^2w^2/t^2;q^2,t^2\big)}
\frac{C^{-}_{\la}(-t;q,t)}{C_{\la}^{-}(q;q,t)}
\frac{C_{\la}^{+}\big(qw^2z^2/t^2;q,t\big)}{C_{\la}^{+}\big({-}w^2z^2/t;q,t\big)}.
\end{gather}
\end{subequations}
\end{Theorem}

To highlight the combinatorial and factorised nature of the
coefficients in \eqref{Eq_cde-def} we have rewritten
each one of them in terms of the ordinary (or type-$\mathrm{A}$)
arm-(co)lengths and leg-(co)lengths of the squares $s\in\la$ as well
as their type-$\mathrm{C}$ analogues (see Section~\ref{Sec_part}):
\begin{gather*}\allowdisplaybreaks
\begin{split}&
c_{\la^2}(w,z;q,t) =\prod_{s\in\la} \left(
\frac{q \big(1-wq^{a'(s)}t^{-2l'(s)}\big)\big(1-wq^{a'(s)}t^{-2l'(s)-1}\big)}
{t\big(1-wq^{a'(s)+1}t^{-2l'(s)-1}\big)\big(1-wq^{a'(s)+1}t^{-2l'(s)-2}\big)}\right. \\
& \left. \hphantom{c_{\la^2}(w,z;q,t) =\prod_{s\in\la}}{} \times
\frac{\big(1-q^{a(s)}t^{2l(s)+1}\big)\big(1-w^2z^2q^{\hat{a}(s)+1}t^{-2\hat{l}(s)-2}\big)}
{\big(1-q^{a(s)+1}t^{2l(s)}\big)\big(1-w^2z^2q^{\hat{a}(s)}t^{-2\hat{l}(s)-1}\big)}\right),
\end{split}\\
d_{2\la}(w,z;q,t) =
\prod_{s\in\la} \left(
\frac{q\big(1-wq^{2a'(s)}t^{-l'(s)}\big)\big(1-wq^{2a'(s)+1}t^{-l'(s)}\big)}
{t\big(1-wq^{2a'(s)+1}t^{-l'(s)-1}\big)\big(1-wq^{2a'(s)+2}t^{-l'(s)-1}\big)} \right. \\
\left. \hphantom{d_{2\la}(w,z;q,t) =\prod_{s\in\la}}{} \times
\frac{\big(1-q^{2a(s)+1}t^{l(s)+1}\big)\big(1-w^2z^2q^{2\hat{a}(s)+2}t^{-\hat{l}(s)-1}\big)}
{\big(1-q^{2a(s)+2}t^{l(s)}\big)\big(1-w^2z^2q^{2\hat{a}(s)+1}t^{-\hat{l}(s)}\big)}\right),
\end{gather*}
and
\begin{gather*}
e_{\la}(w,z;q,t)=\prod_{s\in\la}
\frac{q\big(1-w^2q^{2a'(s)}t^{-2l'(s)}\big)\big(1+q^{a(s)}t^{l(s)+1}\big)
\big(1-w^2z^2q^{\hat{a}(s)+1}t^{-\hat{l}(s)-1}\big)}
{t\big(1-w^2q^{2a'(s)+2}t^{-2l'(s)-2}\big)
\big(1-q^{a(s)+1}t^{l(s)}\big)\big(1+w^2z^2q^{\hat{a}(s)}t^{-\hat{l}(s)}\big)}.
\end{gather*}

For $t=q$, \eqref{Eq_qt-sp} and \eqref{Eq_qt-o}
simplify to the $p=0$ instances of \eqref{Eq_Kratt31} and
\cite[equation~(3.2)]{Krattenthaler98}
\begin{gather}\label{Eq_Kratt32}
s_{(m^{r-p},(m-1)^p)}=
\sum_{\substack{\la\subseteq (m^r) \\ l^{\odd}(\la)=p}}
\ortho_{(m^r)-\la},
\end{gather}
respectively.
Here $\ortho_{\la}$ is a universal orthogonal character
(see Section~\ref{Sec_SF} for details).
Finally, \eqref{Eq_qt-unknown} for $t=-q$ yields
\begin{gather}\label{Eq_s-so}
s_{(m^r)}=\sum_{\la\subset (m^r)}(-1)^{\abs{\la}} \so_{(m^r)-\la},
\end{gather}
where $\so_{\la}$ is a universal special orthogonal character.
The fact that the parameters $q$ and $T$ (the latter
only occurs on the right-hand side of \eqref{Eq_qt-sp}--\eqref{Eq_qt-unknown})
is a consequence of Lemma~\ref{Lem_Lifted} proved in Section~\ref{Sec_Lifted}.

By \eqref{Eq_PCB-PCC}--\eqref{Eq_KK},
Theorem~\ref{Thm_universal} has the following corollary.

\begin{Corollary}\label{Cor_non-universal}
Let $m$, $r$, $n$ be nonnegative integers such that $r\leq n$, and set
\begin{gather*}
x:=(x_1,\dots,x_n) \qquad\text{and}\qquad
x^{\pm}:=\big(x_1,x_1^{-1},\dots,x_n,x_n^{-1}\big).
\end{gather*}
Then
\begin{subequations}\label{Eq_non-universal}
\begin{gather}
P_{(m^r)}\big(x^{\pm};q,t\big)
=\sum_{\substack{\la\subset (m^r) \\ \la' \text{ even}}}
c_{\la}\big(q^{-m},t^{-(n-r)};q,t\big)
P^{(\mathrm{C}_n,\mathrm{C}_n)}_{(m^r)-\la}(x;q,t,t), \label{Eq_CC} \\
P_{(m^r)}\big(x^{\pm};q,t\big)
=\sum_{\substack{\la\subset (m^r) \\ \la \text{ even}}}
d_{\la}\big(q^{-m},t^{-(n-r)};q,t\big)
K_{(m^r)-\la}\big(x;q,t;1,-1,t^{1/2},-t^{1/2}\big), \label{Eq_K} \\
P_{(m^r)}\big(x^{\pm};q^2,t^2\big)
=\sum_{\la\subset (m^r)}
e_{\la}\big(q^{-m},t^{-(n-r)};q,t\big)
P^{(\mathrm{B}_n,\mathrm{C}_n)}_{(m^r)-\la}\big(x;q^2,t^2,-t\big).\label{Eq_BC}
\end{gather}
\end{subequations}
\end{Corollary}

For $t=q$, \eqref{Eq_CC} and \eqref{Eq_K} simplify to the
$p=0$ case of \eqref{Eq_Kratt33} and its dual
\cite[equation~(3.5); $p=0$]{Krattenthaler98}
\begin{gather*}
s_{(m^r)}\big(x_1^{\pm},\dots,x_n^{\pm}\big)=
\sum_{\substack{\la\subset (m^r) \\ \la \text{ even}}}
\ortho_{2n,(m^r)-\la}(x_1,\dots,x_n),
\end{gather*}
respectively.

We conjecture one more branching rule of type $(\mathrm{C}_n,\mathrm{B}_n)$.
For the notation used in this conjecture we refer to
Sections~\ref{Sec_part} and~\ref{Sec_factorials}.

\begin{Conjecture}\label{Con_selfdualcase}
For $m$, $r$ nonnegative integers,
\begin{gather*}
P_{(m^r)}(q,t)=\sum_{\substack{\la\subset (m^r) \\
\twocore{\la}=0}}
f_{\la}\big(q^{-m},q^r/T;q,t\big)\tilde{K}_{(m^r)-\la}
\big(q,t,T;\pm q^{1/2},\pm t^{1/2}\big)
\end{gather*}
and for $m$, $r$, $n$ nonnegative integers such that $r\leq n$,
\begin{gather*}
P_{(m^r)}(x^{\pm};q,t)=\sum_{\substack{\la\subset (m^r) \\
\twocore{\la}=0}}
f_{\la}\big(q^{-m},q^{-(n-r)};q,t\big)
P^{(\mathrm{C}_n,\mathrm{B}_n)}_{(m^r)-\la}(x;q,t,t),
\end{gather*}
where
\begin{gather}
f_{\la}(w,z;q,t) :=
\left(\frac{q}{t}\right)^{\abs{\la}/2}
q^{2\hat{n}^{\textup{o}}(\la')-2\hat{n}^{\textup{e}}(\la')}
t^{n^{\textup{e}}(\la)-n^{\textup{o}}(\la)} \notag \\
\hphantom{f_{\la}(w,z;q,t) :=}{} \times
\frac{C_{\la}^{0}(w;q,t)}{C_{\la}^{0}(qw/t;q,t)}
\frac{C^{-,\textup{e}}_{\la}(t;q,t)}{C_{\la}^{-,\textup{o}}(q;q,t)}
\frac{C_{\la}^{+,\textup{e}}\big(qw^2z^2/t^2;q,t\big)}
{C_{\la}^{+,\textup{o}}\big(w^2z^2/t;q,t\big)}.\label{Eq_f-def}
\end{gather}
\end{Conjecture}

The remainder of the paper is organised as follows.
The next section covers some introductory material on
partitions, various kinds of shifted factorial, and
elliptic hypergeometric series.
Then, in Section~\ref{Sec_SF}, we introduce some of the standard
bases of the ring of symmetric functions and discuss the various
types of classical Schur functions and classical branching rules.
The final introductory section is Section~\ref{Sec_MK}, in which we
survey material from Macdonald--Koornwinder theory,
including the elliptic generalisation of this theory.
In Sections~\ref{Sec_beta-integrals}--\ref{Sec_Phi} we prove a number of
new results needed for our proof of Theorem~\ref{Thm_universal}
and Corollary~\ref{Cor_non-universal}, which is
presented in Section~\ref{Sec_proof}.
These new results include the evaluations of two quadratic elliptic beta
integrals over elliptic interpolation functions
(Theorems~\ref{Thm_integral} and \ref{Thm_integral-Kawa}),
their corresponding discrete analogues
(Corollaries~\ref{Cor_quadratic-sum} and \ref{Cor_Kawanaka-sum}),
a formula for the transition coefficients between Okounkov's
$\mathrm{BC}_n$-symmetric Macdonald interpolation polynomials
and ordinary Macdonald polynomials
(Theorem~\ref{Thm_PPbarast}),
and a number of quadratic summations for a new type of elliptic
hypergeometric series (Theorems~\ref{Thm_Littlewood} and \ref{Thm_Kawanaka}).
Finally, in Section~\ref{Sec_Open} we propose a number of conjectures
in the spirit of Conjecture~\ref{Con_selfdualcase}.
As a simple example, we conjecture the new Littlewood-type identity
\begin{gather*}
\sum_{\la} \kappa^{(1)}_{\la}(q,t) P_{\la}(x;q,t)=
\prod_{i\geq 1} \frac{\big(tx_i^2;q^2\big)_{\infty}}{\big(x_i^2;q^2\big)_{\infty}}
\prod_{i<j} \frac{(tx_ix_j;q)_{\infty}}{(x_ix_j;q)_{\infty}},
\end{gather*}
where $(a;q)_{\infty}:=(1-a)(1-aq)\big(1-aq^2\big)\cdots$,
\begin{gather*}
\kappa^{(1)}_{\la}(q,t):=
\begin{cases}\displaystyle
\frac{\prod_{(i,j)\in\la}^{\textup{e}}
\big(q^{1-\la_i}t^{i-1}-q^{1-j}t^{\la'_j}\big)}
{\prod_{(i,j)\in\la}^{\textup{o}}
\big(q^{1-\la_i}t^{i-1}-q^{2-j}t^{\la'_j-1}\big)} & \text{if $\twocore{\la}=0$}, \vspace{2mm}\\
0 & \text{otherwise},
\end{cases}
\end{gather*}
and
\begin{gather*}
\sideset{}{^\textup{e/o}} \prod_{(i,j)\in\la}f_{ij} :=
\prod_{\substack{s=(i,j)\in\la \\ a(s)+l(s) \textup{ even/odd}}} f_{ij}.
\end{gather*}
To answer a question by the referee regarding the connection between the branching
rule~\eqref{Eq_CC} and two conjectural branching rules by Hoshino and Shiraishi from~\cite{HS18} we have added a postsript on
pages~\pageref{page_PS-start}--\pageref{page_PS-end}.

\section{Preliminaries}

\subsection{Partitions}\label{Sec_part}

A partition $\la=(\la_1,\la_2,\dots)$ is a sequence of weakly decreasing
nonnegative integers such that $\abs{\la}:=\la_1+\la_2+\cdots$ is finite.
If $\abs{\la}=n$ we say that $\la$ is a partition of $n$, written as
$\la\vdash n$.
The number of strictly positive $\la_i$ (the parts of $\la$) is
called the length of the partition $\la$ and denoted by $l(\la)$.
We also use $l^{\odd}(\la)$ to denote the number of odd parts of $\la$.
If $l^{\odd}(\la)=0$ we say that~$\la$ is an even partition.
The set of all partitions of length at most $n$ is denoted by $P_{+}(n)$,
and typically we write $\la=(\la_1,\dots,\la_n)$ for partitions in $P_{+}(n)$.
The multiplicity of parts of size~$i$ in the partition $\la$ is denoted by
$m_i=m_i(\la)$.
We sometimes use the multiplicities to write a~partition $\la$ as
$\big(1^{m_1}2^{m_2}\dots\big)$.
When only a single multiplicity arises, i.e., a partition is of the
form~$(m^n)$, we refer to it as a rectangle.
When $m_i(\la)\leq 1$ for all $i$ we say that $\la$ is a distinct partition.
Given a partition $\la\in P_{+}(n)$ such that $\la_1\leq m$, we write
$(m^n)-\la$ for the complement of $\la$ with respect to the rectangle
$(m^n)$.
That is, $(m^n)-\la:=(m-\la_n,\dots,m-\la_1)\in P_{+}(n)$.
The partition $\la'=(\la'_1,\la'_2,\dots)$ such that
$\la'_i=\sum_{j\geq i} m_j(\la)$ is called the conjugate of $\la$.
Perhaps more simply, if we identify a partition $\la$ with its Young diagram
(in which the parts are represented by~$l(\la)$ left-aligned rows of
boxes or squares, with~$i$th row containing $\la_i$ squares)
then the parts~$\la'$ correspond to the columns of $\la$.
Other special notation for partitions that we will employ is
$2\la:=(2\la_1,2\la_2,\dots)$ and $\la^2:=(\la_1,\la_1,\la_2,\la_2,\dots)$,
so that $2\la$ (resp.\ $\la^2$) corresponds to the partition in which
the length of each row (resp.\ column) of $\la$ has been doubled.
The partition $\mu$ is contained in $\la$, denoted as
$\mu\subset\la$, if $\mu_i\leq\la_i$ for all $i$, i.e., if the
diagram of $\mu$ fits in the diagram of $\la$.
We write $\mu\prec\la$ if $\mu\subset\la$ such that the interlacing
condition $\la_1\geq\mu_1\geq\la_2\geq\mu_2\geq\cdots$ holds.
(Alternatively, $\mu\prec\la$ if the skew shape $\la/\mu$ is a horizontal
strip, see~\cite{Macdonald95}.)
A partition~$\la$ has empty $2$-core, written as $\twocore{\la}=0$,
if its diagram can be tiled by dominoes.
For example, the partition $(5,4,4,1)$ has empty $2$-core since it admits
the tiling

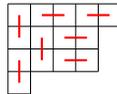
\begin{center}
\begin{tikzpicture}[scale=0.3,line width=0.3pt]
\draw (0,0)--(5,0);
\draw (0,-1)--(5,-1);
\draw (0,-2)--(4,-2);
\draw (0,-3)--(4,-3);
\draw (0,-4)--(1,-4);
\draw (0,0)--(0,-4);
\draw (1,0)--(1,-4);
\draw (2,0)--(2,-3);
\draw (3,0)--(3,-3);
\draw (4,0)--(4,-3);
\draw (5,0)--(5,-1);
\draw[thick,red] (0.5,-0.5)--(0.5,-1.5);
\draw[thick,red] (0.5,-2.5)--(0.5,-3.5);
\draw[thick,red] (1.5,-0.5)--(2.5,-0.5);
\draw[thick,red] (3.5,-0.5)--(4.5,-0.5);
\draw[thick,red] (1.5,-1.5)--(1.5,-2.5);
\draw[thick,red] (2.5,-1.5)--(3.5,-1.5);
\draw[thick,red] (2.5,-2.5)--(3.5,-2.5);
\end{tikzpicture}
\end{center}
as well as four other such tilings.

\begin{Lemma}\label{Lem_2core}
Let $\la\in P_{+}(n)$ and $m$ an even integer such that $m\geq n$.
Then $\twocore{\la}=0$ if and only if the ordered set
\begin{gather*}
\mathcal{A}_{\la}:=\{\la_1+m-1,\la_2+m-2,\dots,\la_n+m-n,m-n-1,\dots,0\}
\end{gather*}
contains $m$ even and $m$ odd integers.
\end{Lemma}

\begin{proof}
The claim is (almost) trivially true by induction on the size of $\la$.
Either we cannot remove a single domino from the border of $\la$, in which
case $\la$ does not have a trivial $2$-core (it is in fact a $2$-core itself)
and is of the form $(n,n-1,\dots,1)$ for some positive~$n$, or it is possible
to remove a~domino from the border of $\la$ to form a~partition of size
$\abs{\la}-2$.
The removal of a~domino of shape\!\!
\raisebox{-1pt}{
\begin{tikzpicture}[scale=0.23,line width=0.3pt]
\draw (0,0)--(2,0)--(2,-1)--(0,-1)--cycle;
\draw (1,0)--(1,-1);
\end{tikzpicture}}
simply decreases one of the elements of $\mathcal{A}_{\la}$ by $2$,
whereas the removal of a~domino of shape\!\!
\raisebox{-3pt}{
\begin{tikzpicture}[scale=0.23,line width=0.3pt]
\draw (0,0)--(1,0)--(1,-2)--(0,-2)--cycle;
\draw (0,-1)--(1,-1);
\end{tikzpicture}}
decreases two consecutive elements of $\mathcal{A}_{\la}$ by $1$.
The claim thus follows.
\end{proof}

Given a square $s=(i,j)\in\la$ the
arm-length, arm-colength, leg-length and leg-colength of
$s$ are defined as
\begin{alignat*}{3}
& a(s)=a_{\la}(s):=\la_i-j, \qquad&& a'(s)=a'_{\la}(s):=j-1, & \\
& l(s)=l_{\la}(s):=\la'_j-i, \qquad && l'(s)=l'_{\la}(s):=i-1.&
\end{alignat*}
Extending this to type-$\mathrm{C}$, we also set
\begin{gather*}
\hat{a}(s)=\hat{a}_{\la}(s):=\la_i+j-1,\qquad
\hat{l}(s)=\hat{l}_{\la}(s):=\la'_j+i-1.
\end{gather*}

The rationale for denoting the set of partitions of length at
most~$n$ as $P_{+}(n)$ is that we identify such partitions with
the dominant (integral) weights of $\mathrm{GL}(n,\mathbb{C})$.
Frequently we also require the superset
\begin{gather*}
P(n):=\big\{(\la_1,\dots,\la_n)\in\mathbb{Z}^n\colon
\la_1\geq\la_2\geq\cdots\geq\la_n\big\}
\end{gather*}
of all (integral) weights.
By mild abuse of notation, we sometimes write for $\mu\in P(n+1)$ with
$\mu_{n+1}=0$ that $\mu\in P_{+}(n)$, i.e., we consider $P_{+}(n)$ not just
as a subset of~$P(n)$ but also of~$P(n+1)$.

For $\la$ a partition, define the statistic
\begin{gather*}
n(\la):=\sum_{s=(i,j)\in\la} (i-1)=
\sum_{i\geq 1} (i-1)\la_i
=\sum_{s=(i,j)\in\la} (\la'_j-1)/2=
\sum_{i\geq 1} \binom{\la'_i}{2}.
\end{gather*}
This extends to skew shapes $\la/\mu$ in the obvious manner:
$n(\la/\mu):=\sum_i (i-1)(\la_i-\mu_i)=n(\la)-n(\mu)$.
By further abuse of notation (since conjugation no longer makes sense)
we will also use $n(\la)$ and $n(\la')$ for $\la\in P(n)$,
defined as~\label{P_part}
\begin{gather*}
n(\la)=\sum_{i=1}^n (i-1)\la_i \qquad\text{and}\qquad
n(\la')=\sum_{i=1}^n \binom{\la_i}{2}.
\end{gather*}

Two describe some of our conjectures we also require three types of
even and odd analogues of $n(\la)$ for~$\la$ a partition:
\begin{subequations}\label{Eq_n-eo}
\begin{gather}
n^{\textrm{e/o}}(\la) :=
\sum_{\substack{s=(i,j)\in\la \\ a(s)+l(s) \textrm{ even/odd}}}
(i-1) \\
\hat{n}^{\textrm{e/o}}(\la) :=
\sum_{\substack{s=(i,j)\in\la \\ a(s)+l(s) \textrm{ even/odd}}}
(\la'_j-1)/2
\intertext{and}
\bar{n}^{\textrm{e/o}}(\la) :=
\sum_{\substack{s=(i,j)\in\la \\ i+j \textrm{ even/odd}}}(i-1).
\end{gather}
\end{subequations}
Although
\begin{gather*}
n(\la)=n^{\textrm{e}}(\la)+n^{\textrm{o}}(\la)=
\hat{n}^{\textrm{e}}(\la)+\hat{n}^{\textrm{o}}(\la)=
\bar{n}^{\textrm{e}}(\la)+\bar{n}^{\textrm{o}}(\la),
\end{gather*}
it is generally not true that $n^{\textrm{e/o}}(\la)$,
$\hat{n}^{\textrm{e/o}}(\la)$ and
$\bar{n}^{\textrm{e/o}}(\la)$ (for fixed parity) coincide.
In fact, $\hat{n}^{\textrm{e/o}}(\la)$ can take
half-integer values.

We will in fact only use the above six functions for partitions
that have empty $2$-core.
In that case, we have the following simple relation.

\begin{Lemma}\label{Lem_two-core}
For $\la$ a partition such that $\twocore{\la}=0$,
\begin{gather*}
\bar{n}^{\textup{e/o}}(\la)=
2\hat{n}^{\textup{o/e}}(\la)-n^{\textup{e/o}}(\la).
\end{gather*}
\end{Lemma}

\begin{proof}
We again prove the claim by induction on the size of the partition $\la$.
For $\la=0$ the claim is trivially true.
Now let $\la$ be a partition of size at least two.
Because $\la$ can be tiled by dominoes, it is always possible to
remove a domino from its border to form a partition $\mu$ of size
$\abs{\la}-2$.

First assume it is possible to remove a domino of shape\!\!
\raisebox{-1pt}{
\begin{tikzpicture}[scale=0.23,line width=0.3pt]
\draw (0,0)--(2,0)--(2,-1)--(0,-1)--cycle;
\draw (1,0)--(1,-1);
\end{tikzpicture}}
such that the row-coordinate of the two boxes of the domino is $i$.
Then
\begin{gather*}
n^{\textup{e/o}}(\la)-n^{\textup{e/o}}(\mu)=
\bar{n}^{\textup{e/o}}(\la)-\bar{n}^{\textup{e/o}}(\mu)=
2\hat{n}^{\textup{e/o}}(\la)-2\hat{n}^{\textup{e/o}}(\mu)=i-1.
\end{gather*}
The above gives
\begin{gather}\label{Eq_nnn}
n^{p_1}(\la)+\bar{n}^{p_2}(\la)-2\hat{n}^{p_3}(\la)=
n^{p_1}(\mu)+\bar{n}^{p_2}(\mu)-2\hat{n}^{p_3}(\mu)
\end{gather}
irrespective of the choice of the parities $p_1$, $p_2$ and~$p_3$.
(This is consistent with the trivial fact that for $\la$ an
even partition, $n^{p_1}(\la)=\bar{n}^{p_2}(\la)=\hat{n}^{p_3}(\la)$.)

Next assume that $\la'$ is a distinct partition so that it is impossible
to remove a domino of shape\!\!
\raisebox{-1pt}{
\begin{tikzpicture}[scale=0.23,line width=0.3pt]
\draw (0,0)--(2,0)--(2,-1)--(0,-1)--cycle;
\draw (1,0)--(1,-1);
\end{tikzpicture}}\,.
We now only need to consider the first column of $\la$ from which
a domino of shape\!\!
\raisebox{-3pt}{
\begin{tikzpicture}[scale=0.23,line width=0.3pt]
\draw (0,0)--(1,0)--(1,-2)--(0,-2)--cycle;
\draw (0,-1)--(1,-1);
\end{tikzpicture}}
can be removed.
If this is the $j$th column, then $\la'$ is a distinct partition
such that $\la'_i-\la'_{i+1}=1$ for $1\leq i\leq j-1$ and
$\la'_j-\la'_{j+1}\geq 2$. (Of course, not each such a partition
necessarily has an empty $2$-core.)
From a case-by-case analysis it follows that
\begin{gather*}
n^{\textup{e}}(\la)-n^{\textup{e}}(\mu)=\la'_1-1, \qquad
n^{\textup{o}}(\la)-n^{\textup{o}}(\mu)=\la'_1-2j,
\\
\bar{n}^{\textup{e}}(\la)-\bar{n}^{\textup{e}}(\mu) =
\begin{cases}
\la'_1-j & \text{if $\la'_1$ is odd}, \\
\la'_1-j-1 & \text{if $\la'_1$ is even},
\end{cases} \\
\bar{n}^{\textup{o}}(\la)-\bar{n}^{\textup{o}}(\mu) =
\begin{cases}
\la'_1-j & \text{if $\la'_1$ is even}, \\
\la'_1-j-1 & \text{if $\la'_1$ is odd},
\end{cases}
\end{gather*}
and
\begin{gather*}
2\hat{n}^{\textup{e}}(\la)-2\hat{n}^{\textup{e}}(\mu) =
\begin{cases}
2\la'_1-3j & \text{if $\la'_1$ is even}, \\
2\la'_1-3j-1 & \text{if $\la'_1$ is odd},
\end{cases} \\
2\hat{n}^{\textup{o}}(\la)-2\hat{n}^{\textup{o}}(\mu) =
\begin{cases}
2\la'_1-j-1 & \text{if $\la'_1$ is odd}, \\
2\la'_1-j-2 & \text{if $\la'_1$ is even}.
\end{cases}
\end{gather*}
Hence \eqref{Eq_nnn} again holds, but now with $p_1=p_2\neq p_3$.
\end{proof}

\subsection{Generalised shifted factorials}\label{Sec_factorials}

In this paper we require several types of shifted factorials.
For complex $q$ such that $\abs{q}<1$ the ordinary $q$-shifted
factorial $(z;q)_{\infty}$ is defined as
\begin{gather}\label{Eq_q-shifted}
(z;q)_{\infty}:=\prod_{k\geq 1} \big(1-zq^{k-1}\big).
\end{gather}
This may be used to define $(z;q)_N$ for arbitrary
integer $N$ as
\begin{gather}\label{Eq_q-shiftedn}
(z;q)_N:=\frac{(z;q)_{\infty}}{(zq^N;q)_{\infty}}.
\end{gather}
In particular, if $N$ is nonnegative, $(z;q)_N=\prod_{k=1}^N\big(1-zq^{k-1}\big)$
and if $N$ is a negative integer $1/(q;q)_N=0$.
To generalise both definitions to the elliptic case, we need the
elliptic gamma function~\cite{Ruijsenaars97}
\begin{gather*}
\Gamma_{p,q}(z):=\prod_{i,j=0}^{\infty}\frac{1-p^{i+1}q^{j+1}/z}{1-zp^iq^j},
\end{gather*}
where $z\in\mathbb{C}^{\ast}$ and $p,q\in\mathbb{C}$
such that $\abs{p},\abs{q}<1$.
This function is symmetric in $p$ and $q$, satisfies the reflection formula
$\Gamma_{p,q}(z)\Gamma_{p,q}(pq/z)=1$ and functional equation
\begin{gather}\label{Eq_Gamma-fun}
\Gamma_{p,q}(qz)=\theta(z;p) \Gamma_{p,q}(z),
\end{gather}
where $\theta(z;p)$ is the modified theta function
\begin{gather*}
\theta(z;p):=(z;p)_{\infty}(p/z;p)_{\infty}.
\end{gather*}
Since $\lim\limits_{p\to 0} 1/\Gamma_{p,q}(z)=(z;q)_{\infty}$, the reciprocal
of the elliptic gamma function can be viewed as an elliptic analogue of~\eqref{Eq_q-shifted}.
The elliptic analogue of \eqref{Eq_q-shiftedn} is then
\begin{gather}\label{Eq_e-shifted-factorial}
(z;q,p)_N:=\frac{\Gamma_{p,q}\big(zq^N\big)}{\Gamma_{p,q}(z)},
\end{gather}
which for nonnegative $N$ can also be expressed as
\begin{gather*}
(z;q,p)_N=\prod_{k=1}^N \theta\big(zq^{k-1};p\big).
\end{gather*}
Clearly, $(z;q,0)_N=(z;q)_N$.

Three important generalisations of $(z;q,p)_N$ to the case of
partitions are given by~\cite{Rains06,W02}
\begin{subequations}\label{Eq_qshiftE}
\begin{gather}
C^0_{\la}(z;q,t;p) :=
\prod_{(i,j)\in\la}\theta\big(zq^{j-1}t^{1-i};p\big), \\
C^{-}_{\la}(z;q,t;p) :=
\prod_{(i,j)\in\la}\theta\big(zq^{\la_i-j}t^{\la'_j-i};p\big), \\
C^{+}_{\la}(z;q,t;p) :=\prod_{(i,j)\in\la}
\theta\big(zq^{\la_i+j-1}t^{2-\la'_j-i};p\big),
\end{gather}
\end{subequations}
where it is noted that
\begin{gather*}
C^0_{(N)}(z;q,t;p)=C^{-}_{(N)}(z;q,t;p)=(z;q,p)_N
\end{gather*}
and
\begin{gather*}
C^{+}_{(N)}(z;q,t;p)=\frac{(z;q,p)_{2N}}{(z;q,p)_N}.
\end{gather*}

From the simple functional equations for the theta function
\begin{gather}\label{Eq_theta-fun}
\theta(pz;p)=\theta(1/z;p)=-z^{-1} \theta(z;p),
\end{gather}
it follows that the elliptic $C$-symbols satisfy the
quasi-periodicities
\begin{subequations}\label{Eq_quasi-C}
\begin{gather}
C^0_{\la}(pz;q,t;p) =
(-z)^{-\abs{\la}} q^{-n(\la')} t^{n(\la)} C^0_{\la}(z;q,t;p),
\label{Eq_quasi-C0} \\
C^{-}_{\la}(pz;q,t;p) =
(-z)^{-\abs{\la}} q^{-n(\la')} t^{-n(\la)} C^{-}_{\la}(z;q,t;p),
\label{Eq_quasi-Cm} \\
C^{+}_{\la}(pz;q,t;p) =(-zq)^{-\abs{\la}}
q^{-3n(\la')} t^{3n(\la)} C^{+}_{\la}(z;q,t;p),
\end{gather}
\end{subequations}
as well as a long list of other simple identities, such as
\begin{subequations}\label{Eq_conj}
\begin{gather}
C^0_{\la'}(z;t,q;p) =C^0_{\la}(p/z;q,t;p), \label{Eq_C0-conj} \\
C^{-}_{\la'}(z;t,q;p) =C^{-}_{\la}(z;q,t;p), \label{Eq_Cm-conj} \\
C^{+}_{\la'}(z;t,q;p) =C^{+}_{\la}(p/zqt;q,t;p), \label{Eq_Cp-conj}
\end{gather}
\end{subequations}\vspace{-10mm}
\begin{gather}\label{Eq_recip}
C^{0,\pm}_{\la}(1/z;1/q,1/t;p)=C^{0,\pm}_{\la}(pz;q,t;p),
\end{gather}
\vspace{-10mm}
\begin{subequations}\label{Eq_double-square}
\begin{gather}
C^{0,\pm}_{2\la}(z;q,t;p)=C^{0,\pm}_{\la}\big(z,zq;q^2,t;p\big),\label{Eq_double} \\
C^0_{\la^2}(z;q,t;p)=C^0_{\la}\big(z,z/t;q,t^2;p\big), \\
C^{-}_{\la^2}(z;q,t;p)=C^{-}_{\la}\big(z,zt;q,t^2;p\big), \\
C^{+}_{\la^2}(z;q,t;p)=C^{+}_{\la}\big(z/t,z/t^2;q,t^2;p\big), \\
C^{0,\pm}_{\la}(z,-z;q,t;p)=C^{0,\pm}_{\la}\big(z^2;q^2,t^2;p^2\big),\label{Eq_pm}
\end{gather}
\end{subequations}
and
\begin{subequations}\label{Eq_add-square}
\begin{gather}
C^0_{\la+(N^n)}(z;q,t;p)=
C^0_{(N^n)}(z;q,t;p) C^0_{\la}\big(zq^N;q,t;p\big), \\
C^{-}_{\la+(N^n)}(z;q,t;p)
=C^{-}_{(N^n)}(z;q,t;p) C^{-}_{\la}(z;q,t;p)
\frac{C^0_{\la}\big(zq^Nt^{n-1};q,t;p\big)}{C^0_{\la}\big(zt^{n-1};q,t;p\big)}, \\
C^{+}_{\la+(N^n)}(z;q,t;p)
=C^{+}_{(N^n)}(z;q,t;p)C^{+}_{\la}\big(zq^{2N};q,t;p\big)
\frac{C^0_{\la}\big(zq^{2N}t^{1-n};q,t;p\big)}{C^0_{\la}\big(zq^Nt^{1-n};q,t;p\big)},
\end{gather}
\end{subequations}
where in the final set of identities it is assumed that $\la\in P_{+}(n)$
and $N$ is a nonnegative integer.
Expressing $C^{0,\pm}_{(N^n)}(z;q,t;p)$ in terms of the elliptic
shifted-factorial~\eqref{Eq_e-shifted-factorial}, we may use~\eqref{Eq_add-square} to extend the elliptic $C$-symbols to arbitrary
weights $\la\in P(n)$:
\begin{subequations}\label{Eq_C-weights}
\begin{gather}
C^0_{\la}(z;q,t;p) =C^0_{\mu}\big(zq^{\la_n};q,t;p\big)
\prod_{i=1}^n \big(zt^{1-i};q,p\big)_{\la_n}, \\
C^{-}_{\la}(z;q,t;p) =C^{-}_{\mu}(z;q,t;p)
\frac{C^0_{\mu}\big(zq^{\la_n}t^{n-1};q,t;p\big)}{C^0_{\mu}\big(zt^{n-1};q,t;p\big)}
\prod_{i=1}^n \big(zt^{n-i};q,p\big)_{\la_n}, \\
C^{+}_{\la}(z;q,t;p)
 =C^{+}_{\mu}\big(zq^{2\la_n};q,t;p\big)
\frac{C^0_{\mu}\big(zq^{2\la_n}t^{1-n};q,t;p\big)}
{C^0_{\mu}\big(zq^{\la_n}t^{1-n};q,t;p\big)}
\prod_{i=1}^n \big(zq^{\la_n} t^{2-n-i};q,p\big)_{\la_n},
\end{gather}
\end{subequations}
where $\mu:=(\la_1-\la_n,\dots,\la_{n-1}-\la_n,0)\in P_{+}(n)$.
All of the above identities, with the exception of
\eqref{Eq_conj} remain valid for non-dominant weights.

For all three elliptic $C$-symbols we also use their non-elliptic
specialisations
\begin{gather}\label{Eq_qshift}
C^{0,\pm}_{\la}(z;q,t):=C^{0,\pm}_{\la}(z;q,t;0).
\end{gather}
They satisfy the obvious analogues of \eqref{Eq_conj}--\eqref{Eq_add-square},
where it is noted that in the case of \eqref{Eq_C0-conj}, \eqref{Eq_Cp-conj}
and~\eqref{Eq_recip} one first needs to eliminate an explicit $p$ in the
argument on the right using~\eqref{Eq_quasi-C} before setting $p$ to $0$.
To avoid having to produce another five identities, we
will always refer to the above relations~-- even in the non-elliptic case~--
when manipulating $C$-symbols.
The reader should have no trouble writing down the explicit $p=0$ versions.
For instance, in the case of~\eqref{Eq_C0-conj} on finds
\begin{gather*}
C^0_{\la'}(z;t,q)=(-z)^{\abs{\la}}q^{-n(\la')}t^{n(\la)}C^0_{\la}(1/z;q,t),
\end{gather*}
and so on.

For all the shifted factorials as well as the elliptic gamma and modified
theta functions adopt the usual multiplicative and plus-minus notations,
such as
\begin{gather*}
C_{\la}^0(z_1,\dots,z_k;q,t;p):=
C_{\la}^0(z_1;q,t;p) \cdots C_{\la}^0(z_k;q,t;p)
\end{gather*}
and
\begin{gather*}
C_{\la}^0(z^{\pm};q,t;p):=C_{\la}^0\big(z,z^{-1};q,t;p\big), \\
C_{\la}^0(w^{\pm}z^{\pm};q,t;p):=
C_{\la}^0\big(wz,wz^{-1},w^{-1}z,w^{-1}z^{-1};q,t;p\big).
\end{gather*}
To further shorten some of our expressions we also introduce
the multiplicative well-poised ratio
\begin{gather*}
\Delta^0_{\la}(a\vert b_1,\dots,b_k;q,t;p):=
\frac{C^0_{\la}(b_1,\dots,b_k;q,t;p)}
{C^0_{\la}(apq/b_1,\dots,apq/b_k;q,t;p)}
\end{gather*}
and the non-multiplicative
\begin{gather*}
\Delta_{\la}(a\vert b_1,\dots,b_k;q,t;p)
:=\frac{C^0_{2\la^2}(apq;q,t;p)}
{C^{-}_{\la}(pq,t;q,t;p)C^{+}_{\la}(a,apq/t;q,t;p)}
\Delta^0_{\la}(a\vert b_1,\dots,b_k;q,t;p).
\end{gather*}

Finally, there are six more non-elliptic $C$-symbols needed to describe some
of our conjectures.
They are defined as
\begin{subequations}
\begin{gather}
C^{0,\textrm{e/o}}_{\la}(z;q,t) :=
\prod_{\substack{s=(i,j)\in\la \\ i+j \textrm{ even/odd}}}
\big(1-zq^{j-1}t^{1-i}\big), \\
C^{-,\textrm{e/o}}_{\la}(z;q,t) :=
\prod_{\substack{s=(i,j)\in\la \\ a(s)+l(s) \textrm{ even/odd}}}
\big(1-zq^{\la_i-j}t^{\la'_j-i}\big), \\
C^{+,\textrm{e/o}}_{\la}(z;q,t) :=
\prod_{\substack{s=(i,j)\in\la \\ a(s)+l(s) \textrm{ even/odd}}}
\big(1-zq^{\la_i+j-1}t^{2-\la'_j-i}\big).
\label{Eq_Cplus-eo}
\end{gather}
\end{subequations}
Clearly,
$C^{\alpha}_{\la}(z;q,t)=C^{\alpha,\textrm{e}}_{\la}(z;q,t)
C^{\alpha,\textrm{o}}_{\la}(z;q,t)$ for $\alpha\in\{0,+,-\}$.

\begin{Lemma}\label{Lem_C-conjugate-eo}
For $\la$ a partition,
\begin{gather*}
C^{0,\textup{e/o}}_{\la'}(z;t,q) =
(-z)^{\|\la\|^{\textup{e/o}}}
q^{-\bar{n}^{\textup{e/o}}(\la')}
t^{\bar{n}^{\textup{e/o}}(\la)}
C^{0,\textup{e/o}}_{\la}(1/z;q,t), \\
C^{-,\textup{e/o}}_{\la'}(z;t,q) =C^{-,\textup{e/o}}_{\la}(z;q,t), \\
C^{+,\textup{e/o}}_{\la'}(z;t,q) =(-tz)^{\abs{\la}^{\textup{e/o}}}
q^{-n^{\textup{e/o}}(\la')-2\hat{n}^{\textup{e/o}}(\la')}
t^{n^{\textup{e/o}}(\la)+2\hat{n}^{\textup{e/o}}(\la)}
C^{+,\textup{e/o}}_{\la}(1/qtz;q,t),
\end{gather*}
where $\|\la\|^{\textup{e/o}}:=\abs{\{(i,j)\in\la\colon i+j \textup{ even/odd}\}}$
and $\abs{\la}^{\textup{e/o}}:=
\abs{\{s\in\la\colon a(s)+l(s) \textup{ even/odd}\}}$.
\end{Lemma}

We will again only be using the above for $\la$ a partition with
empty $2$-core, in which case we simply have
$\|\la\|^{\textup{e/o}}=\abs{\la}^{\textup{e/o}}=\abs{\la}/2$.

\begin{proof}
We will only show the last of the three identities.
Applying definition \eqref{Eq_Cplus-eo} to its left-hand side
and interchanging $i$ and $j$ in the product leads to
\begin{gather*}
C^{+,\textup{e/o}}_{\la'}(z;t,q)
 =\prod_{\substack{s=(i,j)\in\la \\ a(s)+l(s) \textrm{ even/odd}}}
\big(1-(qtz) q^{1-\la_i-j}t^{\la'_j+i-2}\big) \\
\hphantom{C^{+,\textup{e/o}}_{\la'}(z;t,q)}{}
=C^{+,\textup{e/o}}_{\la}(1/qtz;q,t)
\prod_{\substack{s=(i,j)\in\la \\ a(s)+l(s) \textrm{ even/odd}}}
\big((-tz) q^{2-\la_i-j}t^{\la'_j+i-2}\big).
\end{gather*}
By \eqref{Eq_n-eo} the result now follows.
\end{proof}

In a similar manner it may be shown that (the $p=0$ case of)~\eqref{Eq_recip} dissects into even and odd cases as follows.

\begin{Lemma}\label{Lem_C-recip-eo}
For $\la$ a partition,
\begin{gather*}
C^{0,\textup{e/o}}_{\la}(1/z;1/q,1/t) =(-z)^{-\|\la\|^{\textup{e/o}}}
q^{-\bar{n}^{\textup{e/o}}(\la')}t^{\bar{n}^{\textup{e/o}}(\la)}
C^{0,\textup{e/o}}_{\la}(z;q,t), \\
C^{-,\textup{e/o}}_{\la}(1/z;1/q,1/t) =(-z)^{-\abs{\la}^{\textup{e/o}}}
q^{n^{\textup{e/o}}(\la')-2\hat{n}^{\textup{e/o}}(\la')}
t^{n^{\textup{e/o}}(\la)-2\hat{n}^{\textup{e/o}}(\la)}
C^{-,\textup{e/o}}_{\la}(z;q,t), \\
C^{+,\textup{e/o}}_{\la}(1/z;1/q,1/t) =(-qz)^{-\abs{\la}^{\textup{e/o}}}
q^{-n^{\textup{e/o}}(\la')-2\hat{n}^{\textup{e/o}}(\la')}
t^{n^{\textup{e/o}}(\la)+2\hat{n}^{\textup{e/o}}(\la)}
C^{+,\textup{e/o}}_{\la}(z;q,t).
\end{gather*}
\end{Lemma}

From \eqref{Eq_conj}, \eqref{Eq_recip} and
Lemmas~\ref{Lem_C-conjugate-eo}, \ref{Lem_C-recip-eo}
it follows that the rational functions defined in
\eqref{Eq_cde-def} and \eqref{Eq_f-def} satisfy the dualities
\begin{gather*}
c_{\la^2}(w,z;q,t) =c_{\la^2}(1/w,1/z;1/q,1/t)=
\left(\frac{q}{t}\right)^{\abs{\la}}
\frac{C^{-}_{\la^2}(t;q,t)}{C_{\la^2}^{-}(q;q,t)}
d_{2\la'}(1/w,1/z;t,q), \\
d_{2\la}(w,z;q,t) =d_{2\la}(1/w,1/z;1/q,1/t)=
\left(\frac{q}{t}\right)^{\abs{\la}}
\frac{C^{-}_{2\la}(t;q,t)}{C_{2\la}^{-}(q;q,t)}
c_{(\la')^2}(1/w,1/z;t,q), \\
e_{\la}(w,z;q,t) =e_{\la}(1/w,1/z;1/q,1/t)=
\left(-\frac{q}{t}\right)^{\abs{\la}}
\frac{C^{-}_{\la}\big(t^2;q^2,t^2\big)}{C_{\la}^{-}\big(q^2;q^2,t^2\big)}
e_{\la'}(1/w,1/z;t,q), \\
f_{\la}(w,z;q,t) =f_{\la}(1/w,1/z;1/q,1/t)=
\left(\frac{q}{t}\right)^{\abs{\la}/2}
\frac{C^{-}_{\la}(t;q,t)}{C_{\la}^{-}(q;q,t)}
f_{\la'}(1/w,1/z;t,q),
\end{gather*}
where in the final line it is assumed that $\la$ is a partition with
empty $2$-core.

In Sections~\ref{Sec_beta-integrals} and \ref{Sec_Phi} we also use the
$p,q$-symmetric versions of many of the generalised $q$-shifted factorials.
For $\bla=\big(\la^{(1)},\la^{(2)}\big)$ a pair of partitions or weights
and $f_{\la}(q,t;p)$ one of
\begin{gather*}
C_{\la}^{0,\pm}(\,\cdot\,;q,t;p),\qquad
\Delta^0_{\la}(\,\cdot\,\vert\,\cdot\,;q,t;p) \qquad\text{or}\qquad
\Delta_{\la}(\,\cdot\vert\,\cdot\,;q,t;p),
\end{gather*}
we set
\begin{gather*}
f_{\bla}(t;p,q)=f_{\la^{(1)}}(p,t;q)f_{\la^{(2)}}(q,t;p).
\end{gather*}
Hence
\begin{gather*}
f_{(\la^{(1)},\la^{(2)})}(t;p,q)=f_{(\la^{(2)},\la^{(1)})}(t;q,p).
\end{gather*}
By slight abuse of notation we will also write
$(n,m)\bla:=\big(n\la^{(1)},m\la^{(2)}\big)$ for positive integers $n,m$,
so that, for example,
\begin{gather*}
C^0_{(1,2)\bla}(a;t;p,q)=C^0_{\la^{(1)}}(a;p,t;q)C^0_{2\la^{(2)}}(a;q,t;p).
\notag
\end{gather*}
Finally,
$2\bla:=(2,2)\bla$ and $\bla^2:=\big(\big(\la^{(1)}\big)^2,\big(\la^{(2)}\big)^2\big)$.

\subsection{Elliptic hypergeometric series}

Our proof of Theorem~\ref{Thm_universal} relies on
(the $p\to 0$ limit of) two higher-dimensional quadratic
summation formulas for elliptic hypergeometric series.
In the one-dimensional case the simplest form an elliptic
hypergeometric series can take is \cite{GR04,Spiridonov04,W02}
\begin{gather}\label{Eq_EHS}
\sum_{k=0}^{\infty}\Delta_{(k)}(a\vert t,b_1,\dots,b_r;q,t;p)=
\sum_{k=0}^{\infty}
\frac{(apq;q,p)_{2k}}{(a;q,p)_{2k}}
\frac{(a,b_1,\dots,b_r;q,p)_k}
{(pq,c_1,\dots,c_r;q,p)_k},
\end{gather}
where, for reasons of convergence, it is assumed that one of the~$b_i$ is of the form~$q^{-N}$ with $N$ a nonnegative integer
so that the series terminates.
If the upper and lower parameters satisfy
$ab_1\cdots b_r(pq)^3=c_1\cdots c_r$ the series~\eqref{Eq_EHS}
is said to be balanced, and if $b_ic_i=apq$ for all $i$ it is said
to be very-well poised.
If both these conditions are satisfied then~\eqref{Eq_EHS} is an
elliptic function (in multiplicative form) in each of the variables
$a,b_1,\dots,b_r$, see, e.g.,~\cite{Spiridonov04}.

The most important identity for one-dimensional elliptic hypergeometric
series corresponds to~\eqref{Eq_EHS} for $r=5$, and is given by Frenkel
and Turaev's elliptic analogue of Jackson's sum \cite{FT97}:
\begin{gather}
\sum_{k=0}^N \frac{(apq;q,p)_{2k}}{(a;q,p)_{2k}}
\frac{(a,b,c,d,e,q^{-N};q,p)_k}
{(pq,apq/b,apq/c,apq/d,apq/e,apq^{N+1};q,p)_k} \nonumber\\
\qquad{} =\frac{(apq,apq/bc,apq/bd,apq/cd;q,p)_N}
{(apq/b,apq/c,apq/d,apq/bcd;q,p)_N},\label{Eq_FT}
\end{gather}
where $bcdeq^{-N}=a^2pq$.
Four other balanced, very-well poised instances of~\eqref{Eq_EHS}
for $r=7$ that admit closed-form evaluations are given by
\begin{subequations}
\begin{gather}
\sum_{k=0}^N
\frac{(apq;q,p)_{2k}}{(a;q,p)_{2k}}
\frac{\big(a,bq,abq^N,q^{-N};q,p\big)_k}
{\big(pq,ap/b,pq^{1-N}/b,apq^{N+1};q,p\big)_k}
\frac{\big(ap/b;q,p^2\big)_{2k}}{\big(abpq;q,p^2\big)_{2k}} \notag\\
\qquad{}
=\frac{\theta\big(abp;p^2\big)}{\theta\big(abpq^{2N};p^2\big)}
\frac{(apq;q,p)_N}{(b;q,p)_N}
\frac{\big(pq,b^2q;q,p^2\big)_N}{\big(abp,ap^2/b;q,p^2\big)_N}\label{Eq_unknown} \\
\sum_{k=0}^N
\frac{(apq;q,p)_{2k}}{(a;q,p)_{2k}}
\frac{\big(a,bp,abq^N,q^{-N};q,p\big)_k}
{\big(pq,aq/b,pq^{1-N}/b,apq^{N+1};q,p\big)_k}
\frac{\big(aq/b;q^2,p\big)_k}{\big(abpq;q^2,p\big)_k} \notag\\
\qquad{}=\chi(N \text{ even})
\frac{\big(q,b^2;q^2,p\big)_{N/2}}{\big(abpq,apq^2/b;q^2,p\big)_{N/2}}
\frac{(apq;q,p)_N}{(b;q,p)_N}, \label{Eq_W110}\\
\sum_{k=0}^N
\frac{(apq;q,p)_{2k}}{(a;q,p)_{2k}}
\frac{\big(b,pq/b,q^{N+1},q^{-N};q,p\big)_k}
{\big(apq/b,ab,apq^{-N},apq^{N+1};q,p\big)_k}
\frac{(a^2;q^2,p)_k}
{(pq^2;q^2,p)_k} \notag\\
\qquad{}=\frac{(apq,b/a;q,p)_N}{(q/a,abp;q,p)_N}
\frac{\big(abpq^{-N};q^2,p\big)_N}{\big(bq^{-N}/a;q^2,p\big)_N} \label{Eq_W115}
\end{gather}
and
\begin{gather}
\sum_{k=0}^N
\frac{\big(a^2p^2q^2;q^2,p^2\big)_{2k}}{\big(a^2;q^2,p^2\big)_{2k}}
\frac{\big(a^2,b^2p^2q^2,a^2b^2q^{2N},q^{-2N};q^2,p^2\big)_k}
{\big(p^2q^2,a^2/b^2,p^2q^{2-2N}/b^2,a^2p^2q^{2N+2};q^2,p^2\big)_k}
\frac{(a/b;q,p)_{2k}}{(abpq;q,p)_{2k}} \nonumber\\
\qquad{} =\frac{(ab;q,p)_{2N}}{(abpq;q,p)_{2N}}
\frac{\big({-}pq,b^2pq;q,p\big)_N}{(ab,-a/b;q,p)_N}
\frac{\big(a^2p^2q^2;q^2,p^2\big)_N}{\big(b^2;q^2,p^2\big)_N}.\label{Eq_W14}
\end{gather}
\end{subequations}
The last three identities are equations
(1.10), (1.15) and (1.4) of~\cite{W05} respectively.
The identity~\eqref{Eq_unknown} does not appear to have been stated before.
It follows by inverting, using \cite[Lemma~3.1]{W03}, the elliptic Jackson
sum~\eqref{Eq_FT} with
\begin{gather*}
(b,c,d,e,p)\mapsto \big(b^2q,aq^N/b,apq^N/b,pq^{-N},p^2\big). \notag
\end{gather*}

Because of the occurrence of base $q$ and $q^2$ (and nomes~$p$ and~$p^2$),
the above identities are commonly referred to as quadratic summation
formulas.
In Section~\ref{Sec_Phi} we obtain higher-dimensional analogues of
\eqref{Eq_unknown}, \eqref{Eq_W110} and \eqref{Eq_W14}.
Two of these play a key role in our proof of Theorem~\ref{Thm_universal}.
We remark that higher-dimensional analogues of a different type of quadratic
elliptic hypergeometric series, in which the term
$(apq;q,p)_{2k}/(a;q,p)_{2k}$ in~\eqref{Eq_EHS} is replaced by
$(apq;q,p)_{3k}/(a;q,p)_{3k}$, were recently considered in~\cite{RS20}.

\section{Schur functions and classical branching rules}\label{Sec_SF}

In this section we briefly review the definitions of the Schur functions
of classical type as well as their occurrence in some of the branching rules
stated in Section~\ref{Sec_branching}.
For a more in-depth treatment we refer the reader to
\cite{KT87,Krattenthaler98,Littlewood50,Macdonald95,Okada98,RW15}.

Let $\Lambda_n:=\mathbb{Z}[x_1,\dots,x_n]^{S_n}$
denote the ring of symmetric functions in $n$ variables, and
$\Lambda$ the ring of symmetric functions in countably many
variables, see \cite{Macdonald95,Stanley99}.
The monomial symmetric functions $\{m_{\la}\}_{\la\in P_{+}(n)}$
and $\{m_{\la}\}_{\la}$, where
\begin{gather*}
m_{\la}(x_1,\dots,x_n)=\sum_{w\in S_n/S_n^{\la}}w\big(x^{\la}\big), \qquad \la\in P_{+}(n),
\end{gather*}
form $\mathbb{Z}$-bases of $\La_n$ and $\La$ respectively.
The elementary, complete and power-sum symmetric functions~$e_r$,~$h_r$
and~$p_r$ are defined in terms of the monomial symmetric functions
as
\begin{gather*}
e_r:=m_{(1^r)}=\sum_{1<i_1<i_2<\dots<i_r} x_{i_1} x_{i_2}\cdots x_{i_r}, \\
h_r:=\sum_{\la\vdash r} m_{\la}=\sum_{1\leq i_1\leq i_2\leq\cdots\leq i_r}
x_{i_1} x_{i_2}\cdots x_{i_r}, \\
p_r:=m_{(r)}=\sum_{i\geq 1} x_i^r.
\end{gather*}
These functions form algebraic bases of either $\Lambda$ (in the case of the
$e_r$ and $h_r$) or of
$\Lambda_{\mathbb{Q}}:=\Lambda\otimes_{\mathbb{Z}}\mathbb{Q}$ (in the case
of the power sums).

A number of classical branching rules for universal characters
discussed below are related by the involution $\omega$ on $\La$
defined by $\omega(h_r)=e_r$ or $\omega(p_r)=(-1)^{r-1} p_r$ for all $r\geq 1$.

The ordinary (or $\mathrm{GL}(n)$) Schur function indexed by the partition
$\la$ is defined as
\begin{gather}\label{Eq_Schur}
s_{\la}(x_1,\dots,x_n):=\frac{\det_{1\leq i,j\leq n}\big(x_i^{\la_j+n-j}\big)}
{\prod_{1\leq i<j\leq n} (x_i-x_j)}
\end{gather}
if $l(\la)\leq n$ and $0$ otherwise.
To simultaneously extend this to $\Lambda$ as well as skew shapes,
we use the Jacobi--Trudi identity or its dual
\cite[pp.~70--71]{Macdonald95}:
\begin{gather}\label{Eq_JT}
s_{\la/\mu}:=\det_{1\leq i,j\leq n}( h_{\la_i-\mu_j-i+j})
=\det_{1\leq i,j\leq m}(e_{\la'_i-\mu'_j-i+j}),
\end{gather}
where $n$ and $m$ are arbitrary integers such that
$n\geq l(\la)$ and $m\geq \la_1$.
Obviously, $\omega(s_{\la/\mu})=s_{\la'/\mu'}$.
The Littlewood--Richardson coefficients $c_{\mu\nu}^{\la}$
may now be defined by
\begin{gather}
s_{\la/\mu} = \sum_{\nu} c_{\mu\nu}^{\la} s_{\nu}. \notag
\end{gather}
From \eqref{Eq_JT} it follows that
$s_{\la/\mu}=s_{((m^n)-\mu)/((m^n)-\la)}$
for $\la\subseteq (m^n)$ so that the Littlewood--Richardson
coefficients satisfy the complementation symmetry
\begin{gather}\label{Eq_LR-complementation}
c_{\mu\nu}^{\la}=c_{(m^n)-\la,\nu}^{(m^n)-\mu}
\qquad \text{for $\la\subset (m^n)$}.
\end{gather}

For $\la$ a partition, the universal orthogonal and
symplectic characters indexed by $\la$ are given by
\cite[Definition~2.1.1]{KT87}
\begin{gather}
\ortho_{\la}:=\det_{1\leq i,j\leq n}
(h_{\la_i-i+j}-h_{\la_i-i-j}) \label{Eq_universal-ortho} \\
\symp_{\la}:=\det_{1\leq i,j\leq m}
(e_{\la'_i-i+j}-e_{\la'_i-i-j}), \label{Eq_universal-symp}
\end{gather}
with $n$ and $m$ as above.
Hence $\omega(\ortho_{\la})=\symp_{\la'}$.
In particular,
\begin{gather}\label{Eq_univeral-g}
g_{\la}\big(x_1^{\pm},\dots,x_n^{\pm},0,\dots\big)
=\begin{cases}
g_{2n,\la}(x_1,x_2,\dots,x_n) & \text{if $\la\in P_{+}(n)$}, \\
0 & \text{otherwise},
\end{cases}
\end{gather}
where $g=\symp$ (resp.\ $g=\ortho$) corresponds to an
orthogonal or symplectic character indexed by $\la$.
We add to the above the universal special orthogonal character
indexed by $\la$ as
\begin{gather}\label{Eq_universal-special-ortho}
\so_{\la}:=\det_{1\leq i,j\leq n} (h_{\la_i-i+j}+h_{\la_i-i-j+1})
=\det_{1\leq i,j\leq m}
(e_{\la'_i-i+j}+e_{\la'_i-i-j+1}),
\end{gather}
so that $\omega(\so_{\la})=\so_{\la'}$.
The character $\so_{\la}$ is the unique symmetric function such that
\begin{gather}\label{Eq_universal-so}
\so_{\la}\big(x_1^{\pm},\dots,x_n^{\pm},0,\dots\big)
=\begin{cases}
\so_{2n+1,\la}(x_1,x_2,\dots,x_n) & \text{if $\la\in P_{+}(n)$}, \\
0 & \text{otherwise},
\end{cases}
\end{gather}
where $\so_{2n+1,\la}$ is the odd-orthogonal Schur function
indexed by $\la$.
The character $\so_{\la}$ may readily be related to the
universal symplectic and orthogonal characters as
\begin{subequations}\label{Eq_sosymportho-relation}
\begin{gather}
\so_{\la}=\sum_{\mu'\prec\la'} \symp_{\mu}=
\sum_{\mu\prec\la} \ortho_{\mu}\label{Eq_so-symp-ortho}
\end{gather}
and
\begin{gather}\label{Eq_symp-so-ortho-so}
\symp_{\la}=\sum_{\mu\prec\la} (-1)^{\abs{\la/\mu}} \so_{\mu},
\qquad
\ortho_{\la}=\sum_{\mu'\prec\la'} (-1)^{\abs{\la/\mu}} \so_{\mu}.
\end{gather}
\end{subequations}
For the actual symplectic, orthogonal and odd-orthogonal Schur
functions we have \cite{Littlewood50}
\begin{gather*}
\symp_{2n,\la}(x_1,\dots,x_n)=
\frac{\det_{1\leq i,j\leq n}\big(
x_i^{\la_j+2n-j+1}-x_i^{-\la_j+j-1}\big)}
{\prod_{i=1}^n(x_i^2-1)
\prod_{1\leq i<j\leq n}(x_i-x_j)(x_ix_j-1)}, \\
\ortho_{2n,\la}(x_1,\dots,x_n)=
f_{l(\la),n} \frac{\det_{1\leq i,j\leq n}\big(
x_i^{\la_j+2n-j-1}+x_i^{-\la_j+j-1}\big)}
{\prod_{1\leq i<j\leq n}(x_i-x_j)(x_ix_j-1)}, \notag
\end{gather*}
where $f_{n,n}=1$ and $f_{i,n}=1/2$ if $i<n$, and
\begin{gather}\label{Eq_odd-orthogonal}
\so_{2n+1,\la}(x_1,\dots,x_n)=
\frac{\det_{1\leq i,j\leq n}
\big(x_i^{\la_j+2n-j}-x_i^{-\la_j+j-1}\big)}
{\prod_{i=1}^n (x_i-1)
\prod_{1\leq i<j\leq n} (x_i-x_j)(x_ix_j-1)}.
\end{gather}

Littlewood \cite{Littlewood50} and Koike and Terada~\cite{KT87} proved some very general branching
formulas for the classical groups.
For example, in the universal case \cite[Theorem~2.3.1]{KT87},
\begin{gather}\label{Eq_Litt-KT}
s_{\la}=\sum_{\mu} \bigg(\sum_{\nu \text{ even}} c_{\mu\nu}^{\la}\bigg)
\ortho_{\mu}
=\sum_{\mu}
\bigg(\sum_{\nu' \text{ even}} c_{\mu\nu}^{\la}\bigg) \symp_{\mu}.
\end{gather}
By \eqref{Eq_LR-complementation} it is not hard to show that
\eqref{Eq_Kratt31} and \eqref{Eq_Kratt32} follow from
\eqref{Eq_Litt-KT}.
Indeed,
\begin{gather*}
s_{(m^{r-1},m-p)}= \sum_{\mu}
\sum_{\nu \text{ even}} c_{\mu\nu}^{(m^{r-1},m-p)}
\ortho_{\mu} = \sum_{\mu\subset (m^r)}
\sum_{\nu \text{ even}} c_{(1^p),\nu}^{\mu} \ortho_{(m^r)-\mu}.
\end{gather*}
By the $e$-Pieri rule \cite[p.~73]{Macdonald95},
\begin{gather*}
c_{(1^p),\nu}^{\mu}=
\begin{cases}
1 & \text{if $\nu'\prec\mu'$ and $\abs{\mu/\nu}=p$}, \\
0 & \text{otherwise}.
\end{cases}
\end{gather*}
Hence
\begin{gather*}
s_{(m^{r-1},m-p)}=
\sum_{\mu\subset (m^r)}
\sum_{\substack{\nu'\prec\mu' \\ \nu \text{ even} \\
\abs{\nu}=\abs{\mu}-p}} \ortho_{(m^r)-\mu}.
\end{gather*}
Since{\samepage
\begin{gather*}
\sum_{\substack{\nu'\prec\mu' \\ \nu \text{ even} \\
\abs{\nu}=\abs{\mu}-p}} 1 =
\begin{cases}
1 & \text{if $l^{\odd}(\mu)=p$}, \\
0 & \text{otherwise},
\end{cases}
\end{gather*}
the branching rule \eqref{Eq_Kratt31} follows.}

We also remark that the three universal branching rules
\eqref{Eq_Kratt31}, \eqref{Eq_Kratt32} and \eqref{Eq_s-so}
are not independent.
Obviously, \eqref{Eq_Kratt32} follows from
\eqref{Eq_Kratt31} by application of $\omega$ and vice versa.
Also, the rectangular (i.e., $p=0$) cases of each of the
branching rules are related via \eqref{Eq_sosymportho-relation}.
For example,
from \eqref{Eq_s-so} and \eqref{Eq_so-symp-ortho}, and the fact
that $(m^r)-\la\prec (m^r)-\mu$ (for $\la,\mu\subset (m^r)$)
implies $\mu\succ\la$, it follows that
\begin{align*}
s_{(m^r)}&=\sum_{\la\subset(m^r)}(-1)^{\abs{\la}}\so_{(m^r)-\la}
=\sum_{\la\subset(m^r)}(-1)^{\abs{\la}}
\sum_{\mu\succ\la}\ortho_{(m^r)-\mu} \\
&=\sum_{\mu\subset (m^r)}
\sum_{\la\prec\mu}(-1)^{\abs{\la}}\ortho_{(m^r)-\mu}
=\sum_{\substack{\mu\subset (m^r) \\[1pt] \mu \text{ even}}}
\ortho_{(m^r)-\mu},
\end{align*}
where in the last step we have used
\begin{gather*}
\sum_{\la\prec\mu}(-1)^{\abs{\la}}=
\begin{cases}
1 & \text{if $\mu$ is even}, \\
0 & \text{otherwise}.
\end{cases}
\end{gather*}
Conversely, from \eqref{Eq_Kratt32} and \eqref{Eq_symp-so-ortho-so},
\begin{align*}
s_{(m^r)}&=\sum_{\substack{\la\subset(m^r) \\[1pt] \la \text{ even}}}
\ortho_{(m^r)-\la}
=\sum_{\substack{\la\subset(m^r) \\[1pt] \la \text{ even}}}
\sum_{\mu'\succ\la'}(-1)^{\abs{\mu/\la}} \so_{(m^r)-\mu} \\
&=\sum_{\mu\subset (m^r)}(-1)^{\abs{\mu}}
\sum_{\substack{\la'\prec\mu' \\[1pt] \la \text{ even}}}
\so_{(m^r)-\mu}
=\sum_{\mu\subset (m^r)}(-1)^{\abs{\mu}} \so_{(m^r)-\mu},
\end{align*}
since there is a unique even partition $\la$ such that
$\la'\prec\mu'$.

In the $q,t$-case, we neither have analogues of
\eqref{Eq_sosymportho-relation} nor of \eqref{Eq_Litt-KT},
making the proof of Theorem~\ref{Thm_universal} much harder
than in the classical case.
As we shall see, however, \eqref{Eq_qt-sp} and \eqref{Eq_qt-o}
are related by the $q,t$-analogue of the involution $\omega$.

\section{Macdonald--Koornwinder theory}\label{Sec_MK}

In this section we survey some necessary background material
from the theory of Macdonald and Koornwinder polynomials,
covering Macdonald polynomials, $\mathrm{BC}_n$-symmetric
(Macdonald) interpolation polynomials, (lifted) Koornwinder polynomials
and elliptic interpolation functions.

\subsection{Macdonald polynomials}

Let $0<\abs{q},\abs{t}<1$.
For $f$ an $S_n$-symmetric function (not necessarily an $S_n$-symmetric
Laurent polynomial) we define
\begin{gather}\label{Eq_f-Mac-density}
\big\langle f\big\rangle_{q,t}^{(n)}
:=\frac{1}{S_n(q,t)}
\frac{1}{n!(2\pi\iup)^n}
\Int_{\mathbb{T}^n} f(z) \prod_{1\leq i<j\leq n}
\frac{(z_i/z_j,z_j/z_i;q)_{\infty}}{(tz_i/z_j,tz_j/z_i;q)_{\infty}}
\frac{\dup z_1}{z_1}\cdots\frac{\dup z_n}{z_n},
\end{gather}
where $\mathbb{T}$ is the positively-oriented unit circle
and \cite{Macdonald95}
\begin{gather}\label{Eq_Snqt}
S_n(q,t) := \frac{1}{n!(2\pi\iup)^n}
\Int_{\mathbb{T}^n} \prod_{1\leq i<j\leq n}
\frac{(z_i/z_j,z_j/z_i;q)_{\infty}}{(tz_i/z_j,tz_j/z_i;q)_{\infty}}
\frac{\dup z_1}{z_1}\cdots\frac{\dup z_n}{z_n}
=\prod_{i=1}^n \frac{\big(t,qt^{i-1};q\big)_{\infty}}{\big(q,t^i;q\big)_{\infty}}.
\end{gather}

Recall the dominance (partial) order on partitions: for
$\la,\mu\vdash m$, $\la\geq\mu$ if
$\la_1+\dots+\la_i\geq \mu_1+\dots+\mu_i$ for all $i\geq 1$,
and $\la>\mu$ if $\la\geq\mu$ and $\la\neq\mu$.
Also let $x^{-1}:=\big(x_1^{-1},\dots,x_n^{-1}\big)$.
Then Macdonald polynomials $P_{\la}=P_{\la}(q,t)=P_{\la}(x;q,t)$ for
$\la\in P_{+}(n)$ are the unique homogeneous symmetric functions
in
$\Lambda_{\mathbb{Q}(q,t)}:=\Lambda\otimes_{\mathbb{Z}}\mathbb{Q}(q,t)=
\mathbb{Q}(q,t)[x_1,\dots,x_n]^{S_n}$ of the form
\begin{gather*}
P_{\la}=m_{\la}+\sum_{\mu<\la} c_{\la\mu}(q,t) m_{\mu}
\end{gather*}
such that (for $0<\abs{q},\abs{t}<1$)
\begin{gather}\label{Eq_P-ortho}
\big\langle P_{\la}(x;q,t) P_{\mu}\big(x^{-1};q,t\big) \big\rangle_{q,t}^{(n)}=0
\end{gather}
if $\la\neq\mu$, see \cite[pp.~368--376]{Macdonald95}.
For $\la=\mu$,
\begin{align}
\big\langle P_{\la}(x;q,t) P_{\la}\big(x^{-1};q,t\big) \big\rangle_{q,t}^{(n)}&=
\prod_{1\leq i<j\leq n}\frac{\big(t^{j-i+1},qt^{j-i-1};q\big)_{\la_i-\la_j}}
{\big(t^{j-i},qt^{j-i};q\big)_{\la_i-\la_j}} \notag \\
&=\frac{C^0_{\la}\big(t^n;q,t\big) C^{-}_{\la}(q;q,t)}
{C^0_{\la}\big(qt^{n-1};q,t\big) C^{-}_{\la}(t;q,t)}.\label{Eq_P-norm2}
\end{align}

The Macdonald polynomials satisfy the symmetry
\begin{gather}\label{Eq_Mac-recip}
P_{\la}(x;q,t)=P_{\la}(x;1/q,1/t)
\end{gather}
and (dual) Cauchy identity
\begin{gather}\label{Eq_Cauchy-Mac}
\sum_{\la\subset (m^n)} P_{\la}(x_1,\dots,x_n;q,t)P_{\la'}(y_1,\dots,y_m;t,q)
=\prod_{i=1}^n \prod_{j=1}^m (1+x_iy_j).
\end{gather}
Since
\begin{gather*}
P_{(\la_1+1,\dots,\la_n+1)}(x;q,t)=(x_1\cdots x_n)
P_{(\la_1,\dots,\la_n)}(x;q,t),
\end{gather*}
they can be extended from $\la\in P_{+}(n)$ to arbitrary weights
$\la\in P(n)$ via
\begin{gather}\label{Eq_gln-weights}
P_{\la}(x;q,t)=(x_1\cdots x_n)^{\la_n} P_{\mu}(x;q,t),
\end{gather}
where $\mu:=(\la_1-\la_n,\dots,\la_{n-1}-\la_n,0)\in P_{+}(n)$.
Then $\{P_{\la}(q,t)\}_{\la\in P(n)}$
forms a $\mathbb{Q}(q,t)$-basis of the ring of $S_n$-symmetric
Laurent polynomials, which in the following we will denote by
$\Lambda_{\mathrm{GL}(n)}$.
Since,
\begin{gather*}
P_{\la}(x;q,t) P_{\mu}\big(x^{-1};q,t\big)=
P_{(\la_1-m,\dots,\la_n-m)}(x;q,t)
P_{(\mu_1-m,\dots,\mu_n-m)}\big(x^{-1};q,t\big),
\end{gather*}
where $m=\min\{\la_n,\mu_n\}$,
the orthogonality \eqref{Eq_P-ortho} and evaluation \eqref{Eq_P-norm2}
extend to all $\la,\mu\in P(n)$.
In the case of \eqref{Eq_P-norm2} for $\la$ a non-dominant weight,
the first expression on the right should be used as both
$C^{-}_{\la}(zq;q,t)$ and $1/C^0_{\la}\big(zqt^{n-1};q,t\big)$ have a~pole (of order one) at $z=1$.

For later use we note that by \eqref{Eq_gln-weights} and the
complementation symmetry for Macdonald polynomials (see, e.g.,~\cite{BF99})
\begin{gather}\label{Eq_P-comp}
P_{(\la_1,\dots,\la_n)}\big(x^{-1};q,t\big)=
(x_1\cdots x_n)^{-\la_1}
P_{(\la_1-\la_n,\dots,\la_1-\la_2,0)}(x;q,t)
\end{gather}
we have the further symmetry
\begin{gather}\label{Eq_P-symmetry}
P_{(\la_1,\dots,\la_n)}\big(x^{-1};q,t\big)=P_{(-\la_n,\dots,-\la_1)}(x;q,t).
\end{gather}

If the involution $\omega$ on $\Lambda$ is extended to the following
automorphism of $\Lambda_{\mathbb{Q}(q,t)}$:
\begin{gather}\label{Eq_omegaqt}
\omega_{q,t}(p_r)=(-1)^{r-1} \frac{1-q^r}{1-t^r} p_r
\end{gather}
for all $r\geq 1$, then
\begin{gather}\label{Eq_omegaPQ}
\omega_{q,t}\big(P_{\la}(q,t)\big)=Q_{\la'}(t,q),
\end{gather}
where
\begin{gather*}
Q_{\la}(q,t):=\frac{C^{-}_{\la}(t;q,t)}{C^{-}_{\la}(q;q,t)} P_{\la}(q,t).
\notag
\end{gather*}

\subsection[BCn-symmetric interpolation polynomials]{$\boldsymbol{\mathrm{BC}_n}$-symmetric interpolation polynomials}

Let $R$ be a coefficient ring or field, such as
$\mathbb{Q}(q,t)\big[s^{\pm}\big]$ or $\mathbb{Q}(q,t,t_0,t_1,t_2,t_3)$.
A polynomial $f\in R\big[x_1^{\pm},\dots,x_n^{\pm}\big]$
is said to be $\mathrm{BC}_n$-symmetric if it is symmetric under
the canonical action of the hyperoctahedral group
$W:=S_n\ltimes (\mathbb{Z}/2\mathbb{Z})^n$ on
$R\big[x_1^{\pm},\dots,x_n^{\pm}\big]$.
Let $x:=(x_1,\dots,x_n)$.
The monomial basis in the ring of $\mathrm{BC}_n$-symmetric
polynomials, $\Lambda_{\textrm{BC}(n)}=R\big[x_1^{\pm},\dots,x_n^{\pm}\big]^W$,
is given by $\big\{ m^W_{\la}\big\}_{\la\in P_{+}(n)}$, where
\begin{gather*}
m^W_{\la}=m^W_{\la}(x):=\sum_{w\in W/W^{\la}} w\big(x^{\lambda}\big).
\end{gather*}

Any non-constant $\mathrm{BC}_n$-symmetric polynomial is necessarily
inhomogeneous.
It will thus be convenient to extend the dominance order from partitions
of the same size to all partitions in the obvious way:
$\la\geq\mu$, if $\la_1+\cdots+\la_i\geq\mu_1+\cdots+\mu_i$ for all
$i\geq 1$.

Let $R=\mathbb{Q}(q,t)\big[s^{\pm}\big]$, $\la,\mu\in P_{+}(n)$ and
\begin{gather*}
\spec{\la}_{n;q,t}:=\big(q^{\la_1}t^{n-1},\dots,q^{\la_{n-1}}t,q^{\la_n}\big)
\end{gather*}
a spectral vector.
Then the $\mathrm{BC}_n$-symmetric (Macdonald) interpolation
polynomial
\begin{gather*}
\bar{P}^{\ast}_{\mu}(q,t,s)=\bar{P}^{\ast}_{\mu}(x;q,t,s)
\end{gather*}
is the unique polynomial in $\Lambda_{\textrm{BC}(n)}$ of the form
\begin{gather*}
\bar{P}^{\ast}_{\mu}(q,t,s)=m^W_{\mu}+
\sum_{\substack{\la\in P_{+}(n) \\ \la<\mu}}
c_{\mu\la}(q,t,s)m^W_{\la} \notag
\end{gather*}
satisfying the vanishing conditions
\begin{gather}\label{Eq_vanishing}
P_{\mu}\big(s\spec{\la}_{n;q,t};q,t,s\big)=0
\qquad\text{if $\mu\not\subset\la$},
\end{gather}
see \cite{Okounkov98,Rains05}.
Since any triangular $\mathrm{BC}_n$-symmetric polynomial
with leading term~$m^W_{\mu}$ is uniquely determined by its values at
$z\spec{\la}_{n;q,t}$ for $\la<\mu$ and some arbitrary nonzero $z$,
the above vanishing conditions in fact lead to an overdetermined
linear system for the coefficients $c_{\mu\la}$.
One of the main results of~\cite{Okounkov98} is the actual existence
of the interpolation polynomials.\footnote{Alternatively one may
replace~\eqref{Eq_vanishing} by vanishing for $\la<\mu$ so that
uniqueness and existence are immediate and then prove the
extra vanishing conditions.}

The interpolation polynomial $\bar{P}^{\ast}_{\mu}(q,t,s)$,
whose top-degree term coincides with the Macdonald polynomial
$P_{\mu}(q,t)$,
satisfies the symmetries
\begin{gather}\label{Eq_interpolation-sym}
\bar{P}^{\ast}_{\mu}(x;q,t,s)
=\bar{P}^{\ast}_{\mu}(x;1/q,1/t,1/s)
=(-1)^{\abs{\mu}} \bar{P}^{\ast}_{\mu}(-x;q,t,-s)
\end{gather}
and
\begin{gather}\label{Eq_P-addsquare}
\bar{P}^{\ast}_{\mu+(N^n)}(x;q,t,s)
=(-s)^{-nN} q^{-n\binom{N}{2}}
\bar{P}^{\ast}_{\mu}\big(x;q,t,sq^N\big)
\prod_{i=1}^n \big(sx_i^{\pm};q\big)_N
\end{gather}
for $N$ an arbitrary integer such that $N\geq-\mu_n$.
Like the Macdonald polynomials, this can be used to extend
the~$\mathrm{BC}_n$ interpolation polynomials to arbitrary
weights $\mu\in P(n)$:
\begin{gather}\label{Eq_Pbar-weights}
\bar{P}^{\ast}_{\mu}(x;q,t,s)
=(-s)^{-n\mu_n} q^{-n\binom{\mu_n}{2}}
\bar{P}^{\ast}_{\nu}\big(x;q,t,sq^{\mu_n}\big)
\prod_{i=1}^n \big(sx_i^{\pm};q\big)_{\mu_n},
\end{gather}
where $\nu:=(\mu_1-\mu_n,\dots,\mu_{n-1}-\mu_n,0)\in P_{+}(n)$.
Of course, for $\mu$ not dominant, i.e., for $\mu\not\in P_{+}(n)$,
$\bar{P}^{\ast}_{\mu}(x;q,t,s)$ is not a Laurent polynomial
but a rational function in $x$.

The interpolation polynomials also admit a closed-form evaluation at
$x=s\spec{\mu}_{n;q,t}$
\begin{gather}
P_{\mu}\big(s\spec{\mu}_{n;q,t};q,t,s\big) \nonumber\\
\qquad{} =\big(sqt^{n-1}\big)^{-\abs{\mu}}q^{-2n(\mu')}t^{n(\mu)}
C^{-}_{\mu}(q;q,t)C^{+}_{\mu}\big(s^2t^{2n-2};q,t\big)
\qquad\text{for $\mu\in P_{+}(n)$},\label{Eq_Pbar-normalise}
\end{gather}
as well as a principal specialisation formula
\begin{gather}
\bar{P}^{\ast}_{\mu}\big(z\spec{0}_{n;q,t};q,t,s\big)\nonumber\\
\qquad{} =\big({-}st^{n-1}\big)^{-\abs{\mu}}q^{-n(\mu')}t^{2n(\mu)}
\frac{C^0_{\mu}\big(t^n,s/z,zst^{n-1};q,t\big)}{C^{-}_{\mu}(t;q,t)}
\qquad\text{for $\mu\in P(n)$}.\label{Eq_PS-Pbar}
\end{gather}
They may also be used to define generalised $q$-binomial
coefficients \cite[p.~81]{Rains05}
\begin{gather}\label{Eq_qbin-qts}
\qbin{\la}{\mu}_{q,t,s}:=
\frac{\bar{P}^{\ast}_{\mu}\big(st^{1-n} \spec{\la}_{n;q,t};q,t,st^{1-n}\big)}
{\bar{P}^{\ast}_{\mu}\big(st^{1-n} \spec{\mu}_{n;q,t};q,t,st^{1-n}\big)},
\end{gather}
where $\la$, $\mu$ are partitions and $n$ is an arbitrary integer
such that $n\geq l(\la),l(\mu)$.
The independence of $n$ readily follows from the fact that
for any $\mu\in P(n+1)$,
\begin{gather}\label{Eq_xniss}
\bar{P}^{\ast}_{\mu}(x_1,\dots,x_n,s;q,t,s)=
\begin{cases}
\bar{P}^{\ast}_{\mu}(x_1,\dots,x_n;q,t,st)
& \text{if $\mu\in P_{+}(n)$}, \\
0 & \text{otherwise}.
\end{cases}
\end{gather}
Clearly, $\qbin{\la}{\mu}_{q,t,s}=0$ unless $\mu\subset\la$
and $\qbin{\la}{0}_{q,t,s}=\qbin{\la}{\la}_{q,t,s}=1$.
By~\eqref{Eq_interpolation-sym} it also follows that
\begin{gather}\label{Eq_qbin-sym}
\qbin{\la}{\mu}_{q,t,s}=
\qbin{\la}{\mu}_{1/q,1/t,1/s}=\qbin{\la}{\mu}_{q,t,-s},
\end{gather}
so that, in particular, the generalised binomial coefficients are
a function of~$s^2$ only.
Less obvious is that the $q,t,s$-binomial coefficients
satisfy conjugation symmetry \cite[Corollary~4.3]{Rains05}
\begin{gather}\label{Eq_binom-conj}
\qbin{\la'}{\mu'}_{t,q,s}=\qbin{\la}{\mu}_{q,t,s^{-1}(qt)^{-1/2}}.
\end{gather}
As an immediate consequence of \eqref{Eq_P-addsquare},
\begin{gather}
\qbin{\la+(N^n)}{\mu+(N^n)}_{q,t,s} \nonumber\\
\qquad{} =
q^{-\abs{\la/\mu}N}
\frac{C^0_{\la}\big(s^2q^{2N}t^{1-n},q^{N+1}t^{n-1};q,t\big)
C^0_{\mu}\big(s^2q^Nt^{1-n},qt^{n-1};q,t\big)}
{C^0_{\mu}\big(s^2q^{2N}t^{1-n},q^{N+1}t^{n-1};q,t\big)
C^0_{\la}\big(s^2q^Nt^{1-n},qt^{n-1};q,t\big)}
\qbin{\la}{\mu}_{q,t,sq^N}.\label{Eq_binom-add-square}
\end{gather}
Similarly, from \eqref{Eq_Pbar-normalise} and \eqref{Eq_PS-Pbar},
\begin{gather*}
\qbin{(N^n)}{\mu}_{q,t,s}=(-q)^{\abs{\mu}}q^{n(\mu')}t^{n(\mu)}
\frac{C^0_{\mu}\big(t^n,s^2q^Nt^{1-n},q^{-N};q,t\big)}
{C^{-}_{\mu}(q,t;q,t)C^{+}_{\mu}\big(s^2;q,t\big)}.
\end{gather*}

For $t=q$ the $\mathrm{BC}_n$ Macdonald interpolation polynomials
simplify to the corresponding Schur functions, see, e.g.,~\cite{OO98}.
These are expressible as a simple Weyl-type determinant as
\begin{gather*}
\bar{P}^{\ast}_{\mu}(x;q,q,s)=\frac{\det_{1\leq i,j\leq n}
\big(\bar{P}^{\ast}_{(\mu_j+n-j)}(x_i;q,q,s)\big)}
{\prod_{i=1}^n x_i^{1-n} \prod_{1\leq i,j\leq n}(x_i-x_j)(x_ix_j-1)},
\end{gather*}
where
\begin{gather*}
\bar{P}^{\ast}_{(k)}(z;q,q,s)=(-s)^{-k} q^{\binom{k}{2}}\big(sz^{\pm};q\big)_k.
\end{gather*}
By \eqref{Eq_qbin-qts} this yields the following determinantal expression
for the generalised binomial coefficients when $t=q$:
\begin{gather}
\qbin{\la}{\mu}_{q,q,s}=
(-1)^{\abs{\nu}}q^{n(\nu')+n(\nu)-n(\kappa)-(n-2)\abs{\nu}+(n-1)\abs{\kappa}}
\prod_{i=1}^n \frac{\big(s^2q^{2-2n};q\big)_{\nu_i}}
{(q;q)_{\nu_i}\big(s^2q^{2-2n};q\big)_{2\nu_i}} \nonumber\\
\qquad{} \times
\prod_{1\leq i<j\leq n} \frac{\big(1-q^{\nu_i-\nu_j}\big)\big(1-s^2q^{\nu_i+\nu_j-2n+2}\big)}
{\big(1-q^{\kappa_i-\kappa_j}\big)\big(1-s^2q^{\kappa_i+\kappa_j-2n+2}\big)}
\det_{1\leq i,j\leq n}\big( \big(s^2q^{\kappa_i-2n+2},q^{-\kappa_i};q\big)_{\nu_j}
\big),\label{Eq_determinantal}
\end{gather}
where $n\geq l(\mu),l(\la)$ and $\kappa=(\kappa_1,\dots,\kappa_n)$,
$\nu=(\nu_1,\dots,\nu_n)$ are strict partitions defined by
$\kappa_i:=\la_i+n-i$, $\nu_i:=\mu_i+n-i$.

\subsection{(Lifted) Koornwinder polynomials}\label{Sec_Lifted}

Let $\bart:=(t_0,t_1,t_2,t_3)$ and $x:=(x_1,\dots,x_n)$.
For our purposes the most convenient way to define the
Koornwinder polynomials
\begin{gather}
K_{\la}(q,t;\bart)=K_{\la}(x;q,t;\bart) \notag
\end{gather}
for $\la\in P_{+}(n)$
-- which are $\mathrm{BC}_n$-symmetric polynomials
with coefficients in $\mathbb{Q}(q,t,\bart)$, see~\cite{Koornwinder92}~-- is through Okounkov's binomial formula~\cite{Okounkov98}:
\begin{gather}
K_{\la}(q,t;\bart) :=\sum_{\mu\subseteq\la}
(t_0 t^{n-1})^{-\abs{\la/\mu}} t^{n(\la/\mu)}
\frac{C_{\mu}^{-}(t;q,t)C_{\mu}^{+}\big(T^2;q,t\big)}
{C_{\la}^{-}(t;q,t)C_{\la}^{+}\big(T^2;q,t\big)} \notag \\
\hphantom{K_{\la}(q,t;\bart) :=\sum_{\mu\subseteq\la}}{} \times
\frac{C^0_{\la}\big(t^n,t^{n-1}t_0t_1,t^{n-1}t_0t_2,t^{n-1}t_0t_3;q,t\big)}
{C^0_{\mu}\big(t^n,t^{n-1}t_0t_1,t^{n-1}t_0t_2,t^{n-1}t_0t_3;q,t\big)}
\qbin{\la}{\mu}_{q,t,T}
\bar{P}^{\ast}_{\mu}(q,t,t_0),
\label{Eq_binomial-formula}
\end{gather}
where $T^2:=t^{2n-2}t_0t_1t_2t_3/q$, and where we recall that
$\qbin{\la}{\mu}_{q,t,T}$ is a function of $T^2$ only.
By \eqref{Eq_interpolation-sym} and \eqref{Eq_qbin-sym} it follows that
\begin{gather}\label{Eq_K-symmetry}
K_{\la}(x;q,t;\bart)=
K_{\la}(x;1/q,1/t;1/\bart)=
(-1)^{\abs{\la}} K_{\la}(-x;q,t;-\bart),
\end{gather}
where $1/\bart$ is shorthand for $(1/t_0,1/t_1,1/t_2,1/t_3)$.
One drawback of the above definition of the Koornwinder
polynomials is that it hides the $S_4$-symmetry in the parameters
$t_0$, $t_1$, $t_2$, $t_3$.
It however follows from the connection coefficient formula
for the interpolation polynomials \cite[Theorem~3.12]{Rains05}
combined with the multivariable $q$-Pfaff--Saalsch\"utz
summation \cite[Theorem~4.2]{Rains05} that the obvious
$S_3$-symmetry lifts to~$S_4$, see~\cite{Rains05} for details.

The main result for Koornwinder polynomial that we will be needing is
Mimachi's Cauchy identity \cite[Theorem~2.1]{Mimachi01}
\begin{gather}\label{Eq_Mimachi}
\sum_{\la\subseteq (m^n)} (-1)^{\abs{\la}}
K_{(m^n)-\la}(x;q,t;\bart) K_{\la'}(y;t,q;\bart)
=\prod_{i=1}^n\prod_{j=1}^m (x_iy_j)^{-1}(x_i-y_j)(x_iy_j-1),
\end{gather}
where $y:=(y_1,\dots,y_m)$.

As mentioned in the introduction, the lifted Koornwinder polynomial~\cite{Rains05}
\begin{gather}\label{Eq_liftedK}
K_{\la}(q,t,T;\bart)=K_{\la}(x_1,x_2,\dots;q,t,T;\bart)
\end{gather}
is the unique symmetric function such that
\begin{gather}\label{Eq_Lifted-unique}
\tilde{K}_{\la}\big(x_1^{\pm},\dots,x_n^{\pm},0,0,\dots;q,t,t^n;\bart\big)
=\begin{cases}
K_{\la}(x_1,\dots,x_n;q,t;\bart) & \text{if $l(\la)\leq n$}, \\
0 & \text{otherwise}.
\end{cases}
\end{gather}
Some care is required when dealing with this function since the above
equation requires~$\bart$ to be generic.
Issues may arise for $\bart$ such that
$C^{+}_{\la}\big(t^{2n-2}t_0t_1t_2t_3/q;q,t\big)=0$.
This for example happens for the parameter choice
$\bart=\big(1,-1,t^{1/2},-t^{1/2}\big)$ in which case it is
important to specialise $T=t^n$ before specialising~$\bart$.

Recall the universal orthogonal, symplectic and special orthogonal
characters, defined in~\eqref{Eq_universal-ortho},
\eqref{Eq_universal-symp} and~\eqref{Eq_universal-special-ortho}.

\begin{Lemma}\label{Lem_Lifted}
We have
\begin{gather*}
\symp_{\la}=\tilde{K}_{\la}\big(q,q,T;q^{1/2},-q^{1/2},q,-q\big), \\
\ortho_{\la}=\tilde{K}_{\la}\big(q,q,T;1,-1,q^{1/2},-q^{1/2}\big), \\
\so_{\la}=\tilde{K}_{\la}\big(q,q,T;-1,-q^{1/2},q^{1/2},q\big).
\end{gather*}
\end{Lemma}

\begin{proof}Given that the lifted Koornwinder polynomials are the unique symmetric
functions satisfying \eqref{Eq_Lifted-unique} and the universal
characters $\symp_{\la}$, $\ortho_{\la}$ and $\so_{\la}$ are the
unique symmetric functions satisfying~\eqref{Eq_univeral-g} (in the case of $\symp_{\la}$ and $\ortho_{\la}$)
or~\eqref{Eq_universal-so} (in the case of $\so_{\la}$),
it suffices to show that
\begin{gather*}
\symp_{2n,\la}(x_1,\dots,x_n)=K_{\la}\big(x_1,\dots,x_n;q,q;q^{1/2},-q^{1/2},q,-q\big), \\
\ortho_{2n,\la}(x_1,\dots,x_n)=K_{\la}\big(x_1,\dots,x_n;q,q;1,-1,q^{1/2},-q^{1/2}\big), \\
\so_{2n+1,\la}(x_1,\dots,x_n)=K_{\la}\big(x_1,\dots,x_n;q,q;-1,-q^{1/2},q^{1/2},q\big).
\end{gather*}
All three of the above identities follow directly from
\cite[Section~2.6]{RW15}.
For example, for $t=q$ and $\bart=\big(q^{1/2},-q^{1/2},q,-q\big)$ the
Koornwinder density $\Delta(z;q,t;\bart)$ simplifies
to the standard $\mathrm{C}_n$ density
\begin{gather*}
\Delta_{\mathrm{C}}(z)=\prod_{i=1}^n \big({-}z_i^{-2}\big)\big(1-z_i^2\big)^2
\prod_{1\leq i<j\leq n} z_i^{-2}(1-z_i/z_j)^2(1-z_iz_j)^2
\end{gather*}
for which
\begin{gather*}
\frac{1}{2^n n!(2\pi\iup)^n}
\Int_{\mathbb{T}^n} \symp_{2n,\la}(z)\symp_{2n,\mu}(z)\Delta_{\mathrm{C}}(z)
\frac{\dup z_1}{z_1}\cdots\frac{\dup z_n}{z_n} = \delta_{\la,\mu}.
\end{gather*}
Similar reductions to Weyl-type orthogonality relations hold for the other two cases.
\end{proof}

\subsection{Elliptic interpolation functions}

The ($\mathrm{BC}_n$-symmetric) elliptic interpolation
functions $R^{\ast}_{\mu}(a,b;q,t;p)$
\cite{CG06,Rains06,Rains12} are an elliptic analogue of the
$\mathrm{BC}_n$-symmetric interpolation polynomials
$\bar{P}^{\ast}_{\mu}(q,t,s)$.
Although they satisfy analogous vanishing conditions, their definition
is more complicated.
Below we follow the characterisation of these function given
in~\cite{Rains06}.

A $\mathrm{BC}_n$-symmetric theta function of degree $m$ is a
$\mathrm{BC}_n$-symmetric meromorphic function $f$ on
$(\mathbb{C}^{\ast})^n$ such that
\begin{gather*}
f(px_1,x_2,\dots,x_n)=\big(1/px_1^2\big)^m f(x_1,x_2,\dots,x_n).
\end{gather*}
For example, $\prod_{i=1}^n \theta\big(ux_i^{\pm};p\big)$
for $u\in\mathbb{C}^{\ast}$ is a $\mathrm{BC}_n$-symmetric theta
function of degree $1$.

Given two partitions $\la,\mu\subset (m^n)$ such that $\la\neq\mu$ let
\begin{gather*}
l_0=\max\{i\colon \la_i\neq \mu_i\} \qquad\text{and}\qquad l_1=\max\{i\colon \la_i=m\},
\end{gather*}
where $l_1=0$ if $\la_1<m$.
Given such $l_1$ and $l_2$, further let
\begin{gather*}
l=\begin{cases}
l_1 & \text{if $\la_{l_0}<\mu_{l_0}$}, \\
l_0 & \text{otherwise}.
\end{cases}
\end{gather*}
Note that $l=n$ if and only if $\la_n>\mu_n$ and
$l=0$ if $\la_i<\mu_i$ for all $1\leq i\leq n$.

Now fix a nonnegative integer $m$ and partition $\mu\subset (m^n)$.
Then the interpolation theta function
\begin{gather*}
P^{\ast(m)}_{\mu}(x_1,\dots,x_n;a,b;q,t;p)
\end{gather*}
is the unique $\mathrm{BC}_n$-symmetric theta function of degree $m$
such that for all $\la\subset (m^n)$, $\la\neq\mu$,\footnote{It is
assumed that the parameters $a$, $b$, $q$, $t$ of $P^{\ast(m)}_{\mu}(a,b;q,t;p)$
are chosen to be generic, and similarly for~$z$ in the normalisation~\eqref{Eq_Pmn-normalise}.}
\begin{gather*}
P^{\ast(m)}_{\mu}\big(bq^{m-\la_1},\dots,bq^{m-\la_l} t^{l-1},
aq^{\la_{l+1}} t^{n-l-1},\dots,aq^{\la_n};a,b;q,t;p\big)=0
\end{gather*}
and
\begin{gather}
P^{\ast(m)}_{\mu}(z\spec{0}_{n;q,t};a,b;q,t;p) \nonumber\\
\qquad{} =C^0_{(m^n)}\big(t^{n-1} bz,b/z;q,t;p\big)
\Delta^0_{\mu}\big(aq^{-m}t^{n-1}/b\vert t^{n-1}az,a/z;q,t;p\big).\label{Eq_Pmn-normalise}
\end{gather}
Since the interpolation theta functions satisfy
\begin{gather*}
P^{\ast(m+1)}_{\mu}(x;a,b;q,t;p)
=P^{\ast(m)}_{\mu}(x;a,bq;q,t;p) \prod_{i=1}^n \theta\big(bx_i^{\pm};q,p\big),\qquad \mu\subset (m^n),
\end{gather*}
(and thus $P^{\ast(m)}_0(x;a,b;q,t;p)=\prod_{i=1}^n(bx_i^{\pm};q,p)_m$),
the ratio
\begin{gather*}
R^{\ast}_{\mu}(a,b;q,t;p):=
\frac{P^{\ast(m)}_{\mu}\big(a,bq^{-m};q,t;p\big)}
{P^{\ast(m)}_0\big(a,bq^{-m};q,t;p\big)},
\end{gather*}
is independent of $m$ (provided $m\geq\mu_1)$ and
a degree-$0$ (hence elliptic) $\mathrm{BC}_n$-symmetric theta function.
The elliptic interpolation functions satisfy vanishing conditions analogous
to~\eqref{Eq_vanishing}:
\begin{gather}\label{Eq_R-vanishing}
R_{\mu}^{\ast}(a\spec{\la}_{n;q,t};a,b;q,t;p)=0
\qquad\text{if $\mu\not\subset\la$}
\end{gather}
for $\la\in P_{+}(n)$.
They also satisfy
\begin{gather*}
R^{\ast}_{\mu+(1^n)}(x;a,b;q,t;p)=R^{\ast}_{\mu}(x;aq,b/q;q,t;p)
\prod_{i=1}^n \frac{\theta\big(ax_i^{\pm};p\big)}{\theta\big(pqx_i^{\pm}/b;p\big)},
\end{gather*}
so that, once again, they can be extended to all $\mu\in P(n)$ via
\begin{gather}\label{Eq_Rast-weights}
R^{\ast}_{\mu}(x;a,b;q,t;p)
=R^{\ast}_{\nu}\big(x;aq^{\mu_n},bq^{-\mu_n};q,t;p\big)
\prod_{i=1}^n \frac{\big(ax_i^{\pm};q,p\big)_{\mu_n}}
{\big(pqx_i^{\pm}/b;q,p\big)_{\mu_n}},
\end{gather}
where $\nu:=(\mu_1-\mu_n,\dots,\mu_{n-1}-\mu_n,0)$.
We further extend this to pairs of weights
$\bmu=\big(\mu^{(1)},\mu^{(2)}\big)\in P(n)\times P(n)$ as
\begin{gather}\label{Eq_Rast-pairs}
R^{\ast}_{\bmu}(x;a,b;t;p,q):=
R^{\ast}_{\mu^{(1)}}(x;a,b;p,t;q)
R^{\ast}_{\mu^{(2)}}(x;a,b;q,t;p).
\end{gather}
In the limit the elliptic interpolation functions
simplify to the $\mathrm{BC}_n$-symmetric interpolation
polynomials:
\begin{gather}\label{Eq_Rast-limit}
\lim_{p\to 0} R^{\ast}_{\mu}\big(s,bp^{\alpha};q,t;p\big)=
\big({-}st^{n-1}\big)^{\abs{\mu}} q^{n(\mu')}t^{-2n(\mu)}
\frac{C^{-}_{\mu}(t;q,t)}{C^0_{\mu}\big(t^n;q,t\big)}
\bar{P}_{\mu}^{\ast}(q,t,s),
\end{gather}
for $0<\alpha<1$ and $\mu\in P(n)$.

In the following we need a number of identities for the
$\mathrm{BC}_n$ interpolation function from~\cite{Rains06}.
By~\eqref{Eq_Pmn-normalise} and~\eqref{Eq_Rast-weights},
\begin{gather}\label{Eq_Rast-PS}
R^{\ast}_{\mu}(z\spec{0}_{n;q,t};a,b;q,t;p)
=\Delta^0_{\mu}\big(at^{n-1}/b\vert t^{n-1}az,a/z;q,t;p\big)
\end{gather}
for $\mu\in P(n)$, which generalises~\eqref{Eq_PS-Pbar}.
The elliptic analogue of \eqref{Eq_Pbar-normalise} is given by
\begin{gather}
R^{\ast}_{\mu}(a\spec{\mu}_{n;q,t};a,b;q,t;p) \nonumber\\
\qquad{} =\frac{1}{\Delta^0_{\mu}\big(a^2t^{2n-2}\vert abt^{n-1};q,t;p\big)
\Delta_{\mu}\big(at^{n-1}/b\vert t^n;q,t;p\big)}
\frac{C^{+}_{\mu}\big(a^2t^{2n-2};q,t;p\big)}{C^{+}_{\mu}\big(at^{n-1}/b;q,t;p\big)},\label{Eq_P39}
\end{gather}
for $\mu\in P_{+}(n)$ and the analogue of \eqref{Eq_xniss} by
\begin{gather}
R^{\ast}_{\mu}(x_1,\dots,x_n,a;a,b;q,t;p) \nonumber\\
\qquad {} =\begin{cases} \displaystyle
\frac{C^0_{\mu}\big(t^n,apq/bt;q,t;p\big)}{C^0_{\mu}\big(t^{n+1},apq/b;q,t;p\big)}
R^{\ast}_{\mu}(x_1,\dots,x_n;at,b;q,t;p)
& \text{if $\mu\in P_{+}(n)$}, \\
0 & \text{otherwise},\label{Eq_xnisa}
\end{cases}
\end{gather}
for $\mu\in P(n+1)$.
We also require the symmetry
\begin{gather}\label{Eq_Rsqrtp}
R^{\ast}_{\mu}(x;a,b;q,t;p)
=\big(aqt^{n-1}/b\big)^{\abs{\mu}} q^{2n(\la')} t^{-2n(\mu)}
R^{\ast}_{\mu}\big(xp^{1/2};ap^{1/2},bp^{1/2};q,t;p\big),
\end{gather}
for $\mu\in P(n)$.

Given partitions $\la,\mu$ and $n$ an arbitrary integer such that
$n\geq l(\la),l(\mu)$, the elliptic analogue of the binomial
coefficient $\qbin{\la}{\mu}_{q,t,s}$ is defined as \cite{Rains06}
\begin{align*}
\obinomE{\la}{\mu}_{\![a,b];q,t;p}\!
&:=\frac{\Delta^0_{\la}(a\vert b;q,t;p)}
{\Delta^0_{\mu}(a\vert b;q,t;p)}
\frac{C^{+}_{\mu}(a;q,t;p)}{C^{+}_{\mu}(a/b;q,t;p)}
\frac{R^{\ast}_{\mu}\big(a^{1/2}t^{1-n} \spec{\la}_{q,t;n};
a^{1/2}t^{1-n},b/a^{1/2};q,t;p\big)}
{R^{\ast}_{\mu}\big(a^{1/2}t^{1-n} \spec{\mu}_{q,t;n};
a^{1/2}t^{1-n},b/a^{1/2};q,t;p\big)} \\
&\hphantom{:}=
\Delta^0_{\la}(a\vert b;q,t;p)
\Delta_{\mu}\big(a/b\vert t^n;q,t;p\big)
R^{\ast}_{\mu}\big(a^{1/2}t^{1-n} \spec{\la}_{q,t;n};
a^{1/2}t^{1-n},b/a^{1/2};q,t;p\big),
\end{align*}
where the equality of the two expressions on the right follows from~\eqref{Eq_P39}.
Since
\begin{gather*}
R^{\ast}_{\mu}(x;a,b;q,t;p)=R^{\ast}_{\mu}(-x;-a,-b;q,t;p),
\end{gather*}
the elliptic binomial coefficients are a function of $a$ (as opposed
to $a^{1/2}$) so that no choice of branch is required.
Moreover, by \eqref{Eq_xnisa} they are independent of the choice of
$n$ on the right as long as $n$ is sufficiently large, and by
\eqref{Eq_R-vanishing}, they vanish unless $\mu\subset\la$.
Since, for $\mu\subset\la$, $\Delta^0_{\la}(a\vert b;q,t;p)/
\Delta^0_{\mu}(a\vert b;q,t;p)|_{b=1}=\delta_{\la\mu}$ (and since
no poles are hit by taking $b=1$ in any of the other terms in the
definition of the elliptic binomial coefficients)
\begin{gather}\label{Eq_delta}
\obinomE{\la}{\mu}_{[a,1];q,t;p}=\delta_{\la\mu}.
\end{gather}
Furthermore, since $R^{\ast}_0=1$,
\begin{gather}\label{Eq_normalisation}
\obinomE{\la}{0}_{[a,b];q,t;p}=\Delta^0_{\la}(a\vert b;q,t;p)
\qquad\text{and}\qquad
\obinomE{\la}{\la}_{[a,b];q,t;p}=
\frac{C^{+}_{\la}(a;q,t;p)}{C^{+}_{\la}(a/b;q,t;p)}.
\end{gather}
By \eqref{Eq_Rast-PS}, also for $\la=(N^n)$ the binomial
coefficients factor:
\begin{gather}\label{Eq_binomE_rectangle}
\obinomE{(N^n)}{\mu}_{[a,b];q,t;p}
=\Delta^0_{(N^n)}(a\vert b;q,t;p)
\Delta_{\mu}\big(a/b\vert t^n,aq^Nt^{1-n},q^{-N};q,t;p\big).
\end{gather}

The elliptic binomial coefficients satisfy a large number of symmetries
and identities, and for a complete list of these the reader is referred to
the original papers~\cite{Rains06,Rains10,Rains12} or the survey~\cite{RW20}.
Here we state a selection of result needed later.

The most important summation for elliptic binomial coefficients
is the convolution-type formula
\begin{gather}
\sum_{\nu\subset\mu\subset\la}
\Delta^0_{\mu}(a\vert d,e;q,t;p) \obinomE{\la}{\mu}_{[ab,b];q,t;p}
\obinomE{\mu}{\nu}_{[a,c];q,t;p} \nonumber\\
\qquad{} =\frac{\Delta^0_{\la}(ab\vert bd,be;q,t;p)}
{\Delta^0_{\nu}(a/c\vert bd,be;q,t;p)} \obinomE{\la}{\nu}_{[ab,bc];q,t;p},\label{Eq_conv}
\end{gather}
provided $bcde=apq$.
For $bc=1$ the right-hand side trivialises by~\eqref{Eq_delta},
resulting in the inversion relation
\begin{gather}\label{Eq_inversion}
\sum_{\nu\subset\mu\subset\la} \obinomE{\la}{\mu}_{[ab,b];q,t;p}
\obinomE{\mu}{\nu}_{[a,1/b];q,t;p}=\delta_{\la\nu}.
\end{gather}
By \eqref{Eq_normalisation} and \eqref{Eq_binomE_rectangle} it
follows that the special case $\la=(N^n)$ and $\nu=0$ of
\eqref{Eq_conv} corresponds to the $\mathrm{(B)C}_n$-analogue of
the elliptic Jackson sum~\eqref{Eq_FT}, see, e.g., \cite{CG06,IN17,Rains06,Rosengren01,W02}.

Two important symmetries we will rely on in Section~\ref{Sec_Phi}
are the reciprocity and conjugation symmetries
\begin{align}
\obinomE{\la}{\mu}_{[1/a,1/b];1/q,1/t;p}
&=\obinomE{\la'}{\mu'}_{[1/aqt,1/b];t,q;p} \notag \\
&=(-aq)^{\abs{\mu}-2\abs{\la}} q^{n(\mu')-4n(\la')}t^{4n(\la)-n(\mu)}
\obinomE{\la}{\mu}_{[a,b/p];q,t;p}.\label{Eq_binomE-randc}
\end{align}

Finally, as follows from \eqref{Eq_qbin-qts} and \eqref{Eq_Rast-limit},
in the limit the elliptic binomial coefficients
reduce to the binomial coefficients \eqref{Eq_qbin-qts}:
\begin{gather}\label{Eq_qtbinom-limit}
\lim_{p\to 0} \obinomE{\la}{\mu}_{\big[s^2,b/p^{1/2}\big];q,t;p}
=\big(s^2q\big)^{\abs{\la/\mu}} q^{2n(\la'/\mu')}t^{-2n(\la/\mu)}
C^{+}_{\mu}\big(s^2;q,t\big) \qbin{\la}{\mu}_{q,t,s}.
\end{gather}

In Sections~\ref{Sec_beta-integrals} and \ref{Sec_Phi}
we also use the $p,q$-symmetric variant of the elliptic
binomial coefficients, defined as
\begin{gather*}
\obinomE{\bla}{\bmu}_{[a,b];t;p,q}:=\obinomE{\la^{(1)}}{\mu^{(1)}}_{[a,b];p,t;q}
\obinomE{\la^{(2)}}{\mu^{(2)}}_{[a,b];q,t;p}.
\end{gather*}

\section{Elliptic beta integrals and interpolation kernels}\label{Sec_beta-integrals}

\subsection{Elliptic beta integrals}

Let $n$ be a positive integer, $m$ an even nonnegative integer and
$p,q,t,t_0,\dots,t_{m-1}\in\mathbb{C}^{\ast}$
such that $\abs{p},\abs{q}<1$.
Then the elliptic Dixon (or type~I) and Selberg (or type~II) densities
are defined as \cite{Rains10,Rains18,RW20}
\begin{gather*}
\Delta_{\mathrm{D}}(z;t_0,\dots,t_{m-1};p,q):=
\kappa_n \prod_{1\leq i<j\leq n}
\frac{1}{\Gamma_{p,q}\big(z_i^{\pm}z_j^{\pm}\big)} \prod_{i=1}^n \frac{\prod_{r=0}^{m-1} \Gamma_{p,q}\big(t_r z_i^{\pm}\big)}
{\Gamma_{p,q}\big(z_i^{\pm 2}\big)}
\end{gather*}
and
\begin{gather*}
\Delta_{\mathrm{S}}(z;t_0,\dots,t_{m-1};t;p,q):=
\kappa_n \prod_{1\leq i<j\leq n}
\frac{\Gamma_{p,q}\big(tz_i^{\pm}z_j^{\pm}\big)}
{\Gamma_{p,q}\big(z_i^{\pm}z_j^{\pm}\big)}
\prod_{i=1}^n \frac{\Gamma_{p,q}(t)
\prod_{r=0}^{m-1} \Gamma_{p,q}\big(t_r z_i^{\pm}\big)}
{\Gamma_{p,q}\big(z_i^{\pm 2}\big)},
\end{gather*}
where
\begin{gather*}
\kappa_n=\kappa_n(p,q):=\frac{(p;p)_{\infty}^n(q;q)_{\infty}^n} {2^n n!(2\pi\iup)^n}.
\end{gather*}
The actual elliptic Dixon and Selberg integral evaluations are
given by \cite{vDS00,vDS01,Rains10,RW20}
\begin{gather}\label{Eq_Dixon}
\Int_C \Delta_{\mathrm{D}}(z;t_0,\dots,t_{2n+3};p,q)
\frac{\dup z_1}{z_1}\cdots\frac{\dup z_n}{z_n}
=\prod_{0\leq r<s\leq 2n+3} \Gamma_{p,q}(t_rt_s)
\end{gather}
for $t_0\cdots t_{2n+3}=pq$, and
\begin{align}
S_n(t_0,\dots,t_5;t;p,q)&:=
\Int_C \Delta_{\mathrm{S}}(z;t_0,\dots,t_5;t;p,q)
\frac{\dup z_1}{z_1}\cdots\frac{\dup z_n}{z_n} \notag \\
&\hphantom{:}=\prod_{i=1}^n \bigg(\Gamma_{p,q}\big(t^i\big)
\prod_{0\leq r<s\leq 5} \Gamma_{p,q}\big(t^{i-1}t_rt_s\big)\bigg),\label{Eq_e-Selberg}
\end{align}
for $t^{2n-2}t_0\cdots t_5=pq$.
In both integrals $C$ is a deformation of~$\mathbb{T}^n$,
separating the double geometric progressions of poles of the integrands
tending to zero from those tending to infinity.
For~\eqref{Eq_e-Selberg} we must also have that $tC$ is contained in $C$.
When $\abs{u_r},\abs{t_r}<1$ for all $r$ as well as $\abs{t}<1$ in the
elliptic Selberg case, we may simply take
$C=\mathbb{T}^n$.\footnote{For a more detailed analysis of the
contours and potential issues, including existence, see the appendix
of~\cite{Rains10}.} For $n=1$ the integrals~\eqref{Eq_Dixon} and~\eqref{Eq_e-Selberg} coincide and correspond to Spiridonov's
elliptic beta integral~\cite{Spiridonov01}.

Given a $\mathrm{BC}_n$-symmetric function $f$, we define its elliptic
Selberg average by
\begin{gather*}
\big\langle f\big\rangle_{t_0,\dots,t_5;t;p,q}^{(n)}
:=\frac{1}{S_n(t_0,\dots,t_5;t;p,q)}
\Int_C f(z) \Delta_{\mathrm{S}}(z;t_0,\dots,t_5;t;p,q)
\frac{\dup z_1}{z_1}\cdots\frac{\dup z_n}{z_n},
\end{gather*}
where $C=C_f$ is as above and $t^{2n-2}t_0\cdots t_5=pq$.

\subsection{The interpolation kernel}

Let $x,y\in (\mathbb{C}^{\ast})^n$ and $c,p,q,t\in\mathbb{C}^{\ast}$
such that $\abs{p},\abs{q}<1$.
Then the interpolation kernel $\mathcal{K}_c(x;y;t;p,q)$ is
a meromorphic $\mathrm{BC}_n$-symmetric function in both $x$
and $y$, satisfying
\begin{gather}\label{Eq_kernel-symm}
\mathcal{K}_c(x;y;t;p,q)=\mathcal{K}_c(x;y;t;q,p)=\mathcal{K}_c(y;x;t;p,q)=
\mathcal{K}_{-c}(-x;y;t;p,q)
\end{gather}
such that \cite{Rains18}
\begin{gather}\label{Eq_kernel-evaluate}
\mathcal{K}_c(x;a\spec{\bmu}_{n;t;p,q}/c;t;p,q)
=R^{\ast}_{\bmu}(x;a,b;q,t;p)
\prod_{i=1}^n \Big(\frac{pq}{ab}\Big)^{2\mu^{(1)}_i\mu^{(2)}_i}
\frac{\Gamma_{p,q}\big(ax_i^{\pm},bx_i^{\pm}\big)}
{\Gamma_{p,q}\big(t^i,abt^{n-i}\big)},
\end{gather}
where $c^2=abt^{n-1}$ and
\begin{gather*}
\spec{\bmu}_{n;t;p,q}:=
\big(p^{\mu^{(1)}_1}q^{\mu^{(2)}_1}t^{n-1},\dots,
p^{\mu^{(1)}_{n-1}}q^{\mu^{(2)}_{n-1}}t,p^{\mu^{(1)}_n}q^{\mu^{(2)}_n}\big)
\end{gather*}
is an elliptic spectral vector.
The interpolation kernel may recursively be defined using
the initial conditions
\begin{gather}\label{Eq_kernel-one}
\mathcal{K}_c(\leeg;\leeg;t,p,q)=1
\qquad\text{or}\qquad
\mathcal{K}_c(x_1;y_1;t,p,q)=
\frac{\Gamma_{p,q}\big(cx_1^{\pm}y_1^{\pm}\big)}{\Gamma_{p,q}\big(t,c^2\big)}
\end{gather}
and branching rule
\begin{gather}
\mathcal{K}_c(x;y;t;p,q)=
\frac{\prod_{i=1}^n \Gamma_{p,q}\big(cy_n^{\pm}x_i^{\pm}\big)}
{\Gamma^n_{p,q}(t)\Gamma_{p,q}\big(c^2\big)\prod_{1\leq i<j\leq n}
\Gamma_{p,q}\big(tx_i^{\pm}x_j^{\pm}\big)} \nonumber\\
\qquad{} \times \Int_{C} \mathcal{K}_{c/t^{1/2}}(z;\hat{y};t;p,q)
\Delta_{\mathrm{D}}\big(t^{1/2}x^{\pm},pqy_n^{\pm}/ct^{1/2};z;p,q\big)
\frac{\dup z_1}{z_1}\cdots\frac{\dup z_{n-1}}{z_{n-1}},\label{Eq_kernel-def}
\end{gather}
for $x,y\in(\mathbb{C}^{\ast})^n$, $\hat{y}=(y_1,\dots,y_{n-1})$,
$z\in(\mathbb{C}^{\ast})^{n-1}$ and a suitable subset of the
parameter space.
By~\eqref{Eq_kernel-symm}, making the substitution
\begin{gather*}
\big(c,t^{1/2},x\big)\mapsto \big({-}c,-t^{1/2},-x\big)
\end{gather*}
(and also negating the integration variables on the right)
leaves~\eqref{Eq_kernel-def} unchanged.
Hence there is no need to fix a branch of $t^{1/2}$.
We further note that the symmetry of the kernel in $y$ is
not manifest from the recursive definition, but follows
from a similar such symmetry for the ``formal interpolation
kernel'' $K_c(x;y;q,t;p)$ defined in~\cite{Rains18} as a
generalisation of the connection coefficients identity for
$R^{\ast}_{\mu}(x;a,b;q,t)$. (By~\eqref{Eq_kernel-one},
for $n=2$ the symmetry is an immediate consequence of
a special case of the elliptic integral transformation
\cite[Theorem~4.1]{Rains10}.)

\subsection{The dual Littlewood kernel}

Following \cite{Rains18} we consider two further kernels,
known as the dual Littlewood kernel and Kawana\-ka kernel.
For the dual Littlewood kernel, $\mathcal{L}'_c(x;t;p,q)$,
let $x\in(\mathbb{C}^{\ast})^n$ and $c,p,q,t,u,v\in\mathbb{C}^{\ast}$
such that $c^4uv=p$ and $\abs{p},\abs{q}<1$.
Then \cite{Rains18}
\begin{gather}
\mathcal{L}'_c(x;t;p,q):=
\frac{1}{\prod_{i=1}^n\Gamma_{p,q}\big(cux_i^{\pm},cvx_i^{\pm}\big)}\nonumber\\
\hphantom{\mathcal{L}'_c(x;t;p,q):=}{}\times
\int \mathcal{K}_c(z;x;t;p,q)
\Delta_{\mathrm{S}}\big(z;u,v;t;p,q^2\big)
\frac{\dup z_1}{z_1}\cdots\frac{\dup z_n}{z_n}.\label{Eq_dLitt}
\end{gather}
As shown in \cite[Corollary~7.3]{Rains18}, the right-hand side
depends on the product of $u$ and $v$ only so that the dual
Littlewood kernel is well-defined.
Like the interpolation kernel~$\mathcal{K}_c$,
$\mathcal{L}'_{-c}(-x;t,p,q)=\mathcal{L}'_c(-x;t,p,q)$.
It further satisfies the symmetry
\begin{gather*}
\mathcal{L}'_c(x;pq/t;p,q)=\mathcal{L}'_c(x;t;p,q)
\prod_{1\leq i<j\leq n} \Gamma_{p,q}\big(tx_i^{\pm}x_j^{\pm}\big),
\end{gather*}
and factorisation formula
\begin{gather}\label{Eq_Lp-factor}
\mathcal{L}'_{(p/qt)^{1/2}}\big(x;t^2;p^2,q^2\big)
=\Gamma^n_{p^2,q^2}(p/qt)
\prod_{i=1}^n\Gamma_{p^2,q^4}\big(pqx_i^{\pm 2}/t\big)
\prod_{1\leq i<j\leq n} \Gamma_{p^2,q^2}\big(pqx_i^{\pm} x_j^{\pm}/t\big).
\end{gather}

Evaluating \eqref{Eq_dLitt} at $x=a\spec{\bmu}_{n;t;p,q}/c$ and
using \eqref{Eq_kernel-evaluate} as well as the simple relations
$\Gamma_{p,q}(z)=\Gamma_{p,q^2}(z,zq)$ and
\begin{gather}\label{Eq_gamma-kl}
\frac{\Gamma_{p,q}\big(zp^k q^{\ell}\big)}{\Gamma_{p,q}(z)}
=(-z)^{-k\ell} p^{-\ell\binom{k}{2}} q^{-k\binom{\ell}{2}}
(z;p,q)_k (z;q,p)_{\ell},
\end{gather}
we find
\begin{gather*}
\frac{\mathcal{L}'_c(a\spec{\bmu}_{n;t;p,q}/c;t;p,q)}
{\mathcal{L}'_c(a\spec{\boldsymbol{0}}_{n;t;p,q}/c;t;p,q)}
=\big(uvt^{n-1}\big)^{2\sum_{i=1}^n \mu^{(1)}_i\mu^{(2)}_i}
\frac{\big\langle R^{\ast}_{\bmu}(z;a,b;q,t;p)
\big\rangle_{a,aq,b,bq,u,v;t;p,q^2}^{(n)}}
{\Delta^0_{\bmu}\big(at^{n-1}/b\vert aut^{n-1},avt^{n-1};t;p,q\big)},
\end{gather*}
where, as before, $c^4uv=p$ and $c^2=abt^{n-1}$.
Replacing $(p,q,t)\mapsto \big(p^2,q^2,t^2\big)$ and then
specialising $c=(p/qt)^{1/2}$ the left-hand side factors
thanks to~\eqref{Eq_Lp-factor}.
By \eqref{Eq_gamma-kl} this yields
\begin{gather}
 \big\langle R^{\ast}_{\bmu}\big(z;a,b;t^2;p^2,q^2\big)
\big\rangle_{a,aq^2,b,bq^2,u,v;t^2;p^2,q^4}^{(n)} \notag \\
\qquad{} = \Delta^0_{\bmu}\big(at^{2n-2}/b\vert
a^2q^2t^{2n-2},aut^{2n-2},avt^{2n-2};t^2;p^2,q^2\big) \notag\\
 \qquad\quad{} \times
 \frac{C^{-}_{\bmu}\big(pqt;t^2;p^2,q^2\big)
C^{+}_{\bmu}\big(a^2q^2t^{4n-4};t^2;p^2,q^2\big)}
{C^0_{(2,1)\bmu^2}\big(a^2q^4t^{4n-2};t^2;p^2,q^4\big)},\label{Eq_pre-Thm-integral}
\end{gather}
where $uv=(qt)^2$ and $abqt^{2n-1}=p$.
For later convenience we interchange $p$ and $q$
$\big($and $\big(\mu^{(1)},\mu^{(2)}\big)\mapsto \big(\mu^{(2)},\mu^{(1)}\big)\big)$.
Using the $p,q$-symmetry of the Selberg average,
and finally replacing $(a,b,u,v)\mapsto (q/bt,q/at,ptv,pt/v)$,
\eqref{Eq_pre-Thm-integral}~may equivalently be stated as
in our next theorem.

\begin{Theorem}\label{Thm_integral}
For $\bmu\in P_{+}(n)\times P_{+}(n)$ and
$a,b,v,p,q,t\in\mathbb{C}^{\ast}$ such that $ab=pqt^{2n-3}$ and
$\abs{p},\abs{q}<1$,
\begin{gather}
 \big\langle R^{\ast}_{\bmu}\big(z;q/bt,q/at;t^2;p^2,q^2\big)
\big\rangle^{(n)}_
{q/at,p^2q/at,q/bt,p^2q/bt,ptv,pt/v;t^2;p^4,q^2} \notag \\
\qquad =\Delta^0_{\bmu}\big(a^2t/pq\big\vert
a^2t^{2-2n},atv^{\pm};t^2;p^2,q^2\big)
\frac{C^{-}_{\bmu}\big(pqt;t^2;p^2,q^2\big)C^{+}_{\bmu}\big(a^2;t^2;p^2,q^2\big)}
{C^0_{(1,2)\bmu^2}\big((apt)^2;t^2;p^4,q^2\big)}.\label{Eq_weird}
\end{gather}
\end{Theorem}

Eliminating $b$, let $f_{\mu}(a,v;t;p,q)$ denote the left-hand side
of~\eqref{Eq_weird} for $\bmu=(0,\mu)$ and \mbox{$\mu\in P(n)$}.
By~\eqref{Eq_Rast-weights} and~\eqref{Eq_gamma-kl}
\begin{gather*}
f_{\mu}(a,v;t;p,q)=f_{\nu}\big(aq^{2\mu_n},v;t;p,q\big),
\end{gather*}
where $\nu:=(\mu_1-\mu_n,\dots,\mu_{n-1}-\mu_n,0)$.
For this special choice of $\bmu$, the same functional equation is
satisfied by the right-side of~\eqref{Eq_Rast-weights} thanks to~\eqref{Eq_C-weights}, leading to the following corollary.

\begin{Corollary}\label{Cor_E-integral}
For $\mu\in P(n)$ and $a,b,v,p,q,t\in\mathbb{C}^{\ast}$ such that
$ab=pqt^{2n-3}$ and $\abs{p},\abs{q}<1$,
\begin{gather}
\big\langle R^{\ast}_{\mu}\big(z;q/bt,q/at;q^2,t^2;p^2\big)
\big\rangle^{(n)}_
{q/at,p^2q/at,q/bt,p^2q/bt,ptv,pt/v;t^2;p^4,q^2} \notag \\
\qquad {}=\Delta^0_{\mu}\big(a^2t/pq\big\vert
a^2t^{2-2n},atv^{\pm};q^2,t^2;p^2\big)
\frac{C^{-}_{\mu}\big(pqt;q^2,t^2;p^2\big)C^{+}_{\mu}\big(a^2;q^2,t^2;p^2\big)}
{C^0_{2\mu^2}\big((apt)^2;q^2,t^2;p^4\big)}.\label{Eq_weird2}
\end{gather}
\end{Corollary}

According to \cite[Corollary~4.14]{Rains12}, if $a^2b^2uvt^{2n-2}=p$
then\footnote{The statement of \cite[Corollary~4.14]{Rains12} contains
a minor typo in the argument of~$\Delta^0_{\bla}$.}
\begin{gather}
 \big\langle R^{\ast}_{\bla}(z;a,b;t;p,q)\big\rangle^{(n)}_
{a,aq,b,bq,u,v;t;p,q^2}
=\Delta_{\bla}^0\big(at^{n-1}/b\vert
a^2qt^{n-1},aut^{n-1},avt^{n-1};t;p,q\big) \notag \\
 \qquad{}\times\sum_{\bmu}
\frac{\Delta_{\bmu}\big(1/b^2\vert t^n,a^2qt^{n-1};t;p,q^2\big)}
{\Delta_{(1,2)\bmu}\big(1/b^2\vert t^n,a^2qt^{n-1};t;p,q\big)}
\obinomE{\bla}{(1,2)\bmu}_{[at^{n-1}/b,abt^{n-1}];t;p,q}. \label{Eq_Cor414}
\end{gather}
Combining this with \eqref{Eq_pre-Thm-integral} and using
\begin{gather*}
\frac{\Delta^0_{\bmu}\big(1/b^2\vert t^n,a^2qt^{n-1};t;p,q^2\big)}
{\Delta^0_{(1,2)\bmu}\big(1/b^2\vert t^n,a^2qt^{n-1};t;p,q\big)}
=\frac{1}{\Delta^0_{\bmu}\big(1/b^2\vert qt^n,a^2q^2t^{n-1};p,q^2\big)},
\end{gather*}
which is equal to $1$ if $a^2b^2qt^{2n-1}=p$,
leads to the following quadratic summation formula.

\begin{Corollary}\label{Cor_quadratic-sum}
For $\bla\in P_{+}(n)\times P_{+}(n)$,
\begin{gather*}
\sum_{\bmu} \frac{\Delta_{\bmu}\big(a\vert\leeg;t^2;p^2,q^4\big)}
{\Delta_{(1,2)\bmu}\big(a\vert\leeg;t^2;p^2,q^2\big)}
\obinomE{\bla}{(1,2)\bmu}_{[apq/t,p/qt];t^2;p^2,q^2}\\
\qquad{} = \frac{C^{-}_{\bla}\big(pqt;t^2;p^2,q^2\big)
C^{+}_{\bla}\big(ap^2/t^2;t^2;p^2,q^2\big)}
{C^0_{(2,1)\bla^2}\big(ap^2q^2;t^2;p^2,q^4\big)}.
\end{gather*}
\end{Corollary}

\subsection{The Kawanaka kernel}

For the Kawanaka kernel, $\mathcal{L}^{-}_c(x;t;p,q)$,
we have a very similar definition and set of results as for
the dual Littlewood kernel.

Let $x\in(\mathbb{C}^{\ast})^n$ and $c,p,q,t,u,v\in\mathbb{C}^{\ast}$
such that $c^2uv=pq$ and $\abs{p},\abs{q}<1$.
Then the Kawanaka kernel is defined as \cite{Rains18}
\begin{gather}
\mathcal{L}^{-}_c(x;t;p,q) :=\prod_{i=1}^n \frac{1}
{\Gamma_{p^2,q^2}\big(cu^2x_i^{\pm},cv^2x_i^{\pm}\big)} \notag \\
\hphantom{\mathcal{L}^{-}_c(x;t;p,q) :=}{} \times \int \mathcal{K}_c\big(z^2;x;t^2;p^2,q^2\big)
\Delta_{\mathrm{S}}(z;u,v;t;p,q)
\frac{\dup z_1}{z_1}\cdots\frac{\dup z_n}{z_n},\label{Eq_Kawa}
\end{gather}
which once again does not depend on the individual choice of~$u$ and~$v$.
The Kawanaka kernel satisfies the symmetry
\begin{gather*}
\mathcal{L}^{-}_c(x;pq/t;p,q)=\mathcal{L}^{-}_c(x;-t;p,q)
\Gamma^n_{p^2,q^2}\big(t^2\big) \prod_{1\leq i<j\leq n}
\Gamma_{p^2,q^2}\big(t^2x_i^{\pm}x_j^{\pm}\big),
\end{gather*}
and factorisation formula
\begin{gather}
\mathcal{L}^{-}_{(pq/t)^{1/2}}(x;-t;p,q)
=\Gamma^n_{p^2,q^2}(pq/t)
\prod_{i=1}^n\Gamma_{p,q}\big({-}(pq/t)^{1/2}x_i^{\pm}\big)\!
\prod_{1\leq i<j\leq n} \Gamma_{p^2,q^2}\big(pqx_i^{\pm} x_j^{\pm}/t\big).\!\!\!\label{Eq_L-factor}
\end{gather}

Evaluating \eqref{Eq_Kawa} at $x=a\spec{\bmu}_{n;t^2;p^2,q^2}/c$ and
proceeded exactly as in the dual Littlewood case, also using
\begin{gather}\label{Eq_Gamma-square}
\Gamma_{p^2,q^2}\big(z^2\big)=\Gamma_{p,q}(z,-z),
\end{gather}
we find
\begin{gather*}
 \frac{\mathcal{L}^{-}_c\big(a\spec{\bmu}_{n;t^2;p^2,q^2}/c;t;p,q\big)}
{\mathcal{L}^{-}_c\big(a\spec{\boldsymbol{0}}_{n;t^2;p^2,q^2}/c;t;p,q\big)} \\
 \qquad{} =\big(u^2v^2t^{2n-2}\big)^{2\sum_{i=1}^n \mu^{(1)}_i\mu^{(2)}_i}
\frac{\big\langle R^{\ast}_{\bmu}\big(z^2;a,b;q^2,t^2;p^2\big)
\big\rangle_{a^{1/2},-a^{1/2},b^{1/2},-b^{1/2},u,v;t;p,q}^{(n)}}
{\Delta^0_{\bmu}\big(at^{2n-2}/b\vert au^2t^{2n-2},av^2t^{2n-2};t^2;p^2,q^2\big)},
\end{gather*}
where $c^2uv=pq$ and $c^2=abt^{2n-2}$.
Replacing $t\mapsto -t$ and then specialising $c=(pq/t)^{1/2}$, the
left-hand side factors by~\eqref{Eq_L-factor}, leading to our
next theorem.

\begin{Theorem}\label{Thm_integral-Kawa}
For $\bmu\in P_{+}(n)\times P_{+}(n)$ and
$a,b,u,v,p,q,t\in\mathbb{C}^{\ast}$ such that $abt^{2n-1}=pq$, $uv=t$
and $\abs{p},\abs{q}<1$,
\begin{gather}
\big\langle R^{\ast}_{\bmu}\big(z^2;a,b;t^2;p^2,q^2\big)
\big\rangle^{(n)}_
{a^{1/2},-a^{1/2},b^{1/2},-b^{1/2},u,v;-t;p,q} \notag \\
\qquad=\Delta^0_{\bmu}\big(at^{2n-2}/b\big\vert
a^2t^{2n-2},au^2t^{2n-2},av^2t^{2n-2};t^2;p^2,q^2\big) \notag \\
\qquad\quad{}\times
\frac{C^{-}_{\bmu}\big(pqt;t^2;p^2,q^2\big)C^{+}_{\bmu}\big(a^2t^{4n-4};t^2;p^2,q^2\big)}
{C^0_{2\bmu^2}(-at^{2n-1};-t;p,q)}.\label{Eq_weird-Kawa}
\end{gather}
\end{Theorem}

From \eqref{Eq_add-square}, \eqref{Eq_gamma-kl} and \eqref{Eq_Gamma-square}
it follows that~\eqref{Eq_weird-Kawa} for $\bmu=(0,\mu)$ again extends
from partitions to weights.

\begin{Corollary}For $\mu\in P(n)$ and $a,b,u,v,p,q,t\in\mathbb{C}^{\ast}$ such that $abt^{2n-1}=pq$, $uv=t$
and $\abs{p},\abs{q}<1$,
\begin{gather*}
 \big\langle R^{\ast}_{\mu}\big(z^2;a,b;q^2,t^2;p^2\big)
\big\rangle^{(n)}_
{a^{1/2},-a^{1/2},b^{1/2},-b^{1/2},u,v;-t;p,q} \\
 \qquad{}=\Delta^0_{\mu}\big(at^{2n-2}/b\big\vert
a^2t^{2n-2},au^2t^{2n-2},av^2t^{2n-2};q^2,t^2;p^2\big) \\
 \qquad\quad{}\times
\frac{C^{-}_{\mu}\big(pqt;q^2,t^2;p^2\big)C^{+}_{\mu}\big(a^2t^{4n-4};q^2,t^2;p^2\big)}
{C^0_{2\mu^2}\big({-}at^{2n-1};q,-t;p\big)}.
\end{gather*}
\end{Corollary}

The analogue of \eqref{Eq_Cor414} in the Kawanaka case is given by
\cite[Corollary~4.16]{Rains12}
\begin{gather*}
 \big\langle R^{\ast}_{\bla}\big(z^2;a,b;t^2;p^2,q^2\big)
\big\rangle^{(n)}_{a^{1/2},-a^{1/2},b^{1/2},-b^{1/2},u,v;-t;p,q}
 \\
 \qquad{}=\Delta^0_{\bla}\big(at^{2n-2}/b\vert
a^2t^{2n-2},au^2t^{2n-2},av^2t^{2n-2};t^2;p^2,q^2\big) \\
 \qquad\quad{} \times \sum_{\bmu}
\frac{\Delta_{\bmu}\big(1/b\vert (-t)^n,-a(-t)^{n-1};-t;p,q\big)}
{\Delta_{\bmu}\big(1/b^2\vert t^{2n},a^2t^{2n-2};t^2;p^2,q^2\big)}
\obinomE{\bla}{\bmu}_{[at^{2n-2}/b,abt^{2n-2}];t^2;p^2,q^2}
\end{gather*}
for $abuvt^{2n-2}=pq$.
Since for $abt^{2n-1}=pq$
\begin{gather*}
\frac{\Delta^0_{\bmu}\big(1/b\vert (-t)^n,-a(-t)^{n-1};-t;p,q\big)}
{\Delta^0_{\bmu}\big(1/b^2\vert t^{2n},a^2t^{2n-2};t^2;p^2,q^2\big)}
=\frac{1}
{\Delta^0_{\bmu}\big(1/b\vert {-}(-t)^n,a(-t)^{n-1};-t;p,q\big)}=1,
\end{gather*}
we thus obtain the quadratic summation formula of the next corollary.

\begin{Corollary}\label{Cor_Kawanaka-sum}
For $\bla\in P_{+}(n)\times P_{+}(n)$,
\begin{gather}
 \sum_{\bmu} \frac{\Delta_{\bmu}(a\vert\leeg;-t;p,q)}
{\Delta_{\bmu}\big(a^2\vert\leeg;t^2;p^2,q^2\big)}
\obinomE{\bla}{\bmu}_{[a^2pq/t,pq/t];t^2;p^2,q^2} \notag \\
 \qquad = \frac{C^{-}_{\bla}\big(pqt;t^2;p^2,q^2\big)
C^{+}_{\bla}\big(a^2p^2q^2/t^2;t^2;p^2,q^2\big)}
{C^0_{2\bla^2}(-apq;-t;p,q)}.\label{Eq_Kawanaka-sum}
\end{gather}
\end{Corollary}
Here we note that $a$ in \eqref{Eq_weird-Kawa} has been replaced by
$apqt^{1-2n}$.

\section{Transition coefficients via elliptic hypergeometric integrals}

\subsection{Transition coefficients}

Recall that the Macdonald polynomials $P_{\la}$ indexed by weights
$\la\in P(n)$ form a basis $\Lambda_{\mathrm{GL}(n)}$.
In particular, any $\mathrm{BC}_n$-symmetric polynomial can be expanded in
terms of Macdonald polynomials.
If $\{f_{\la}\}$ (for $\la\in P_{+}(n)$ or $\la\in P(n)$) is a basis of
$\Lambda_{\mathrm{BC}(n)}$ and $g$ an arbitrary element of
$\Lambda_{\mathrm{BC}(n)}$ which expands in this basis as
\begin{gather*}
g=\sum_{\la} c_{\la} f_{\la},
\end{gather*}
we will write $[f_{\la}] g$ to denote the coefficient $c_{\la}$.
By~\eqref{Eq_P-symmetry} it then follows that
\begin{gather}\label{Eq_Pf}
\big[P_{(\la_1,\dots,\la_n)}(q,t)\big] g =
\big[P_{(-\la_n,\dots,-\la_1)}(q,t)\big] g,
\end{gather}
so that it suffices to consider $\big[P_{\la}(q,t)\big]g$
for $\la\in P(n)$ such that $\la_1\geq 0$.

We are concerned with computing the transition coefficients
\begin{gather}\label{Eq_c-small}
\cc_{\la\mu}^{(n)}(q,t,s):=\big[P_{\la}(q,t)\big]
\bar{P}^{\ast}_{\mu}(q,t,s)\in \mathbb{Q}(q,t)\big[s,s^{-1}\big]
\end{gather}
for $\la\in P(n)$ and $\mu\in P_{+}(n)$.
Apart from
\begin{gather*}
\cc_{(\la_1,\dots,\la_n),\mu}^{(n)}(q,t,s)=
\cc_{(-\la_n,\dots,-\la_1),\mu}^{(n)}(q,t,s),
\end{gather*}
it follows from \eqref{Eq_interpolation-sym}, the homogeneity of
the Macdonald polynomials and \eqref{Eq_Mac-recip} that
$\cc_{\la\mu}^{(n)}$ satisfies the simple relations
\begin{gather*}
\cc_{\la\mu}^{(n)}(q,t,s)=\cc_{\la\mu}^{(n)}(1/q,1/t,1/s)=
(-1)^{\abs{\la}-\abs{\mu}}\cc_{\la\mu}^{(n)}(q,t,-s).
\end{gather*}

It will be convenient to scale $\cc_{\la\mu}^{(n)}(q,t,s)$ to a function
that depends polynomially on $s^2$.
To this end we define
\begin{gather}\label{Eq_C-large}
\CC_{\la\mu}^{(n)}(q,t,s):=(-st^{n-1})^{\abs{\mu}-\abs{\la}}
q^{n(\mu')-n(\la')} t^{2n(\la)-2n(\mu)}
\cc_{\la\mu}^{(n)}(q,t,s)\in \mathbb{Q}(q,t)\big[s^2\big],
\end{gather}
where we recall that $n(\la)$ and $n(\la')$ are defined for arbitrary
weights $\la$ on page~\pageref{P_part}.
Some of the above symmetries for $\cc_{\la\mu}^{(n)}$ translate to
\begin{gather}\label{C-sym}
\CC_{\la\mu}^{(n)}(q,t,s)=\CC_{\la\mu}^{(n)}(q,t,-s)=
\left(\frac{q}{s^2}\right)^{\abs{\la}}
\CC_{(-\la_n,\dots,-\la_1),\mu}^{(n)}(q,t,s).
\end{gather}

\begin{Lemma}Let $\la,\mu\in P_{+}(n)$. Then
\begin{gather*}
\CC_{\la\mu}^{(n)}(q,t,s)=0\qquad \text{if $\;\la\not\subset\mu$}
\end{gather*}
and $\CC_{\la\la}^{(n)}(q,t,s)=1$.
\end{Lemma}

\begin{proof}\label{Lem_c-vanishing}
According to \cite[Theorem~6.16]{Rains05}, for $\la,\mu\in P_{+}(n)$,
\begin{gather*}
\bar{C}_{\la\mu}:= \big[P_{\la}\big(x^{\pm};q,t\big)\big]\bar{P}^{\ast}_{\mu}(x;q,t,s)=0
\end{gather*}
if $\la\not\subset\mu$
(where $\bar{C}_{\la\mu}\in\mathbb{Q}(q,t)[s^{\pm}]$),
and $\bar{C}_{\mu\mu}=1$.
We also have, for $\la\in P_{+}(n)$ and $\nu\in P(n)$, that
\begin{gather*}
\hat{C}_{\nu\la}:=\big[P_{\nu}(x;q,t)\big]P_{\la}\big(x^{\pm};q,t\big)=0
\end{gather*}
if $\abs{\la}-\abs{\nu}$ is odd, or if there exists an
$1\leq i\leq n$ such that $\abs{\nu_i}>\la_i$
(where $\hat{C}_{\nu\la}\in\mathbb{Q}(q,t)$),
and $\hat{C}_{\la\la}=1$.
Combining these two results the claim immediately follows.
\end{proof}

\begin{Corollary}\label{Cor_CC-vanish}
Let $N\in\mathbb{Z}$ and $\mu\in P_{+}(n)$. Then
\begin{gather*}
\CC_{(N^n),\mu}^{(n)}(q,t,s)=0\qquad \text{if $\big(\abs{N}^n\big)\not\subset\mu$}.
\end{gather*}
\end{Corollary}

It seems to be a hard problem to get a handle on the general form
of $\CC_{\la\mu}^{(n)}(q,t,s)$.
When $n=1$ it is a straightforward exercise in
basic hypergeometric series to show that
for $N$ an integer and $k$ a nonnegative integer
\begin{gather}\label{Eq_nis1}
\CC_{(N),(k)}^{(1)}\big(q,t,sq^{1/2}\big)=
\sum_{i=0}^k s^{2i} q^{i(i+N)} \qbin{k}{i}_q\qbin{k}{i+N}_q.
\end{gather}
Here $\qbin{k}{i}_q$ is the standard $q$-binomial coefficient
\begin{gather}
\qbin{k}{i}_q=
\begin{cases} \displaystyle
\frac{(q;q)_k}{(q;q)_i(q;q)_{k-i}} & \text{for $0\leq i\leq k$}, \\
0 & \text{otherwise}.
\end{cases} \notag
\end{gather}
For $s=1$ the sum in \eqref{Eq_nis1} can be performed by the
$q$-Chu--Vandermonde summation \cite[equation~(II.7)]{GR04} to give
\begin{gather*}
\CC_{(N),(k)}^{(1)}\big(q,t,q^{1/2}\big)=\qbin{2k}{k+N}_q, \notag
\end{gather*}
which generalises nicely for $\la=(N^n)$.

\begin{Theorem}\label{Thm_PPbarast}
For $N$ a nonnegative integer and $\mu\in P_{+}(n)$ such that
$(N^n)\subset\mu$,
\begin{gather}\label{Eq_PPbarast}
\CC_{(N^n),\mu}^{(n)}\big(q,t,q^{1/2}\big)=
\frac{C^0_{\mu}\big(t^n;q,t\big)C^{+}_{\mu}\big(qt^{2n-2};q,t\big)}{C^{-}_{\mu}(t;q,t)}
\prod_{i=1}^n \frac{\big(qt^{n-i};q\big)_{\mu_i}}{\big(qt^{n-i};q\big)_{\mu_i-N}\big(qt^{n-i};q\big)_{\mu_i+N}}.
\end{gather}
\end{Theorem}

By \eqref{C-sym} with $s^2=q$, the restriction that $N$ is nonnegative
is non-essential, and it is of course clear that the right-hand side
of~\eqref{Eq_PPbarast} is invariant under negation of~$N$.
Replacing $\mu\mapsto\mu+(N^n)$ and using
\eqref{Eq_add-square} and~\eqref{Eq_P-addsquare},
it follows that~\eqref{Eq_PPbarast} may also be stated as
\begin{gather}\label{Eq_PbarPast-2}
\CC_{(N^n),\mu+(N^n)}^{(n)}\big(q,t,q^{1/2}\big)=
\frac{C^0_{\mu}\big(t^n;q,t\big)C^{+}_{\mu}\big(q^{2N+1}t^{2n-2};q,t\big)}
{C^0_{\mu}\big(qt^{n-1};q,t\big)C^{-}_{\mu}(t;q,t)}.
\end{gather}
By the connection coefficient formula for $\bar{P}^{\ast}_{\mu}(q,t,s)$,
see \cite[Theorem~3.12]{Rains05}, \eqref{Eq_PPbarast} leads to an
expression for the more general transition coefficient
$\CC_{(N^n),\mu}^{(n)}\big(q,t,sq^{1/2}\big)$.
Since this result is not needed later, we omit the details.
We do remark, however, that this expression does not (in an obvious manner)
generalise~\eqref{Eq_nis1}, but instead extends the alternative form
\begin{gather}\label{Eq_nis1-alt}
\CC_{(N),(k)}^{(n)}\big(q,t,sq^{1/2}\big)=
\sum_{i=0}^k s^{i-N} \frac{(sq;q)_k(s;q)_{k-i}}{(sq;q)_i}
\qbin{k}{i}_q\qbin{2i}{i+N}_q,
\end{gather}
which obscures the fact that this is polynomial in $s^2$.
The equality of~\eqref{Eq_nis1} and~\eqref{Eq_nis1-alt} follows from
the transformation formula
\begin{gather*}
\qhypc{2}{1}\left[\genfrac{}{}{0pt}{}
{q^{-N}/b,q^{-N}}{bq};q,a^2q^{2N+1}\right]
=(a/b,aq;q)_N \: \qhypc{4}{3}\left[\genfrac{}{}{0pt}{}
{bq^{1/2},-bq^{1/2},-bq,q^{-N}}{b^2q,aq,bq^{1-N}/a};q,q\right],
\end{gather*}
which we have not yet succeeded in generalising to the multivariable
setting. The more general expression for $\CC_{(N^n),\mu+(N^n)}^{(n)}\big(q,t,sq^{1/2}\big)$ does however show it to be
a~rational function in $\mathbb{Q}\big(q,t,s,t^n,q^Nt^n\big)$.

\subsection{Proof of Theorem~\ref{Thm_PPbarast}}\label{Sec_proof}

In this section we give a proof of Theorem~\ref{Thm_PPbarast} based on the elliptic hypergeometric integral~\eqref{Eq_weird2}.

\begin{Proposition}\label{Prop_integral}
For $\mu\in P(n)$ and $q,t,s\in\mathbb{C}$ such that
$0<\abs{q},\abs{t}<1$ and $\abs{sq^{\mu_n}}<1$,
\begin{gather}
\bigg\langle \bar{P}^{\ast}_{\mu}(z;q,t,s)
\prod_{i=1}^n \frac{\theta(vz_i;q)}{\big(sz_i^{\pm};q\big)_{\infty}}
\bigg\rangle_{q,t}^{(n)} \notag \\
\qquad {} =\big({-}st^{n-1}\big)^{-\abs{\mu}} q^{-n(\mu')} t^{2n(\mu)}
\frac{C^0_{\mu}\big(t^n,s^2t^{n-1};q,t\big)C^{+}_{\mu}\big(s^2t^{2n-2};q,t\big)}
{C^0_{\mu}\big(svt^{n-1},sv^{-1}qt^{n-1};q,t\big)C^{-}_{\mu}(t;q,t)} \notag \\
\qquad\quad{} \times
\prod_{i=1}^n \frac{\big(svt^{i-1},sv^{-1}qt^{i-1};q\big)_{\infty}}
{\big(qt^{i-1},s^2t^{i-1};q\big)_{\infty}}.\label{Eq_prop-limit}
\end{gather}
\end{Proposition}

The integrand on the left has simple poles at $z_i=\big(sq^{\mu_n+k}\big)^{\sigma}$
for $\sigma\in\{-1,1\}$, $1\leq i\leq n$ and~$k$ a nonnegative integer.
The condition $\abs{sq^{\mu_n}}<1$ ensures that the
poles at $z_i=sq^{\mu_n+k}$ lie in the interior of~$\mathbb{T}^n$
and the poles at $z_i=s^{-1}q^{-\mu_n-k}$ lie in the exterior.

Before showing how \eqref{Eq_prop-limit} follows from~\eqref{Eq_weird2},
we first use the former to prove Theorem~\ref{Thm_PPbarast}.

\begin{proof}[Proof of Theorem~\ref{Thm_PPbarast}]
Let $N$ be a nonnegative integer and replace $s,v\mapsto q^{N+1/2}$
in~\eqref{Eq_prop-limit}.
By~\eqref{Eq_theta-fun} and~\eqref{Eq_P-addsquare} this yields
\begin{gather}
 \big({-}q^{N+1/2}t^{n-1}\big)^{\abs{\mu}} q^{n(\mu')} t^{-2n(\mu)}
\bigg\langle \bar{P}^{\ast}_{\mu+(N^n)}\big(z;q,t,q^{1/2}\big)
\prod_{i=1}^n z_i^{-N} \bigg\rangle_{q,t}^{(n)} \notag \\
 \qquad{} =\frac{C^0_{\mu}(t^n;q,t)C^{+}_{\mu}\big(q^{2N+1}t^{2n-2};q,t\big)}
{C^0_{\mu}\big(qt^{n-1};q,t\big)C^{-}_{\mu}(t;q,t)}.\label{Eq_barPziN}
\end{gather}
By the orthogonality \eqref{Eq_P-ortho},
\begin{gather*}
\cc_{\la\mu}(q,t,s)=
\frac{\big\langle P_{\la}\big(z^{-1};q,t\big)\bar{P}^{\ast}_{\mu}(z;q,t,s)
\big\rangle_{q,t}^{(n)}} {\big\langle P_{\la}\big(z^{-1};q,t\big)P_{\la}(z;q,t)
\big\rangle_{q,t}^{(n)}}.
\end{gather*}
Since $P_{(N^n)}(z;q,t)=\prod_{i=1}^n z_i^N$, for $\la=(N^n)$ this
simplifies to
\begin{gather*}
\cc_{(N^n),\mu}(q,t,s)=\bigg\langle \bar{P}^{\ast}_{\mu}(z;q,t,s) \prod_{i=1}^n z_i^{-N} \bigg\rangle_{q,t}^{(n)}.
\end{gather*}
Shifting $\mu\mapsto\mu+(N^n)$ and using \eqref{Eq_C-large},
this yields
\begin{gather*}
\CC_{(N^n),\mu+(N^n)}\big(q,t,q^{1/2}\big)
=\frac{C^0_{\mu}\big(t^n;q,t\big)C^{+}_{\mu}\big(q^{2N+1}t^{2n-2};q,t\big)}
{C^0_{\mu}\big(qt^{n-1};q,t\big)C^{-}_{\mu}(t;q,t)}.
\end{gather*}
Equating this with \eqref{Eq_barPziN} we obtain \eqref{Eq_PbarPast-2}.
\end{proof}

\begin{proof}[Proof of Proposition~\ref{Prop_integral}]
We begin by making the substitutions $a\mapsto st^{2n-2}$ and $v\mapsto v/p$ in~\eqref{Eq_weird2}, resulting in
\begin{gather}
\big\langle R^{\ast}_{\mu}\big(z;s/p,qt^{1-2n}/s;q^2,t^2;p^2\big)
\big\rangle^{(n)}_
{qt^{1-2n}/s,p^2qt^{1-2n}/s,s/p,sp,tv,p^2t/v;t^2;p^4,q^2} \notag \\
\qquad{} = \Delta^0_{\mu}\big(s^2t^{4n-3}/pq\big\vert
s^2t^{2n-2},st^{2n-1}v/p,spt^{2n-1}/v;q^2,t^2;p^2\big) \notag \\
 \qquad\quad{}\times
\frac{C^{+}_{\mu}\big(s^2t^{4n-4};q^2,t^2;p^2\big)
C^{-}_{\mu}\big(pqt;q^2,t^2;p^2\big)}
{C^0_{2\mu^2}\big(sp^2t^{4n-2};q^2,t^2;p^4\big)}.\label{Eq_int-sub}
\end{gather}
The general method for taking the $p\to 0$ limit of integrals such as~\eqref{Eq_int-sub} was developed in~\cite{Rains09} and has also applied more recently in~\cite{ARW20}.
It relies on a trick to break the~$\mathrm{BC}_n$ symmetry, resulting in an
integral with $S_n$-symmetry only.
Denote the left-hand side of \eqref{Eq_int-sub} by~$\mathscr{L}$. Then, by
\begin{gather*}
\Gamma_{p^2,q}(z,pz)=\Gamma_{p,q}(z),
\end{gather*}
we have
\begin{gather*}
\mathscr{L} =\frac{1}
{S_n\big(qt^{1-2n}/s,p^2qt^{1-2n}/s,s/p,sp,tv,p^2t/v;t^2;p^4,q^2\big)} \\
\hphantom{\mathscr{L} =}{} \times \int
R^{\ast}_{\mu}\big(z;s/p,qt^{1-2n}/s;q^2,t^2;p^2\big)
\Delta_{\mathrm{S}}\big(z;tv,p^2t/v;t^2;p^4,q^2\big) \\
\hphantom{\mathscr{L} =\times \int}{} \times
\prod_{i=1}^n \Gamma_{p^2,q^2}\big(sz_i^{\pm}/p,qt^{-n}z_i^{\pm}/s\big)
\frac{\dup z_1}{z_1}\cdots\frac{\dup z_n}{z_n}.
\end{gather*}
By $\mathrm{BC}_n$-symmetry and the identity
\begin{gather*}
1=\sum_{\sigma\in\{\pm 1\}^n} \prod_{i=1}^n
\frac{\theta\big(t_0 z_i^{\sigma_i},t_1 z_i^{\sigma_i},t_2 z_i^{\sigma_i},
t^{n-1}t_0t_1t_2z_i^{-\sigma_i};q\big)}
{\theta\big(z_i^{2\sigma_i},t^{i-1}t_0t_1,t^{i-1}t_0t_2,t^{i-1}t_1t_2;q\big)}
\prod_{1\leq i<j\leq n} \frac{\theta\big(t z_i^{\sigma_i}z_j^{\sigma_j};q\big)}
{\theta\big(z_i^{\sigma_i}z_j^{\sigma_j};q\big)},
\end{gather*}
we may multiply the integrand by the symmetry-breaking factor
\begin{gather*}
2^n \prod_{i=1}^n
\frac{\theta\big(sz_i/p,qt^{1-2n}z_i/s,tvz_i,qvz_i^{-1}/p;q^2\big)}
{\theta\big(z_i^2,qt^{2i-2n-1}/p,
st^{2i-1}v/p,qt^{2i-2n}v/s;q^2\big)}
\prod_{1\leq i<j\leq n} \frac{\theta\big(t^2z_iz_j;q^2\big)}
{\theta\big(z_iz_j;q^2\big)}. \notag
\end{gather*}
By the functional equation~\eqref{Eq_Gamma-fun} for the elliptic
gamma function, this leads to
\begin{gather*}
\mathscr{L} =
\frac{2^n\kappa_n\big(p^4,q^2\big) \Gamma^n_{p^4,q^2}\big(t^2\big)}
{S_n\big(qt^{1-2n}/s,p^2qt^{1-2n}/s,sp,s/p,tv,p^2t/v;t^2;p^4,q^2\big)} \notag \\
\hphantom{\mathscr{L} =}{} \times\int
R^{\ast}_{\mu}\big(z;s/p,qt^{1-2n}/s;q^2,t^2;p^2\big)
\prod_{i=1}^n \theta\big(pqz_i/v;q^2\big) \\
\hphantom{\mathscr{L} =\times\int}{} \times
\prod_{i=1}^n \frac{\Gamma_{p^2,q^2}\big(s(pz_i)^{\pm},
pqt^{1-2n}(pz_i)^{\pm}/s\big)
\Gamma_{p^4,q^2}\big(ptv(pz_i)^{\pm},p^2tz_i^{\pm}/v\big)}
{\Gamma_{p^4,q^2}\big(p^2\big(p^2z_i^2\big)^{\pm}\big)} \\
\hphantom{\mathscr{L} =\times\int}{} \times \prod_{1\leq i<j\leq n}
\frac{\Gamma_{p^4,q^2}\big(p^2t^2\big(p^2z_iz_j\big)^{\pm},t^2(z_i/z_j)^{\pm}\big)}
{\Gamma_{p^4,q^2}\big(p^2\big(p^2z_iz_j\big)^{\pm},(z_i/z_j)^{\pm}\big)}
 \frac{\dup z_1}{z_1}\cdots\frac{\dup z_n}{z_n}.
\end{gather*}
We can scale the contour $C$ by a factor~$1/p$ without crossing any poles.
Then replacing $z_i\mapsto z_i/p$ (so that the contour is once again
given by $C$) and using~\eqref{Eq_Rsqrtp} with
\begin{gather*}
(x,a,b,p,q,t)\mapsto \big(z/p,s/p,qt^{1-2n}/s,p^2,q^2,t^2\big),
\end{gather*}
this yields
\begin{gather*}
\mathscr{L} = \frac{2^n q^{4n(\mu')} t^{-4n(\mu)} \big(s^2qt^{4n-3}/p\big)^{\abs{\mu}}
\kappa_n\big(p^4,q^2\big) \Gamma^n_{p^4,q^2}\big(t^2\big)}
{S_n\big(qt^{1-2n}/s,p^2qt^{1-2n}/s,sp,s/p,tv,p^2t/v;t^2;p^4,q^2\big)} \\
\hphantom{\mathscr{L} =}{} \times \int
R^{\ast}_{\mu}\big(z;s,pqt^{1-2n}/s;q^2,t^2;p^2\big)
\prod_{i=1}^n \theta\big(qz_i/v;q^2\big) \\
\hphantom{\mathscr{L} =\times \int}{} \times
\prod_{i=1}^n
\frac{\Gamma_{p^2,q^2}\big(sz_i^{\pm},pqt^{1-2n}z_i^{\pm}/s\big)
\Gamma_{p^4,q^2}\big(ptvz_i^{\pm},ptz_i/v,p^3tz_i^{-1}/v\big)}
{\Gamma_{p^4,q^2}\big(p^2z_i^{\pm 2}\big)} \notag \\
\hphantom{\mathscr{L} =\times \int}{} \times \prod_{1\leq i<j\leq n}
\frac{\Gamma_{p^4,q^2}\big(p^2t^2(z_iz_j)^{\pm},t^2(z_i/z_j)^{\pm}\big)}
{\Gamma_{p^4,q^2}\big(p^2(z_iz_j\big)^{\pm},(z_i/z_j)^{\pm})}
\frac{\dup z_1}{z_1}\cdots\frac{\dup z_n}{z_n}.
\end{gather*}
Taking the $p\to 0$ limit using \eqref{Eq_Rast-limit} and
\begin{gather*}
\lim_{p\to 0} \Gamma(p^{\alpha}z;p,q)=
\begin{cases}
1/(z;q)_{\infty} & \text{if $\alpha=0$}, \\
1 & \text{if $0<\alpha<1$}, \\
(q/z;q)_{\infty} & \text{if $\alpha=1$},
\end{cases}
\end{gather*}
we obtain
\begin{gather}
\lim_{p\to 0} \left(\frac{p}{qt}\right)^{\abs{\mu}}\mathscr{L} =
\big({-}st^{2n-2}\big)^{3\abs{\mu}} q^{6n(\mu')} t^{-8n(\mu)}
\prod_{i=1}^n \frac{\big(q^2t^{2i-2},s^2t^{2i-2};q^2\big)_{\infty}}
{\big(sqt^{2i-2}v^{\pm};q^2\big)_{\infty}} \notag \\
\hphantom{\lim_{p\to 0} \left(\frac{p}{qt}\right)^{\abs{\mu}}\mathscr{L} =}{} \times
\frac{C^{-}_{\mu}\big(t^2;q^2,t^2\big)}{C^0_{\mu}\big(t^{2n};q^2,t^2\big)}
\bigg\langle \bar{P}^{\ast}_{\mu}\big(z;q^2,t^2,s^2\big)
\prod_{i=1}^n \frac{\theta\big(qz_i/v;q^2\big)}{\big(sz_i^{\pm};q^2\big)_{\infty}}
\bigg\rangle_{q^2,t^2}^{(n)},\label{Eq_Llim}
\end{gather}
where we have also used the evaluation \eqref{Eq_Snqt}.

Taking the same limit in the right-hand side of \eqref{Eq_int-sub} yields
\begin{gather}\label{Eq_Rlim}
\lim_{p\to 0} \left(\frac{p}{qt}\right)^{\abs{\mu}} \mathscr{R}=
\big(st^{2n-2}\big)^{2\abs{\mu}}q^{4n(\mu')}t^{-4n(\mu)}
\frac{C^0_{\mu}\big(s^2t^{2n-2};q^2,t^2\big) C^{+}_{\mu}\big(s^2t^{4n-4};q^2,t^2\big)}
{C^0_{\mu}\big(sqt^{2n-2}v^{\pm};q^2,t^2\big)}.
\end{gather}
Equating the right-hand sides of~\eqref{Eq_Llim} and~\eqref{Eq_Rlim},
and replacing $v\mapsto q/v$ results in~\eqref{Eq_prop-limit}
with $(q,t)\mapsto\big(q^2,t^2\big)$, completing the proof.
\end{proof}

\section[The elliptic hypergeometric function Phi lambda]{The elliptic hypergeometric function $\boldsymbol{\Phi_{\la}}$}\label{Sec_Phi}

In this section we define a new elliptic hypergeometric function,
\begin{gather*}
\Phi_{\la}=\Phi_{\la}(q,t;p)=\Phi_{\la}(a;b,c,d;q,t;p),
\end{gather*}
study its symmetries and prove two summation formulas for one-parameter
specialisations of $\{a,b,c,d\}$.
The $p\to 0$ limit of $\Phi_{\la}(q,t;p)$ will play an important role
in proving the $q,t$-branching rules~\eqref{Eq_universal} and~\eqref{Eq_non-universal}.

For $\la$ a partition and $a,b,c,d,p,q,t\in\mathbb{C}^{\ast}$ such that
$\abs{p}<1$, the elliptic hypergeometric function $\Phi_{\la}(q,t;p)$ is defined as
\begin{gather}
\Phi_{\la}(a;b,c,d;q,t;p) \notag \\
\qquad{} := \frac{1}{\Delta^0_{\la}\big(e\vert f;q^2,t^2;p^2\big)}
\sum_{\mu\subseteq\la}
\frac{C^{-}_{\mu}\big(pqt;q^2,t^2;p^2\big)C_{\mu}^{+}\big(a^2p^2/q^2;q^2,t^2;p^2\big)}
{C^0_{\mu}\big(ap,bp,cp,dp;q^2,t^2;p\big)}
\obinomE{\la}{\mu}_{[e,f];q^2,t^2;p^2} \notag \\
\qquad \hphantom{:}{}=\frac{1}{\Delta^0_{\la}\big(e\vert f;q^2,t^2;p^2\big)}
\sum_{\mu\subseteq\la}\bigg(
\big({-}eq^2\big)^{\abs{\mu}} q^{2n(\mu')}t^{-2n(\mu)} \notag \\
 \qquad \qquad\qquad{} \times
\frac{C^{-}_{\mu}\big(pqt;q^2,t^2;p^2\big)C_{\mu}^{+}\big(a^2/q^2;q^2,t^2;p^2\big)}
{C^0_{\mu}\big(a,b,c,d;q^2,t^2;p\big)}
\obinomE{\la}{\mu}_{[e,f];q^2,t^2;p^2} \bigg),\label{Eq_PhiE}
\end{gather}
where $e:=bcd/aq^2$ and $f:=bcdq/a^3pt$.
The equality of the two expressions on the right of~\eqref{Eq_PhiE}
is a direct consequence of the quasi-periodicity~\eqref{Eq_quasi-C}
of the elliptic $C$-symbols.
We also note that $\Phi_{\la}(a;-a,c,d;q,t;p)$ is a function of
$a^2$ only, so that $\Phi_{\la}(a;-a,c,d;q,t;p)=\Phi_{\la}(-a;a,c,d;q,t;p)$.

For later use we also define the following basic hypergeometric analogue
of $\Phi_{\la}(q,t;p)$:
\begin{gather}\label{Eq_PhiB}
\Phi_{\la}(q,t)=\Phi_{\la}(a;b,c,d;q,t):=
\sum_{\mu\subseteq\la}
(-1)^{\abs{\mu}}q^{-n(\mu')}t^{n(\mu)}
\frac{C_{\mu}^{+}\big(a^2/q,s^2;q,t\big)}
{C^0_{\mu}(a,b,c,d;q,t)} \qbin{\la}{\mu}_{q,t,s},
\end{gather}
where $s^2:=bcd/aq$.
The reason for renaming $e$ in \eqref{Eq_PhiE} as $s^2$
is the convention of writing $\qbin{\la}{\mu}_{q,t,s}$.
Here we recall that the $q,t,s$-binomial coefficient
is a~function of~$s^2$ only.
For~$\la$ a partition of length at most one
$\Phi_{\la}(q,t)$ is independent of $t$ and
simplifies to a balanced, terminating $\qhypc{5}{4}$
basic hypergeometric series~\cite{GR04}:
\begin{gather}\label{Eq_5ph4}
{\Phi_{(N)}}(a;b,c,d;q,t)=
\qhypc{5}{4}\left[\genfrac{}{}{0pt}{}
{aq^{-1/2},-aq^{-1/2},-a,bcdq^{N-1}/a,q^{-N}}
{a^2q^{-1},b,c,d};q,q\right].
\end{gather}

\begin{Lemma}\label{Lem_PhiE-limit}
We have
\begin{gather*}
\lim_{p\to 0} \Phi_{\la}(a;b,c,d;q,t;p)=\Phi_{\la}\big(a;b,c,d;q^2,t^2\big).
\end{gather*}
\end{Lemma}

\begin{proof}By \eqref{Eq_quasi-C0},
\begin{gather}
\Delta^0_{\la}(e\vert f;q,t;p)=
(eq)^{\abs{\la}} q^{2n(\la)}t^{-2n(\la)}
\Delta^0_{\la}(e\vert fp;q,t;p). \notag
\end{gather}
Replacing $(p,q,t)\mapsto \big(p^2,q^2,t^2\big)$ and
substituting the above in the second form for
$\Phi$, the $p\to 0$ limit can be taken using~\eqref{Eq_qtbinom-limit} and
$\lim_{p\to 0} C_{\la}^{0,\pm}\big(zp;q^2,t^2;p^2\big)=1$,
resulting in the claim.
\end{proof}

The function $\Phi_{\la}$ satisfies the following symmetries.

\begin{Lemma}\label{Eq_PhiE-sym}For $\la$ a partition
\begin{subequations}
\begin{gather}
\Phi_{\la}(1/a;1/b,1/c,1/d;1/q,1/t;p)=
\Phi_{\la}(a;b,c,d;q,t;p)
\end{gather}
and
\begin{gather}\label{Eq_PhiE-conj}
\Phi_{\la'}(a;b,c,d;t,q;p)=
\Phi_{\la}(1/a;1/b,1/c,1/d;q,t;p).
\end{gather}
\end{subequations}
\end{Lemma}

Combining these two results further yields
\begin{gather}\label{Eq_PhiE-conj2}
\Phi_{\la}(a;b,c,d;1/q,1/t;p)=\Phi_{\la'}(a;b,c,d;t,q;p).
\end{gather}

\begin{proof}Throughout the proof $e$ and $f$ are fixed as
$e=bcd/aq^2$ and $f=bcdq/a^3pt$.

By \eqref{Eq_recip} and \eqref{Eq_binomE-randc},
\begin{gather*}
 \Phi_{\la}(1/a;1/b,1/c,1/d;1/q,1/t;p) \notag \\
\qquad=\frac{1}{\Delta^0_{\la}\big(1/e\vert 1/fp^2;1/q^2,1/t^2;p^2\big)}
\sum_{\mu\subseteq\la}
\bigg(
({-}eq^2)^{-\abs{\mu}} q^{-2n(\mu')}t^{2n(\mu)} \notag \\
 \qquad\quad{}\times
\frac{C^{-}_{\mu}\big(p/qt;1/q^2,1/t^2;p^2\big)C_{\mu}^{+}\big(q^2/a^2;1/q^2,1/t^2;p^2\big)}
{C^0_{\mu}\big(1/a,1/b,1/c,1/d;1/q^2,1/t^2;p\big)}
\obinomE{\la}{\mu}_{[1/e,1/fp^2];1/q^2,1/t^2;p^2} \bigg) \\
\qquad= \left(\frac{f^2p^6}{e^3q^6}\right)^{\abs{\la}}
\frac{q^{-8n(\la')}t^{8n(\la)}}{\Delta^0_{\la}\big(eq^2/p^4\vert fp^2;q^2,t^2;p^2\big)}
 \\
 \qquad\quad{}\times
\sum_{\mu\subseteq\la} (epq/t)^{\abs{\mu}} t^{-4n(\mu)}
\frac{C^{-}_{\mu}\big(qt/p;q^2,t^2;p^2\big)C_{\mu}^{+}\big(a^2/q^2;q^2,t^2;p^2\big)}
{C^0_{\mu}\big(a,b,c,d;q^2,t^2;p\big)}
\obinomE{\la}{\mu}_{[e,f];q^2,t^2;p^2}.
\end{gather*}
The first claim now follows by applying \eqref{Eq_quasi-C0} to
$\Delta_{\la}^0$ and \eqref{Eq_quasi-Cm} to $C^{-}_{\mu}$.

Similarly, by \eqref{Eq_conj} and \eqref{Eq_binomE-randc},
\begin{gather*}
 \Phi_{\la'}(a;b,c,d;t,q;p) \\
 \qquad{}=
\frac{1}{\Delta^0_{\la'}\big(eq^2/t^2\vert ft^2/q^2;t^2,q^2;p^2\big)}
\sum_{\mu\subseteq\la}\bigg(
({-}eq^2)^{\abs{\mu}}q^{-2n(\mu')}t^{2n(\mu)} \\
 \qquad \quad{}\times
\frac{C^{-}_{\mu'}\big(pqt;t^2,q^2;p^2\big)C_{\mu'}^{+}\big(a^2/t^2;t^2,q^2;p^2\big)}
{C^0_{\mu'}\big(a,b,c,d;t^2,q^2;p\big)}
\obinomE{\la'}{\mu'}_{[eq^2/t^2,ft^2/q^2];t^2,q^2;p^2} \bigg) \\
 \qquad=\left(\frac{e^3p^2q^{10}}{f^2t^4}\right)^{\abs{\la}}
\frac{q^{-8n(\la')}t^{8n(\la)}}
{\Delta^0_{\la}\big(1/ep^4q^4\vert q^2/ft^2;q^2,t^2;p^2\big)}
\sum_{\mu\subseteq\la} \bigg(
({-}eq^2)^{-\abs{\mu}} q^{2n(\mu')}t^{-2n(\mu)} \\
 \qquad \quad{}\times
\frac{C^{-}_{\mu}\big(pqt;q^2,t^2;p^2\big)C_{\mu'}^{+}\big(1/a^2q^2;q^2,t^2;p^2\big)}
{C^0_{\mu'}\big(1/a,1/b,1/c,1/d;q^2,t^2;p\big)}
\obinomE{\la}{\mu}_{[1/eq^4,q^2/fp^2t^2];q^2,t^2;p^2}\bigg).
\end{gather*}
By \eqref{Eq_quasi-C0} applied to $\Delta_{\la}^0$ this yields
\begin{gather*}
 \Phi_{\la'}(a;b,c,d;t,q;p) \\
 \qquad{}=\frac{1}{\Delta^0_{\la}\big(1/eq^4\vert q^2/fp^2t^2;q^2,t^2;p^2\big)}
\sum_{\mu\subseteq\la}\bigg(
({-}eq^2)^{-\abs{\mu}} q^{2n(\mu')}t^{-2n(\mu)} \\
 \qquad \quad{} \times
\frac{C^{-}_{\mu}\big(pqt;q^2,t^2;p^2\big)C_{\mu'}^{+}\big(1/a^2q^2;q^2,t^2;p^2\big)}
{C^0_{\mu'}\big(1/a,1/b,1/c,1/d;q^2,t^2;p\big)}
\obinomE{\la}{\mu}_{[1/eq^4,q^2/fp^2t^2];q^2,t^2;p^2}\bigg) \\
 \qquad=\Phi_{\la}(1/a;1/b,1/c,1/d;q,t;p).
\end{gather*}
For the final equality we note that
if for $e=e(a,b,c,d;q)$ and $f=f(a,b,c,d;q,t;p)$
we define $\hat{e}:=e(1/a,1/b,1/c,1/d;q)$ and
$\hat{f}:=f(1/a,1/b,1/c,1/d;q,t;p)$,
then $eq^2=1/\hat{e}q^2$ (so that $1/eq^4=\hat{e}$) and
$q^2/fp^2t^2=\hat{f}$.
\end{proof}

The elliptic hypergeometric series \eqref{Eq_PhiE} is balanced but
not, generally, (very-)well-poised.
For example, as follows from \eqref{Eq_binomE_rectangle}, the one-row
case of $\Phi_{\la}$ is given by
\begin{gather*}
\Phi_{(N)}(a;b,c,d;q,t;p) =\sum_{k=0}^N \bigg(
\frac{\big(ep^2q^2/f;q^2,p^2\big)_{2k}}{\big(e/f;q^2,p^2\big)_{2k}}
\frac{\big(a^2p^2,a^2p^2/q^2;q^4,p^2\big)_k}{\big(ap,bp,cp,dp;q^2,p\big)_k} \\
\hphantom{\Phi_{(N)}(a;b,c,d;q,t;p) =\sum_{k=0}^N}{} \times
\frac{\big(e/f,pqt,eq^{2N},q^{-2N};q^2,p^2\big)_k}
{\big(p^2q^2,a^2p^2/q^2,p^2q^{2-2N}/f,ep^2q^{2N+2}/f;q^2,p^2\big)_k}\bigg),
\end{gather*}
generalising \eqref{Eq_5ph4}.
Since
\begin{gather}\label{Eq_frac}
\frac{\big(a^2p^2,a^2p^2/q^2;q^4,p^2\big)_k}{\big(ap,bp,cp,dp;q^2,p\big)_k}=
\frac{\big({-}ap,-ap^2,\pm ap/q,\pm ap^2/q;q^2,p^2\big)_k}
{\big(bp,bp^2,cp,cp^2,dp,dp^2;q^2,p^2\big)_k},
\end{gather}
$\Phi_{(N)}(a;b,c,d;q,t;p)$ is a balanced elliptic
hypergeometric series of the form~\eqref{Eq_EHS} for $r=9$.
The ratio of elliptic shifted factorials~\eqref{Eq_frac}
is however not well-poised, and there are exactly
eight choices for $b$, $c$, $d$ for which it is:
$b\in\{-a,-at/q\}$ and $c$, $d$ one of
\begin{subequations}
\begin{gather}
c=-d=at, \label{Case1} \\
c=-d=a/q, \label{Case2} \\
c=\sigma at, \qquad d=-\sigma a/q, \label{Case3}
\end{gather}
\end{subequations}
where $\sigma\in\{-1,1\}$.
By \eqref{Eq_PhiE-conj} only four of these are independent,
and irrespective of the choice of~$b$, \eqref{Case1} and~\eqref{Case2}
as well as the two choices for $\sigma$ in~\eqref{Case3} are related by conjugation.

For two of these independent choices we have a closed-form evaluation,
generalising \eqref{Eq_W110} and \eqref{Eq_W14}.
Our first result is a generalisation of \eqref{Eq_W110},
which is recovered for $\la=(N)$ after making the
substitution $(a,t)\mapsto\big(aq^4/bp^2,bp/q\big)$ followed by
$\big(p^2,q^2\big)\mapsto (p,q)$.

\begin{Theorem}\label{Thm_Littlewood}For $\la$ a partition,
\begin{gather}\label{Eq_Gen-W110}
\Phi_{\la}\big(a^{1/2};-a^{1/2},a^{1/2}t,-a^{1/2}t;q,t;p\big)
 = \begin{cases}\displaystyle
\frac{\Delta_{\mu}\big(at^2/q^2\vert\leeg;q^4,t^2;p^2\big)}
{\Delta_{2\mu}\big(at^2/q^2\vert qt/p;q^2,t^2;p^2\big)}
& \text{if $\la=2\mu$}, \\
0 & \text{otherwise}.
\end{cases}
\end{gather}
\end{Theorem}

\begin{proof}
By the inversion relation \eqref{Eq_inversion}, the quadratic summation
of Corollary~\ref{Cor_quadratic-sum} implies
\begin{gather}
 \sum_{\bmu} \frac{C^{-}_{\bmu}\big(pqt;t^2;p^2,q^2\big)
C^{+}_{\bmu}\big(ap^2/t^2;t^2;p^2,q^2\big)}
{C^0_{(2,1)\bmu^2}\big(ap^2q^2;t^2;p^2,q^4\big)}
\obinomE{\bla}{\bmu}_{[a,qt/p];t^2;p^2,q^2} \nonumber\\
 \qquad{} = \begin{cases}\displaystyle
\frac{\Delta_{\bnu}\big(a\vert\leeg;t^2;p^2,q^4\big)}
{\Delta_{(1,2)\bnu}\big(a\vert\leeg;t^2;p^2,q^2\big)} &
\text{if $\bla=(1,2)\bnu$}, \\
0 & \text{otherwise}.
\end{cases}
\label{Eq_afterinversion}
\end{gather}
Choosing $\bla=(0,\la)$ and noting that
\begin{gather*}
C^0_{\mu^2}\big(ap^2q^2;q^4,t^2;p^2\big)=
C^0_{\mu}\big(\pm a^{1/2}pq,\pm a^{1/2}pq/t;q^2,t^2;p\big),
\end{gather*}
leads to
\begin{gather*}
 \sum_{\mu}\frac{C^{-}_{\mu}\big(pqt;q^2,t^2;p^2\big)
C^{+}_{\mu}\big(ap^2/t^2;q^2,t^2;p^2\big)}
{C^0_{\mu}\big(\pm a^{1/2}pq,\pm a^{1/2}pq/t;q^2,t^2;p\big)}
\obinomE{\la}{\mu}_{[a,qt/p];q^2,t^2;p^2} \\
 \qquad{} = \begin{cases}\displaystyle
\frac{\Delta_{\nu}\big(a\vert\leeg;q^4,t^2;p^2\big)}
{\Delta_{2\nu}\big(a\vert\leeg;q^2,t^2;p^2\big)} &
\text{if $\la=2\nu$}, \\
0 & \text{otherwise}.
\end{cases}
\end{gather*}
Dividing both sides by $\Delta^0_{\la}\big(a\vert qt/p;q^2,t^2;p^2\big)$
and replacing $a\mapsto at^2/q^2$ results in~\eqref{Eq_Gen-W110}.
\end{proof}

The $p\to 0$ limit of Theorem~\ref{Thm_Littlewood} yields a summation
generalising Andrews' terminating $q$-analogue of Watson's $_3F_2$ summation
\cite[Theorem 1]{Andrews76}
\begin{gather}\label{Eq_Andrews}
\qhypc{4}{3}\left[\genfrac{}{}{0pt}{}
{a^{1/2},-a^{1/2},bq^{N-1},q^{-N}}{a,b^{1/2},-b^{1/2}};q,q\right]
=\begin{cases} \displaystyle
\frac{a^{N/2} \big(q,b/a;q^2\big)_{N/2}}
{\big(aq,b;q^2\big)_{N/2}} & \text{if $N$ is even}, \\
0 & \text{otherwise}.
\end{cases}
\end{gather}

\begin{Corollary}\label{Cor_Littlewood-basic}
For $\la$ a partition,
\begin{gather}
 \Phi_{\la}\big(a^{1/2};-a^{1/2},(at)^{1/2},-(at)^{1/2};q,t\big) \nonumber\\
 \qquad{}=\begin{cases}\displaystyle
(a/q)^{\abs{\mu}} t^{-2n(\mu)}
\frac{C^{-}_{\mu}\big(q,qt;q^2,t\big)C^{+}_{\mu}\big(a,at;q^2,t\big)}
{C^0_{2\mu^2}\big(at;q^2,t\big)}
& \text{if $\la=2\mu$}, \\
0 & \text{otherwise}.
\end{cases}\label{Eq_Littlewood-basic}
\end{gather}
\end{Corollary}

Andrews' summation is obtained by taking $\la=(N)$ and
replacing $(a,t)\mapsto (aq,b/aq)$.
In Conjecture~\ref{Con_Phi-conj} below we give a second multi-sum
generalisation of~\eqref{Eq_Andrews}.

\begin{proof}
In the following we consider the right-hand side of~\eqref{Eq_Gen-W110}
with $\big(p^2,q^2,t^2\big)\mapsto (p,q,t)$.

By \eqref{Eq_quasi-C},
\begin{gather*}
\Delta_{\la}(a\vert\leeg;q,t,p)=
\left(\frac{t}{a^3q^3}\right)^{\abs{\la}} q^{-4n(\la')} t^{6n(\la)}
\frac{C^0_{2\la^2}(aq;q,t;p)}{C^{-}_{\la}(q,t;q,t;p)C^{+}_{\la}(a,aq/t;q,t;p)}.
\end{gather*}
This implies
\begin{gather*}
\lim_{p\to 0}
\frac{\Delta_{\mu}\big(at/q\vert\leeg;q^2,t,p\big)}
{\Delta_{2\mu}\big(at/q\vert\leeg;q,t,p\big)} \\
\qquad{} =\big(a^3qt^2\big)^{\abs{\mu}} q^{8n(\mu')} t^{-6n(\mu)}
\frac{C^0_{2\mu^2}\big(aqt;q^2,t\big)C^{-}_{2\mu}(q,t;q,t)C^{+}_{2\mu}(a,at/q;q,t)}
{C^0_{4\mu^2}(at;q,t)C^{-}_{\mu}\big(q^2,t;q^2,t\big)C^{+}_{\mu}\big(aq,at/q;q^2,t\big)}
 \\
\qquad{}=\big(a^3qt^2\big)^{\abs{\mu}} q^{8n(\mu')} t^{-6n(\mu)}
\frac{C^{-}_{\mu}\big(q,qt;q^2,t\big)C^{+}_{\mu}\big(a,at;q^2,t\big)}{C^0_{2\mu^2}\big(at;q^2,t\big)},
\end{gather*}
where the last line follows from~\eqref{Eq_double}. Also
\begin{gather*}
\frac{1}{\Delta_{2\mu}^0(at/q\vert b;q,t;p)}=
(aqt)^{-2\abs{\mu}}q^{-8n(\mu')}t^{4n(\mu)}
\frac{1}{\Delta_{2\mu}^0(at/q\vert bp;q,t;p)},
\end{gather*}
and thus
\begin{gather*}
\lim_{p\to 0}\frac{1}{\Delta_{2\mu}^0(at/q\vert (qt/p)^{1/2};q,t;p)}
=(aqt)^{-2\abs{\mu}}q^{-8n(\mu')}t^{4n(\mu)}.
\end{gather*}
The identity \eqref{Eq_Littlewood-basic} now follows from Lemma~\ref{Lem_PhiE-limit}.
\end{proof}

Before stating our second summation formula for $\Phi_{\la}$
we remark that if instead of specialising $\bla=(0,\la)$ in~\eqref{Eq_afterinversion} we take $\bla=(\la,0)$ and then
swap $p$ and $q$, we obtain the following higher-dimensional
analogue of~\eqref{Eq_unknown}:
\begin{gather*}
\sum_{\mu} \frac{C^{-}_{\mu}\big(pqt;q^2,t^2;p^2\big)
C^{+}_{\mu}\big(aq^2/t^2;q^2,t^2;p^2\big)}
{C^0_{2\mu^2}\big(ap^2q^2;q^2,t^2;p^4\big)}
\obinomE{\la}{\mu}_{[a,pt/q];q^2,t^2;p^2}
=\frac{\Delta_{\la}\big(a\vert\leeg;q^2,t^2;p^4\big)}
{\Delta_{\la}\big(a\vert\leeg;q^2,t^2;p^2\big)}.
\end{gather*}

Our second theorem generalises \eqref{Eq_W14}, obtained by
taking $\la=(N)$ and replacing $(a,b)\mapsto \big({-}a/bp,b^2pq\big)$.

\begin{Theorem}\label{Thm_Kawanaka}
For $\la$ a partition,
\begin{gather}\label{Eq_Gen-W14}
\Phi_{\la}\big(aq;a,-aqt,-at;q^2,t^2;p^2\big)
=\frac{\Delta_{\la}(at/q\vert\leeg;q,-t;p)}
{\Delta_{\la}\big(a^2t^2/q^2\vert t/pq;q^2,t^2;p^2\big)}.
\end{gather}
\end{Theorem}

\begin{proof}
The proof is basically the same as that of Theorem~\ref{Thm_Littlewood}.
Inverting \eqref{Eq_Kawanaka-sum} using \eqref{Eq_inversion}, we have
\begin{gather*}
\sum_{\bmu} \frac{C^{-}_{\bmu}\big(pqt;t^2;p^2,q^2\big)
C^{+}_{\bmu}\big(a^2p^2q^2/t^2;t^2;p^2,q^2\big)}
{C^0_{2\bmu^2}(-apq;-t;p,q)}
\obinomE{\bla}{\bmu}_{[a^2,t/pq];t^2;p^2,q^2}
=\frac{\Delta_{\bla}(a\vert\leeg;-t;p,q)}
{\Delta_{\bla}\big(a^2\vert\leeg;t^2;p^2,q^2\big)}.
\end{gather*}
Taking $\bla=(0,\la)$, noting that by \eqref{Eq_double-square}
\begin{gather*}
C^0_{2\mu^2}(-apq;q,-t;p)=
C^0_{\mu}\big({-}apq,-apq^2,apq/t,apq^2/t;q^2,t^2;p\big),
\end{gather*}
and finally dividing both sides by $\Delta^0_{\la}\big(a^2\vert t/pq;q^2,t^2;p^2\big)$,
we obtain~\eqref{Eq_Gen-W14} with $a\mapsto aq/t$.
\end{proof}

The $p\to 0$ limit of~\eqref{Eq_Gen-W14} yields a generalisation of
the following quadratic summation due to Bressoud, Ismail and Stanton
\cite[equation~(2.1)]{BIS00}:
\begin{gather}\label{Eq_BIS}
\qhypc{4}{3}\left[\genfrac{}{}{0pt}{}
{a,aq,b^2q^{2N-2},q^{-2N}}
{a^2,b,bq};q^2,q^2\right]
=\frac{a^N (1-b/q) (-q,b/a;q)_N}{\big(1-bq^{2N-1}\big)(-a,b/q;q)_N}.
\end{gather}

\begin{Corollary}\label{Cor_Kawanaka-basic}
For $\la$ a partition,
\begin{gather}\label{Eq_Gen-BIS}
\Phi_{\la}\big(aq;a,aqt,at;q^2,t^2\big)
=(-a)^{\abs{\la}} t^{-2n(\la)}
\frac{C^{-}_{\la}(-q,-t;q,t)C^{+}_{\la}(a,at/q;q,t)}
{C^0_{2\la^2}(at;q,t)}.
\end{gather}
\end{Corollary}

We note that for $\la=(N)$ and $(a,t)\mapsto (-a,b/a)$ this is~\eqref{Eq_BIS} and that in going from~\eqref{Eq_Gen-W14} to
\eqref{Eq_Gen-BIS} the parameter $t$ has been replaced by~$-t$.

The proof of Corollary~\ref{Cor_Kawanaka-basic} proceeds along the
same lines as the proof of Corollary~\ref{Cor_Littlewood-basic}, and
we omit the details.

\section{Proof of Theorem~\ref{Thm_universal}}

In Section~\ref{Sec_branching} we stated
\eqref{Eq_non-universal} as a corollary of
\eqref{Eq_universal}, but in fact both results
are equivalent, and proving \eqref{Eq_non-universal}
for fixed $m$, $r$ and all $n\geq r$ is the same as
proving \eqref{Eq_universal} for fixed~$m$,~$r$.
To avoid the use of virtual Koornwinder polynomials in our
proof, we will in the following establish
Corollary~\ref{Cor_non-universal} instead of Theorem~\ref{Thm_universal}.

The first step in our proof is to dualise the three claims of
Corollary~\ref{Cor_non-universal}, an approach that was also
utilised in~\cite{RW15} to prove bounded Littlewood identities
for Macdonald polynomials.

Let $x:=(x_1,\dots,x_n)$ and $y:=(y_1,\dots,y_m)$.
By the complementation symmetry~\eqref{Eq_P-comp} and
homogeneity of the Macdonald polynomials,
the (dual) Cauchy identity~\eqref{Eq_Cauchy-Mac}
can be written in the form (see also~\cite{Okounkov98})
\begin{gather*}
\sum_{\mu\subseteq (m^n)} (-1)^{\abs{\mu}}
P_{(m^n)-\mu}(x;q,t) P_{\mu'}(y;t,q)=
\prod_{i=1}^n \prod_{j=1}^m (x_i-y_j).
\end{gather*}
Replacing $n\mapsto 2n$ and then specialising $x_{i+n}=x_i^{-1}$ for
all $1\leq i\leq n$ yields
\begin{gather*}
\sum_{\mu\subseteq (m^{2n})} (-1)^{\abs{\mu}}
P_{(m^{2n})-\mu}\big(x^{\pm};q,t\big) P_{\mu'}(y;t,q)
=\prod_{i=1}^n \prod_{j=1}^m x_i^{-1} (x_i-y_j)(1-x_iy_j).
\end{gather*}
Up to the simple factor $(-1)^{mn} (y_1\cdots y_m)^n$
the right-hand side coincides with the right-hand side of the
Cauchy identity \eqref{Eq_Mimachi} for Koornwinder polynomials.
Correcting for this factor, we can thus equate the respective
left-hand sides, resulting in
\begin{gather}
 \sum_{\mu\subseteq (m^{2n})} (-1)^{\abs{\mu}}
P_{(m^{2n})-\mu}\big(x^{\pm};q,t\big) P_{\mu'}(y;t,q) \notag \\
 \qquad{} =\sum_{\nu\subseteq (m^n)} (-1)^{mn+\abs{\nu}}
K_{(m^n)-\nu}(x;q,t;\bart) (y_1\cdots y_m)^n K_{\nu'}(y;t,q;\bart).\label{Eq_PP-KK}
\end{gather}

Let $r$ be an integer such that $0\leq r\leq n$,
$\la$ a partition contained in $(m^r)$, and $s:=2n-r\geq n$.
Extracting the coefficient of
\begin{gather*}
K_{(m^r)-\la}(x;q,t;\bart)P_{(s^m)}(y;t,q)
\end{gather*}
in \eqref{Eq_PP-KK} picks out the term $\mu=(m^s)$ in the sum on the
left and
\begin{gather*}
\nu=(\underbrace{m,\dots,m}_{n-r \text{ times}},\la_1,\dots,\la_r)
\end{gather*}
in the sum on the right.
For such $\mu$ and $\nu$,
\begin{gather*}
\big(m^{2n}\big)-\mu=(m^r) \qquad \text{and} \qquad \nu'=\la'+\big(N^m\big),
\end{gather*}
where $N:=n-r$.
Hence
\begin{gather}
 \big[K_{(m^r)-\la}(x;q,t;\bart)\big] P_{(m^r)}\big(x^{\pm};q,t\big) \notag \\
 \qquad{} =(-1)^{\abs{\la}} \big[P_{(s^m)}(y;t,q)\big]
(y_1\cdots y_m)^n K_{\la'+(N^m)}(y;t,q;\bart),\label{Eq_KPPyK}
\end{gather}
where the reader is warned that in the above right-hand side we
have not followed our earlier convention, and $\big[P_{\la}(y;t,q)\big]f(y)$
denotes the coefficient of $P_{\la}(t,q)$ of~$f\in\Lambda_{\mathrm{GL}(m)}$.

Recall that in \eqref{Eq_gln-weights} we extended the Macdonald
polynomials to arbitrary weights.
Accordingly, for $k\in\mathbb{Z}$ and $\mu\in P(m)$,
\begin{gather*}
P_{\mu}(y;q,t)=(y_1\cdots y_m)^k P_{(\mu_1-k,\dots,\mu_m-k)}(y;q,t).
\end{gather*}
This implies that if $f$ is an $S_m$-symmetric Laurent polynomial in~$y$, then
\begin{gather*}
\big[P_{\mu}(y;t,q)\big] (y_1\cdots y_m)^k f(y)= \big[P_{(\mu_1-k,\dots,\mu_m-k)}(y;t,q)\big] f(y).
\end{gather*}
Equation \eqref{Eq_KPPyK} therefore simplifies to
\begin{gather*}
\big[K_{(m^r)-\la}(x;q,t;\bart)\big] P_{(m^r)}\big(x^{\pm};q,t\big)
=(-1)^{\abs{\la}} \big[P_{(N^m)}(y;t,q)\big] K_{\la'+(N^m)}(y;t,q;\bart),
\end{gather*}
where we recall that $N:=n-r\geq 0$.
Comparing this with~\eqref{Eq_non-universal} (where
$P^{(\mathrm{B}_n,\mathrm{C}_n)}(q,t,s)$ and
$P^{(\mathrm{C}_n,\mathrm{C}_n)}(q,t,s)$ are given by~\eqref{Eq_PCB-PCC}) it follows that we must prove three
identities for
\begin{gather*}
\big[P_{(N^m)}(y;t,q)\big] K_{\la'+(N^m)}(y;t,q;\bart).
\end{gather*}
(Of course, by duality we only need to actually prove two identities.)
To state these in the simplest possible form we first
define, in analogy with~\eqref{Eq_c-small},
\begin{gather}\label{Eq_d-small}
\dd_{\la\mu}^{(n)}(q,t;\bart):=\big[P_{\la}(x;q,t)\big]
K_{\mu}(x;q,t;\bart)\in \mathbb{Q}(q,t,t_0,t_1,t_2,t_3)
\end{gather}
for $\la\in P(n)$, $\mu\in P_{+}(n)$.
By~\eqref{Eq_K-symmetry}, \eqref{Eq_Pf} and
the homogeneity of the Macdonald polynomials,
these coefficients satisfy the symmetries
\begin{gather}\label{Eq_d-symm}
\dd_{\la\mu}^{(n)}(q,t;\bart)=\dd_{\la\mu}^{(n)}(1/q,1/t;1/\bart)=
(-1)^{\abs{\la}-\abs{\mu}}\dd_{\la\mu}^{(n)}(q,t;-\bart)=
\dd_{(-\la_n,\dots,-\la_1),\mu}^{(n)}(q,t;\bart),
\end{gather}
where $\la=(\la_1,\dots,\la_n)$.

In the case of \eqref{Eq_CC} we need to show that
\begin{gather*}
\dd_{(N^m),\la'+(N^m)}^{(m)}
\big(t,q;t^{1/2},-t^{1/2},(qt)^{1/2},-(qt)^{1/2}\big)
=\begin{cases}
c_{\la}\big(q^{-m},t^{-N};q,t\big) & \text{if $\la'$ even}, \\
0 & \text{otherwise},
\end{cases}
\end{gather*}
where $N$ is an arbitrary nonnegative integer and $\la\in P_{+}(m)$.
Replacing
\begin{gather*}
(m,\la,y_1,\dots,y_m,q,t)\mapsto (n,\la',x_1,\dots,x_n,t,q)
\end{gather*}
and defining
\begin{gather*}
\hat{c}_{2\la}(w,z;q,t):=c_{(\la')^2}\big(z^{-1},w^{-1};t,q\big),
\end{gather*}
this can be rewritten as
\begin{gather}\label{Eq_PK-c}
\dd_{(N^n),\la+(N^n)}^{(n)}
\big(q,t;q^{1/2},-q^{1/2},(qt)^{1/2},-(qt)^{1/2}\big)
=\begin{cases}
\hat{c}_{\la}\big(q^N,t^n;q,t\big) & \text{if $\la$ even}, \\
0 & \text{otherwise},
\end{cases}
\end{gather}
where now $\la\in P_{+}(n)$.
By \eqref{Eq_c-def} and \eqref{Eq_conj},
\begin{gather*}
\hat{c}_{2\la}(w,z;q,t)=
\frac{C^0_{2\la}(z;q,t)}{C^0_{2\la}(qz/t;q,t)}
\frac{C^{-}_{\la}\big(q;q^2,t\big)}{C_{\la}^{-}\big(t;q^2,t\big)}
\frac{C^{+}_{\la}\big(q^2w^2z^2/t^2;q^2,t\big)}
{C_{\la}^{+}\big(qw^2z^2/t;q^2,t\big)}.
\end{gather*}
Similarly, the dual case of \eqref{Eq_K} translates to
\begin{gather}\label{Eq_PK-d}
\dd_{(N^n),\la+(N^n)}^{(n)}\big(q,t;1,-1,q^{1/2},-q^{1/2}\big)=
\begin{cases}
\hat{d}_{\la}\big(q^N,t^n;q,t\big) & \text{if $\la'$ even}, \\
0 & \text{otherwise},
\end{cases}
\end{gather}
with $N$ and $\la$ as above, and
\begin{gather*}
\hat{d}_{\la^2}(w,z;q,t) =d_{2\la'}\big(z^{-1},w^{-1};t,q\big)
 =\frac{C^0_{\la^2}(z;q,t)}{C^0_{\la^2}(qz/t;q,t)}
\frac{C^{-}_{\la}\big(qt;q,t^2\big)}{C_{\la}^{-}\big(t^2;q,t^2\big)}
\frac{C^{+}_{\la}\big(qw^2z^2/t^4;q,t^2\big)}
{C_{\la}^{+}\big(w^2z^2/t^3;q,t^2\big)}.
\end{gather*}
Finally, in the case of \eqref{Eq_BC} we get
\begin{gather}\label{Eq_PK-e}
\dd_{(N^n),\la+(N^n)}^{(n)}\big(q^2,t^2;-1,-q,-t,-qt\big)= \hat{e}_{\la}\big(q^N,t^n;q,t\big),
\end{gather}
where
\begin{align*}
\hat{e}_{\la}(w,z;q,t)&=(-1)^{\abs{\la}} e_{\la'}\big(z^{-1},w^{-1};t,q\big) \\
&=\frac{C^0_{\la}\big(z^2;q^2,t^2\big)}{C^0_{\la}\big(q^2z^2/t^2;q^2,t^2\big)}
\frac{C^{-}_{\la}(-q;q,t)}{C_{\la}^{-}(t;q,t)}
\frac{C^{+}_{\la}\big(qw^2z^2/t^2;q,t\big)}{C_{\la}^{+}\big({-}w^2z^2/t;q,t\big)}.
\end{align*}
By \eqref{Eq_d-symm} we also have the companion identity
\begin{gather*}
\dd_{(N^n),\la+(N^n)}^{(n)}\big(q^2,t^2;1,q,t,qt\big)=
(-1)^{\abs{\la}} \hat{e}_{\la}\big(q^N,t^n;q,t\big).
\end{gather*}

In the following it will be convenient to define
\begin{gather}
 \DD_{\la}^{(N,n)}(q,t;t_0\!:\!t_1,t_2,t_3) := \big(t_0q^Nt^{n-1}\big)^{\abs{\la}} t^{-n(\la)}
\dd_{(N^n),\la+(N^n)}^{(n)}(q,t;\bart) \notag \\
\qquad\qquad\qquad{}\times
\frac{C^0_{\la}\big(qt^{n-1};q,t\big)C_{\la}^{-}(t;q,t)
C_{\la}^{+}\big(t_0t_1t_2t_3q^{2N-1}t^{2n-2};q,t\big)}
{C^0_{\la}\big(t^n,q^{N+1}t^{n-1},t_0t_1q^Nt^{n-1},t_0t_2q^Nt^{n-1},
t_0t_3q^Nt^{n-1};q,t\big)},
\label{Eq_DD-dd}
\end{gather}
for $\la\in P_{+}(n)$.
The colon between $t_0$ and $t_1$, $t_2$, $t_3$ indicates the
absence of full $S_4$-symmetry. Combining~\eqref{Eq_binomial-formula} with Corollary~\ref{Cor_CC-vanish} then gives
\begin{gather*}
 \DD_{\la}^{(N,n)}(q,t;t_0\!:\!t_1,t_2,t_3) \\
 \qquad{}=\sum_{\mu\subseteq\la} \bigg(
(-1)^{\abs{\mu}} q^{-n(\mu')} t^{n(\mu)}
\frac{C^0_{\mu}\big(qt^{n-1};q,t\big)C_{\mu}^{-}(t;q,t)
C_{\mu}^{+}\big(t_0t_1t_2t_3q^{2N-1}t^{2n-2};q,t\big)}
{C^0_{\mu}\big(t^n,q^{N+1}t^{n-1},t_0t_1q^Nt^{n-1},
t_0t_2q^Nt^{n-1},t_0t_3q^Nt^{n-1};q,t\big)} \\
 \qquad\hphantom{=\sum_{\mu\subseteq\la}}{} \times
\qbin{\la}{\mu}_{q,t,(t_0t_1t_2t_3q^{2N-1}t^{2n-2})^{1/2}}
 \CC_{(N^n),\mu+(N^n)}^{(n)}(q,t,t_0)\bigg),
\end{gather*}
where we have also used \eqref{Eq_add-square} (with $p=0$) and~\eqref{Eq_binom-add-square}, as well as the definitions
\eqref{Eq_c-small} and~\eqref{Eq_C-large}.
Recalling that $\CC_{\la,\mu}^{(n)}(q,t,t_0)$ is a function of~$t_0^2$, it follows from \eqref{Eq_PbarPast-2} that
\begin{gather}
 \DD_{\la}^{(N,n)}
\big(q,t;-q^{1/2}\!:\!-q^{1/2}t_1,-q^{1/2}t_2,-q^{1/2}t_3\big) \notag \\
 \qquad{}=\sum_{\mu\subseteq\la} \bigg(
(-1)^{\abs{\mu}} q^{-n(\mu')} t^{n(\mu)}
\qbin{\la}{\mu}_{q,t,(t_1t_2t_3q^{2N+1}t^{2n-2})^{1/2}} \notag \\
 \qquad \hphantom{=\sum_{\mu\subseteq\la}}{} \times
\frac{C_{\mu}^{+}\big(q^{2N+1}t^{2n-2},t_1t_2t_3q^{2N-1}t^{2n-2};q,t\big)}
{C^0_{\mu}(q^{N+1}t^{n-1},t_1q^{N+1}t^{n-1},
t_2q^{N+1}t^{n-1},t_3q^{N+1}t^{n-1};q,t)} \bigg) \notag \\
 \qquad{}=\Phi_{\la}\big(q^{N+1}t^{n-1};t_1q^{N+1}t^{n-1},
t_2q^{N+1}t^{n-1},t_3q^{N+1}t^{n-1};q,t\big),\label{Eq_PK-Phi}
\end{gather}
with $\Phi_{\la}$ defined in \eqref{Eq_PhiB}.
If we specialise $\{t_1,t_2,t_3\}=\big\{{-}1,t^{1/2},-t^{1/2}\big\}$
then, by Corollary~\ref{Cor_Littlewood-basic},
\begin{gather*}
 \DD_{\la}^{(N,n)}
\big(q,t;-q^{1/2}\!:\!q^{1/2},-(qt)^{1/2},(qt)^{1/2}\big) \\
\qquad{} =\big(q^{2N+1}t^{2n-2}\big)^{\abs{\mu}} t^{-2n(\mu)}
\frac{C^{-}_{\mu}\big(q,qt;q^2,t\big)
C^{+}_{\mu}\big(q^{2N+2}t^{2n-2},q^{2N+2}t^{2n-1};q^2,t\big)}
{C^0_{2\mu^2}\big(q^{2N+2}t^{2n-1};q^2,t\big)}
\end{gather*}
if $\la=2\mu$ and $0$ otherwise.
By \eqref{Eq_DD-dd} the non-vanishing case finally gives
\begin{gather*}
 \dd_{(N^n),\la+(N^n)}^{(n)}
\big(q,t;q^{1/2},-q^{1/2},(qt)^{1/2},-(qt)^{1/2}\big) \\
 \qquad{} =\frac{C^0_{2\mu}\big(t^n;q,t\big)C^{-}_{\mu}\big(q;q^2,t\big)
C^{+}_{\mu}\big(q^{2N+2}t^{2n-2};q^2,t\big)}
{C^0_{2\mu}\big(qt^{n-1};q,t\big)C^{-}_{\mu}\big(t;q^2,t\big)
C^{+}_{\mu}\big(q^{2N+1}t^{2n-1};q^2,t\big)} \\
 \qquad=\hat{c}_{\la}\big(q^N,t^n;q,t\big),
\end{gather*}
where we have also used \eqref{Eq_double-square} to simplify the $C$-symbols in the second line.
This completes the proof of~\eqref{Eq_PK-c} and, by duality, that of~\eqref{Eq_PK-d}.

To prove the identity \eqref{Eq_PK-e}, we replace $(q,t)\mapsto\big(q^2,t^2\big)$
in~\eqref{Eq_PK-Phi} and then make the specialisation $\{t_1,t_2,t_3\} = \{1/q,t,t/q\}$. This gives
\begin{gather*}
 \DD_{\la}^{(N,n)}\big(q^2,t^2;-q\!:\!-1,-t,-qt\big) \\
 \qquad{}=\Phi_{\la}\big(q^{2N+2}t^{2n-2};q^{2N+1}t^{2n-2},q^{2N+2}t^{2n-1},q^{2N+1}t^{2n-1};q^2,t^2\big).
\end{gather*}
By Corollary~\ref{Cor_Kawanaka-basic} this series on the right once
again can be summed to yield
\begin{gather*}
 \DD_{\la}^{(N,n)}\big(q^2,t^2;-q\!:\!-1,-qt,-t\big) \\
 \qquad{}=\big({-}q^{2N+1}t^{2n-2}\big)^{\abs{\la}} t^{-2n(\la)}
\frac{C^{-}_{\la}(-q,-t;q,t)C^{+}_{\la}\big(q^{2N+1}t^{2n-2},q^{2N}t^{2n-1};q,t\big)}
{C^0_{2\la^2}\big(q^{2N+1}t^{2n-1};q,t\big)},
\end{gather*}
and hence by \eqref{Eq_DD-dd} and \eqref{Eq_double-square},
\begin{align*}
\dd_{(N^n),\la+(N^n)}^{(n)}\big(q^2,t^2;-1,-q,-t,-qt\big)
&=\frac{C^0_{\la}\big(t^{2n};q^2,t^2\big) C^{-}_{\la}(-q;q,t)
C^{+}_{\la}(a/q;q,t)}{C^0_{\la}\big(q^2t^{2n-2};q^2,t^2\big)
C_{\la}^{-}(t;q,t)C_{\la}^{+}\big({-}at/q^2;q,t\big)} \\
&=\hat{e}_{\la}\big(q^N,t^n;q,t\big).
\end{align*}

\section{Open problems}\label{Sec_Open}

The problem with proving the branching rule of
Conjecture~\ref{Con_selfdualcase} is that it requires the following
curious identity for the basic hypergeometric function
\begin{gather*}
\Phi_{\la}\big((aq)^{1/2};-(aq)^{1/2},(at)^{1/2},-(at)^{1/2};q,t\big).
\end{gather*}

\begin{Conjecture}\label{Con_Phi-conj}
For $\la$ a partition
\begin{gather}
\sum_{\mu\subseteq\la} (-1)^{\abs{\mu}}q^{-n(\mu')}t^{n(\mu)}
\frac{C_{\mu}^{+}(a,at/q;q,t)}
{C^0_{\mu}\big(aq,at;q^2,t^2\big)} \qbin{\la}{\mu}_{q,t,(at/q)^{1/2}} \nonumber\\
=\begin{cases}\displaystyle a^{\abs{\la}/2}
q^{2\hat{n}^{\textup{o}}(\la')-2\hat{n}^{\textup{e}}(\la')}
t^{-2n^{\textup{o}}(\la)}
\frac{C^{-,\textup{e}}_{\la}(q,t;q,t)C^{+,\textup{e}}_{\la}(a,at/q;q,t)}
{C^0_{\la}\big(aq,at;q^2,t^2\big)} & \text{if $\twocore{\la}=0$}, \vspace{1mm}\\
0 & \text{otherwise}.
\end{cases}\label{Eq_Phi-conj}
\end{gather}
\end{Conjecture}

For $\la=(N)$ this is \eqref{Eq_Andrews} with $b$ replaced by $at$.
By the $p=0$ case of~\eqref{Eq_PhiE-conj2} it follows that,
up to a rescaling of~$a$, \eqref{Eq_Phi-conj} is invariant under conjugation of~$\la'$.
Hence it also holds for $\la=(1^n)$.
Since the elliptic hypergeometric series
\begin{gather*}
\Phi_{\la}\big((aq)^{1/2};-(aq)^{1/2},(at)^{1/2},-(at)^{1/2};q,t;p\big)
\end{gather*}
is not very-well poised, it remains unclear what the elliptic analogue
Conjecture~\ref{Con_Phi-conj} should be.
Another obvious special case arises when $t=q$, in which case we can use
the determinantal expression~\eqref{Eq_determinantal} for the generalised binomial coefficients.
Up to an overall factor, the left-hand side of~\eqref{Eq_Phi-conj}
for $a\mapsto at^{2n-2}$ may then be written as
\begin{gather*}
\sum_{\nu_1,\dots,\nu_n=0}^{\kappa_1}
\prod_{1\leq i<j\leq n} \frac{\big(q^{\nu_i}-q^{\nu_j+1}\big)}
{\big(1-aq^{\nu_i+\nu_j}\big)} \det_{1\leq i,j\leq n}\left(
\frac{\big(aq^{\kappa_i},q^{-\kappa_i};q\big)_{\nu_j}\big(a;q^2\big)_{\nu_j}}
{(q,a;q)_{\nu_j}\big(aq;q^2\big)_{\nu_j}} q^{\nu_j} \right),
\end{gather*}
where $\kappa_i:=\la_i+n-i$. It is again unclear why this vanishes unless $\la\in P_{+}(n)$ has empty $2$-core.

We conclude with several conjectures closely related to
Conjecture~\ref{Con_Phi-conj}, such as new vanishing integrals in the sense
of~\cite{Rains05,RV07} and a number of new Littlewood-type identities.

Let
\begin{gather*}
\dup T(x):=\frac{1}{2^n n!(2\pi\iup)^n} \frac{\dup x_1}{x_1}\cdots\frac{\dup x_n}{x_n}
\end{gather*}
and, for $\abs{a},\abs{b},\abs{q},\abs{t}<1$,
\begin{align*}
Z_n(a,b;q,t)&:=\Int_{\mathbb{T}^n}
\prod_{i=1}^n \frac{\big(x_i^{\pm 2};q\big)_{\infty}}
{\big(ax_i^{\pm 2},bx_i^{\pm 2};q^2\big)_{\infty}}
\prod_{1\leq i<j\leq n} \frac{\big(x_i^{\pm} x_j^{\pm};q\big)_{\infty}}
{\big(tx_i^{\pm} x_j^{\pm};q\big)_{\infty}} \dup T(x) \\
&\hphantom{:}=\prod_{i=1}^n \frac{\big(t,abt^{n+i-2};q\big)_{\infty}}
{\big(q,t^i,-at^{i-1},-bt^{i-1};q\big)_{\infty}\big(abt^{2i-2};q^2\big)_{\infty}^2},
\end{align*}
where the explicit evaluation is a special case of Gustafson's $\mathrm{BC}_n$ analogue of the Askey--Wilson integral~\cite{Gustafson90}.

\begin{Conjecture}[vanishing integral]
Let $\la\in P_{+}(2n)$ and $a,b,q,t\in\mathbb{C}$
such that $\abs{a},\abs{b},\abs{q},\allowbreak \abs{t}<1$. Then
\begin{gather}
I_{\la}\big(a,b;q,t,t^n\big):=\frac{1}{Z_n(a,b;q,t)} \nonumber\\
\qquad{} \times
\Int_{\mathbb{T}^n} P_{\la}\big(x_1^{\pm},\dots,x_n^{\pm};q,t\big)
\prod_{i=1}^n \frac{\big(x_i^{\pm 2};q\big)_{\infty}}
{\big(ax_i^{\pm 2},bx_i^{\pm 2};q^2\big)_{\infty}}
\prod_{1\leq i<j\leq n} \frac{\big(x_i^{\pm} x_j^{\pm};q\big)_{\infty}}
{\big(tx_i^{\pm} x_j^{\pm};q\big)_{\infty}} \dup T(x)\label{Eq_vanishing-Iab}
\end{gather}
vanishes unless $\twocore{\la}=0$. Moreover
\begin{subequations}\label{Eq_nonvanishingcases}
\begin{gather}\label{Eq_IqtqtT}
I_{\la}(q,t;q,t,T) =
q^{n^{\textup{e}}(\la')-n^{\textup{o}}(\la')}
t^{2\hat{n}^{\textup{o}}(\la)-2\hat{n}^{\textup{e}}(\la)}
\frac{C^{0,\textup{e}}_{\la}\big(T^2;q,t\big)}
{C^{0,\textup{o}}_{\la}\big(qT^2/t;q,t\big)}
\frac{C_{\la}^{-,\textup{e}}(q;q,t)}{C_{\la}^{-,\textup{o}}(t;q,t)}
\intertext{and}
\label{Eq_I1qtqtT}
I_{\la}(1,qt;q,t,T) =
\frac{u_{\la}(q,t)+v_{\la}(q,t) T}{1+T}
\frac{C^{0,\textup{e}}_{\la}(T^2;q,t)}
{C^{0,\textup{o}}_{\la}\big(qT^2/t;q,t\big)}
\frac{C_{\la}^{-,\textup{e}}(q;q,t)}{C_{\la}^{-,\textup{o}}(t;q,t)},
\end{gather}
\end{subequations}
where
\begin{gather*}
u_{\la}(q,t) :=
q^{2\hat{n}^{\textup{o}}(\la')-2\hat{n}^{\textup{e}}(\la')}
t^{n^{\textup{e}}(\la)-n^{\textup{o}}(\la)}, \\
v_{\la}(q,t) :=
q^{2n^{\textup{e}}(\la')+2\hat{n}^{\textup{e}}(\la')
-2n^{\textup{o}}(\la')-2\hat{n}^{\textup{o}}(\la')}
t^{4\hat{n}^{\textup{o}}(\la)-4\hat{n}^{\textup{e}}(\la)
+n^{\textup{o}}(\la)-n^{\textup{e}}(\la)}.
\end{gather*}
\end{Conjecture}

That the integral on the right of \eqref{Eq_vanishing-Iab} depends on~$n$
only through $t^n$ follows from the fact that this integral is equal to
the $T=t^n$ instance of
\begin{gather*}
\big[\tilde{K}_0\big(x_1,\dots,x_n;q,t,T;a^{1/2},-a^{1/2},
b^{1/2},-b^{1/2}\big)\big] P_{\la}\big(x_1^{\pm},\dots,x_n^{\pm};q,t\big)
\end{gather*}
with $\tilde{K}_{\la}$ the lifted Koornwinder polynomial~\eqref{Eq_liftedK},
see~\cite{Rains05}.

When $t=q$ we have a proof of the vanishing part of the conjecture
and alternative expressions for the right-hand sides of
\eqref{Eq_IqtqtT} and \eqref{Eq_I1qtqtT} in terms of pfaffians.
Define $I_{\la}(a,b;q,T):=I_{\la}(a,b;q,q,T)$ and
$Z_n(a,b;q):=Z_n(a,b;q,q)$.

\begin{Proposition}\label{Prop_ab-vanishing}
Let $\la\in P_{+}(2n)$ and $a,b,q\in\mathbb{C}$ such that
$\abs{a},\abs{b},\abs{q}<1$. Then
\begin{gather}\label{Eq_vanish-ab}
\Int_{\mathbb{T}^n} s_{\la}\big(x_1^{\pm},\dots,x_n^{\pm}\big)
\prod_{i=1}^n \frac{\big(x_i^{\pm 2};q\big)_{\infty}}
{\big(ax_i^{\pm 2},bx_i^{\pm 2};q^2\big)_{\infty}}
\prod_{1\leq i<j\leq n} \big(1-x_i^{\pm} x_j^{\pm}\big) \dup T(x)
\end{gather}
vanishes unless $\la$ has empty $2$-core. Moreover,
\begin{subequations}
\begin{gather}
I_{\la}\big(q,q;q,q^n\big) =
\prod_{i=1}^n \frac{\big(1-q^{2i-1}\big)^{2n-2i+1}}{\big(1-q^{2i}\big)^{2n-2i}} \notag\\
\hphantom{I_{\la}\big(q,q;q,q^n\big) =}{}
\times\pf_{1\leq i,j\leq 2n}\left(
\frac{q^{(\la_i-\la_j-i+j-1)/2}}{1-q^{\la_i-\la_j-i+j}}
\chi(\la_i-\la_j-i+j \text{ is odd})\right) \label{Eq_Iqqq}
\intertext{and}
I_{\la}\big(1,q^2;q,q^n\big) =\frac{1}{2^n}
\frac{2}{1+q^n}
\prod_{i=1}^n \frac{(1-q^{2i-1})^{2n-2i+1}}{(1-q^{2i})^{2n-2i}}\notag\\
\hphantom{I_{\la}\big(1,q^2;q,q^n\big) =}{}\times\pf_{1\leq i,j\leq 2n}\left(
\frac{1+q^{\la_i-\la_j-i+j}}
{1-q^{\la_i-\la_j-i+j}}
\chi(\la_i-\la_j-i+j \text{ is odd})\right).\label{Eq_I1q2q}
\end{gather}
\end{subequations}
\end{Proposition}

Before we prove this, we remark that for any fixed choice of partition $\la$
with empty $2$-core each of the above two pfaffians can be written as
an $n\times n$ determinant containing the non-vanishing entries, up to sign. For example,
\begin{align*}
I_{(r^{2n})}\big(q,q;q,q^n\big) &=
\prod_{i=1}^n \frac{\big(1-q^{2i-1}\big)^{2n-2i+1}}{\big(1-q^{2i}\big)^{2n-2i}}
\pf_{1\leq i,j\leq 2n}\left(\frac{q^{(j-i-1)/2}}{1-q^{j-i}}
\chi(j-i \text{ is odd})\right) \\
&=\prod_{i=1}^n \frac{(1-q^{2i-1})^{2n-2i+1}}{(1-q^{2i})^{2n-2i}}
\det_{1\leq i,j\leq n}\left(
\frac{(-1)^{j-i} q^{j-i}}{1-q^{2j-2i+1}}\right) \\
&=1.
\end{align*}
We do not know, however, how to compute the pfaffian for arbitrary $\la$.
Of course, from~\eqref{Eq_IqtqtT} it follows that we must have
\begin{gather*}
 \prod_{i=1}^n \frac{\big(1-q^{2i-1}\big)^{2n-2i+1}}{\big(1-q^{2i}\big)^{2n-2i}}
 \pf_{1\leq i,j\leq 2n}\left(
\frac{q^{(\la_i-\la_j-i+j-1)/2}}{1-q^{\la_i-\la_j-i+j}}
\chi(\la_i-\la_j-i+j \text{ is odd})\right) \\
 \qquad{} =q^{n^{\textup{e}}(\la')-n^{\textup{o}}(\la')+2\hat{n}^{\textup{o}}(\la)
-2\hat{n}^{\textup{e}}(\la)}
\frac{C^{0,\textup{e}}_{\la}\big(q^{2n};q,q\big)}
{C^{0,\textup{o}}_{\la}\big(q^{2n};q,q\big)}
\frac{C_{\la}^{-,\textup{e}}(q;q,q)}{C_{\la}^{-,\textup{o}}(q;q,q)},
\end{gather*}
with a similar result for the second pfaffian.

\begin{proof}[Proof of Proposition \ref{Prop_ab-vanishing}]
Using~\eqref{Eq_Schur} to express the Schur function~$s_{\la}$
as a determinant, we have
\begin{gather}\label{Eq_Schur-int}
I_{\la}\big(a,b;q,q^n\big)=\frac{1}{Z_n(a,b;q)}
\Int_{\mathbb{T}^n} \det_{1\leq i,j\leq 2n}\big(y_i^{\la_j+2n-j}\big)
\prod_{i=1}^n \frac{x_i^{-1}\theta\big(x_i^2;q\big)}
{\big(ax_i^{\pm 2},bx_i^{\pm 2};q^2\big)_{\infty}} \dup T(x),
\end{gather}
where $y=(x_1,x_1^{-1},\dots,x_n,x_n^{-1})$.
Let $\big(a_{ij}(x)\big)$ be a $2n\times 2n$ matrix such that
$a_{2i-1,j}(x)=\phi_j(x_i)$ and $a_{2i,j}(x)=\psi_j(x_i)$.
Then \cite[equation~(7.3)]{deBruijn56}
\begin{gather*}
\frac{1}{n!} \int \det_{1\leq i,j\leq 2n}
\big(a_{ij}(x)\big)\dup\mu(x_1)\cdots\dup\mu(x_n)
=\pf_{1\leq i,j\leq 2n}\bigg( \int
\big(\phi_i(x)\psi_j(x)-\phi_j(x)\psi_i(x)\big)\dup\mu(x)\bigg).
\end{gather*}
Applying this to \eqref{Eq_Schur-int} yields
\begin{gather*}
I_{\la}\big(a,b;q,q^n\big) =\frac{1}{2^n Z_n(a,b;q)} \\
\hphantom{I_{\la}\big(a,b;q,q^n\big) =}{} \times
\pf_{1\leq i,j\leq 2n}\bigg(\frac{1}{2\pi\iup} \Int_{\mathbb{T}}
\big(z^{\la_i-\la_j-i+j}-z^{\la_j-\la_i+i-j}\big)
\frac{z^{-1}\theta(z^2;q)}
{\big(az^{\pm 2},bz^{\pm 2};q^2\big)_{\infty}} \frac{\dup z}{z}\bigg).
\end{gather*}
The $(i,j)$-entry of above pfaffian vanishes if $\la_i-\la_j-i+j$ is even.
For the pfaffian to not vanish the set
\begin{gather*}
\{\la_1+2n-1,\dots,\la'_{2n-1}+1,\la_{2n}\}
\end{gather*}
must thus have $n$ even and $n$ odd elements.
By Lemma~\ref{Lem_2core} this is exactly the case if $\la$ has empty $2$-core, which settles the non-vanishing part of the proposition.

For the second part, we take $a=q^{1-\alpha}$ and $b=q^{\alpha+1}$ with
$\alpha\in\{0,1\}$ so that
\begin{gather*}
\frac{\theta(z^2;q)}{\big(az^{\pm 2},bz^{\pm 2};q^2\big)_{\infty}^2}=
\frac{\theta\big(q^{\alpha}z^2;q^2\big)}{\theta\big(q^{\alpha+1}z^2;q^2\big)_{\infty}}.
\end{gather*}
\looseness=-1 Assume now that $\la_i-\la_j-i+j$ is odd, say $2k+1$.
Then the $(i,j)$-entry of the pfaffian is given by
\begin{gather*}\allowdisplaybreaks
 \frac{1}{2\pi\iup} \Int_{\mathbb{T}}
\big(z^{2k+1}-z^{-2k-1}\big)
\frac{z^{-1}\theta\big(q^{\alpha}z^2;q^2\big)}
{\theta\big(q^{\alpha+1}z^2;q^2\big)} \frac{\dup z}{z} \\
 \qquad {}= \frac{1}{2\pi\iup} \sum_{i,j=0}^{\infty} \Int_{\mathbb{T}}
\big(z^{2k+1}-z^{-2k-1}\big) z^{2i-2j-1} q^{i+j+\alpha(i-j)}
\frac{\big(1/q;q^2\big)_i\big(q;q^2\big)_j}{\big(q^2;q^2\big)_i\big(q^2;q^2\big)_j}
 \frac{\dup z}{z} \\
 \qquad{} = \sum_{i=0}^{\infty} q^{2i-(\alpha-1)k}
\frac{\big(1/q;q^2\big)_i\big(q;q^2\big)_{i+k}}{\big(q^2;q^2\big)_i\big(q^2;q^2\big)_{i+k}}-
\sum_{i=0}^{\infty} q^{2i+(\alpha-1)(k+1)}
\frac{\big(1/q;q^2\big)_i\big(q;q^2\big)_{i-k-1}}{\big(q^2;q^2\big)_i\big(q^2;q^2\big)_{i-k-1}} \\
 \qquad{} = \frac{q^{-(\alpha-1)k}}{1-q^{2k+1}}
\frac{\big(q;q^2\big)_{\infty}}{\big(q^2;q^2\big)_{\infty}}
-\frac{q^{(\alpha-1)(k+1)}}{1-q^{-2k-1}}
\frac{\big(q;q^2\big)_{\infty}}{\big(q^2;q^2\big)_{\infty}} \notag \\
 \qquad{} =\frac{q^{-(\alpha-1)k}+q^{\alpha+(\alpha+1)k}}{1-q^{2k+1}}
\frac{\big(q;q^2\big)_{\infty}^2}{\big(q^2;q^2\big)_{\infty}^2}, \notag
\end{gather*}
where the second equality follows from double use of the $q$-binomial theorem
\cite[equation~(II.3)]{GR04}
\begin{gather*}
{_1\phi_0}(a;\text{--}\,;q,z)=\frac{(az;q)_{\infty}}{(z;q)_{\infty}}
\end{gather*}
and the second-last equality from the $q$-Gauss sum \cite[equation~(II.8)]{GR04}
\begin{gather*}
{_2\phi_1}(a,b;c;q,z)=\frac{(c/a,c/b;q)_{\infty}}{(c,c/ab;q)_{\infty}}.
\end{gather*}
Since
\begin{gather*}
Z_n\big(q^{1-\alpha},q^{\alpha+1};q\big)=\frac{1+q^{\alpha n}}{2}
\left(\frac{\big(q;q^2\big)_{\infty}}{\big(q^2;q^2\big)_{\infty}}\right)^{2n}
\prod_{i=1}^n \frac{\big(1-q^{2i}\big)^{2n-2i}}{\big(1-q^{2i-1}\big)^{2n-2i+1}},
\end{gather*}
we obtain~\eqref{Eq_Iqqq} and~\eqref{Eq_I1q2q}.
\end{proof}

By \cite[Proposition~4.9]{RW15} the integrals~\eqref{Eq_IqtqtT}
and~\eqref{Eq_I1qtqtT} are equivalent to a pair of bounded Littlewood
identities.
To state these we define, for $\la$ a partition such that $\twocore{\la}=0$,
\begin{gather*}
\kappa^{(1)}_{\la}(z;q,t) :=q^{n^{\text{e}}(\la')-n^{\text{o}}(\la')}
t^{2\hat{n}^{\text{o}}(\la)-2\hat{n}^{\text{e}}(\la)}
\left(\frac{q}{t}\right)^{\abs{\la}/2}
\frac{C^{0,\text{e}}_{\la}\big(z^2;q,t\big)}{C^{0,\text{o}}_{\la}\big(z^2q/t;q,t\big)}
 \frac{C^{-,\text{e}}_{\la}(t;q,t)}{C^{-,\text{o}}_{\la}(q;q,t)}, \\
\kappa^{(2)}_{\la}(z;q,t) :=
\frac{q^{2\hat{n}^{\text{o}}(\la')-2\hat{n}^{\text{e}}(\la')}
t^{n^{\text{e}}(\la)-n^{\text{o}}(\la)}+
z q^{4\hat{n}^{\text{e}}(\la')-4n^{\text{o}}(\la')}
t^{2n^{\text{o}}(\la)+3\hat{n}^{\text{o}}(\la)-5\hat{n}^{\text{e}}(\la)}}{1+z}
\notag \\
\hphantom{\kappa^{(2)}_{\la}(z;q,t) :=}{} \times \left(\frac{q}{t}\right)^{\abs{\la}/2}
\frac{C^{0,\text{e}}_{\la}\big(z^2;q,t\big)}{C^{0,\text{o}}_{\la}\big(z^2q/t;q,t\big)}
\frac{C^{-,\text{e}}_{\la}(t;q,t)}{C^{-,\text{o}}_{\la}(q;q,t)},
\end{gather*}
and set $\kappa_{\la}^{(1)}(z;q,t)=\kappa_{\la}^{(2)}(z;q,t)=0$ if $\la$ has a non-trivial $2$-core.
As follows from Lemmas~\ref{Lem_two-core} and~\ref{Lem_C-conjugate-eo},
in the non-vanishing case these two functions are related to
$I_{\la}(q,t;q,t,T)$ and $I_{\la}(1,qt;q,t,T)$ as conjectured on the
right-hand side of~\eqref{Eq_nonvanishingcases} via
\begin{gather*}
\kappa_{\la}^{(1)}(z;q,t)=I_{\la'}(t,q;t,q,1/z)
\qquad\text{and}\qquad
\kappa_{\la}^{(2)}(z;q,t)=I_{\la'}(1,qt;t,q,1/z).
\end{gather*}

\begin{Conjecture}[bounded Littlewood identities] \label{Con_new-Littlewood-bounded}
For $m$, $n$ nonnegative integers,
\begin{gather*}
\sum_{\la\in P_{+}(n)} \kappa^{(1)}_{\la}(q^{-m};q,t) P_{\la}(x;q,t)
=(x_1\cdots x_n)^m K_{(m^n)}\big(x;q,t;q^{1/2},-q^{1/2},t^{1/2},-t^{1/2}\big)
\end{gather*}
and
\begin{gather*}
\sum_{\la\in P_{+}(n)} \kappa^{(2)}_{\la}(q^{-m};q,t) P_{\la}(x;q,t)
=(x_1\cdots x_n)^m K_{(m^n)}\big(x;q,t;1,-1,(qt)^{1/2},-(qt)^{1/2}\big).
\end{gather*}
\end{Conjecture}

Both $\kappa^{(1)}_{\la}(q^{-m};q,t)$ and
$\kappa^{(2)}_{\la}(q^{-m};q,t)$ vanish if $\la_1>2m$ so that
only partitions $\la$ contained in the rectangle $((2m)^n)$ with
$\twocore{\la}=0$ contribute to the sum.
We also remark that bounded Littlewood identities and integrals such as~\eqref{Eq_vanishing-Iab} satisfy a duality
in which the Koornwinder parameters $t_r$ (for $0\leq r\leq t_3$)
are mapped to $-(qt)^{1/2}/t_r$, see~\cite{RW15}.
With respect to this duality both vanishing integrals and
bounded Littlewood identities are self-dual.

Finally, define
\begin{gather*}
\kappa^{(1)}_{\la}(q,t) :=
\lim_{z\to\infty} \kappa^{(1)}_{\la}(z;q,t)=
q^{2 \hat{n}^{\text{o}}(\la')-2\hat{n}^{\text{e}}(\la')}
t^{n^{\text{e}}(\la)-n^{\text{o}}(\la)}
\frac{C^{-,\text{e}}_{\la}(t;q,t)}{C^{-,\text{o}}_{\la}(q;q,t)}, \\
\kappa^{(2)}_{\la}(q,t) :=
\lim_{z\to\infty} \kappa^{(2)}_{\la}(z;q,t)=
q^{n^{\text{e}}(\la')-n^{\text{o}}(\la')}
t^{2\hat{n}^{\text{o}}(\la)-2\hat{n}^{\text{e}}(\la)}
\frac{C^{-,\text{e}}_{\la}(t;q,t)}{C^{-,\text{o}}_{\la}(q;q,t)}.
\end{gather*}
In the large-$m,n$ limit, Conjecture~\ref{Con_new-Littlewood-bounded} then simplifies to the following pair of unbounded Littlewood identities.

\begin{Conjecture}[Littlewood identities]
We have
\begin{gather*}
\sum_{\la} \kappa^{(1)}_{\la}(q,t) P_{\la}(x;q,t)=
\prod_{i\geq 1} \frac{\big(tx_i^2;q^2\big)_{\infty}}{\big(x_i^2;q^2\big)_{\infty}}
\prod_{i<j} \frac{(tx_ix_j;q)_{\infty}}{(x_ix_j;q)_{\infty}}
\end{gather*}
and
\begin{gather*}
\sum_{\la} \kappa^{(2)}_{\la}(q,t) P_{\la}(x;q,t)=
\prod_{i\geq 1} \frac{\big(qtx_i^2;q^2\big)_{\infty}}{\big(qx_i^2;q^2\big)_{\infty}}
\prod_{i<j} \frac{(tx_ix_j;q)_{\infty}}{(x_ix_j;q)_{\infty}}.
\end{gather*}
\end{Conjecture}

The above identities are no longer self-dual and, as follows from
Lemma~\ref{Lem_C-conjugate-eo} and \eqref{Eq_omegaPQ} as well as
Lemma~\ref{Lem_omegaqt} below, they form a dual pair with respect
to $\omega_{q,t}$.

\begin{Lemma}\label{Lem_omegaqt}
Let $\omega_{q,t}$ be the automorphism of $\Lambda_{\mathbb{Q}(q,t)}$
given by \eqref{Eq_omegaqt}.
Then
\begin{gather*}
\omega_{q,t} \bigg(\prod_{i\geq 1}
\frac{\big(tx_i^2;q^2\big)_{\infty}}{\big(x_i^2;q^2\big)_{\infty}}
\prod_{i<j} \frac{(tx_ix_j;q)_{\infty}}{(x_ix_j;q)_{\infty}}\bigg)
=\prod_{i\geq 1} \frac{\big(qtx_i^2;t^2\big)_{\infty}}{\big(tx_i^2;t^2\big)_{\infty}}
\prod_{i<j} \frac{(qx_ix_j;t)_{\infty}}{(x_ix_j;t)_{\infty}}.
\end{gather*}
\end{Lemma}

\begin{proof} We have
\begin{gather*}
\mathcal{L}(a;q,t) :=\prod_{i\geq 1} \frac{(atx_i^2;q^2)_{\infty}}{(ax_i^2;q^2)_{\infty}}
\prod_{i<j} \frac{(tx_ix_j;q)_{\infty}}{(x_ix_j;q)_{\infty}} \notag \\
\hphantom{\mathcal{L}(a;q,t)}\; =\sum_{r\geq 0} \bigg(
\sum_{i\geq 1}\big(
\log\big(1-q^{2r}atx_i^2\big)-\log\big(1-aq^{2r}x_i^2\big)\big) \\
\hphantom{\mathcal{L}(a;q,t) :=}\hphantom{\sum_{r\geq 0}}{} + \sum_{i<j}\big(
\log\big(1-q^rtx_ix_j\big)-\log\big(1-q^rx_ix_j\big)\big)\bigg).
\end{gather*}
Using $\log(1-x)=-\sum_{n\geq 1} x^n/n$, the sum over $r$ can be carried out by the geometric series.
Since $\sum_{i<j} (x_ix_j)^n=(p_n^2-p_{2n})/2$, we thus find
\begin{gather*}
\mathcal{L}(a;q,t)=\sum_{n\geq 1}\frac{1-t^n}{n} \left(
\frac{a^n p_{2n}}{1-q^{2n}}+\frac{1}{2}
\frac{p_n^2-p_{2n}}{1-q^n}\right).
\end{gather*}
In particular,
\begin{gather*}
\mathcal{L}(q;q,t)=\frac{1}{2} \sum_{n\geq 1}\frac{1-t^n}{n}\left(\frac{p_n^2}{1-q^n}-\frac{p_{2n}}{1+q^n}\right).
\end{gather*}
and
\begin{gather*}
\mathcal{L}(1;q,t)=\frac{1}{2}\sum_{n\geq 1}\frac{1-t^n}{n} \left(\frac{p_n^2}{1-q^n}+ \frac{p_{2n}}{1+q^n}\right).
\end{gather*}
Applying $\omega_{q,t}$ to this last expression yields
\begin{gather*}
\mathcal{L}(1;q,t)=\frac{1}{2}\sum_{n\geq 1}\frac{1-q^n}{n}
\left(\frac{p_n^2}{1-t^n}-\frac{p_{2n}}{1+t^n}\right)=\mathcal{L}(t;t,q),
\end{gather*}
completing the proof.
\end{proof}

\section*{Postscript} \label{page_PS-start}
One of the referees posed the question as to how the branching rule~\eqref{Eq_CC} compares to conjectural branching formulas of Hoshino and Shiraishi~\cite{HS18} between certain type~$\mathrm{A}$ and~$\mathrm{C}$
asymptotically-free eigenfunctions of Macdonald operators. In this postscript we address the referee's question.

Hoshino and Shiraishi \cite[Section 9]{HS18} considered functions
\begin{gather*}
\tilde{\varphi}^{(\mathrm{A}_{2n-1})}(s\vert x\vert q,t)
\qquad\text{and}\qquad \varphi^{(\mathrm{C}_n)}(s\vert x\vert q,t),
\end{gather*}
where $s=(s_1,\dots,s_n)$ and $x=(x_1,\dots,x_n)$, such that for $\la\in P_{+}(n)$
\begin{subequations}\label{Eq_HS-AC}
\begin{gather}\label{Eq_HS-A}
x^{\la} \tilde{\varphi}^{(\mathrm{A}_{2n-1})}
\big(t^n q^{\la_1},\dots,t q^{\la_n}\vert x\vert q,t\big) =P_{\la}\big(x^{\pm};q,t\big), \\
x^{\la} \varphi^{(\mathrm{C}_n)} \big(t^n q^{\la_1},\dots,t q^{\la_n}\vert x\vert q,t\big)
 =P_{\la}^{(\mathrm{C}_n,\mathrm{C}_n)}(x;q,t,t),\label{Eq_HS-C}
\end{gather}
where $x^{\la}:=x_1^{\la_1}\cdots x_n^{\la_n}$.
\end{subequations}
Here we note that the function $\tilde{\varphi}^{(\mathrm{A}_{2n-1})}$
arises by folding the as\-ymp\-tot\-i\-cal\-ly-free solution of
the trigonometric Ruijsenaars model~\cite{NS12,Shiraishi05} of rank $2n-1$,
see~\cite{HS18} for more details. In the rank-one case,
\begin{gather}\label{Eq_rank-one}
\tilde{\varphi}^{(\mathrm{A}_1)}(s_1\vert x_1\vert q,t)=
\varphi^{(\mathrm{C}_1)}(s_1\vert x_1\vert q,t)=
\qhypc{2}{1}\left[\genfrac{}{}{0pt}{}
{t,ts_1}{qs_1};q,\frac{qx_1^2}{t}\right].
\end{gather}

Let $\theta=\{\theta_{ij}\}_{1\leq i<j\leq n}$ be a set of nonnegative
integers and, for $1\leq i\leq n$, define
\begin{gather*}
\phi_i:=\sum_{j=1}^{i-1} \theta_{ji}+\sum_{j=i+1}^n \theta_{ij}
\end{gather*}
and $\phi:=(\phi_1,\dots,\phi_n)$.
Hoshino and Shiraishi considered the problem of determining the rational
functions $e_n(s;\theta;q,t)$ for $n\geq 1$, such that the following
branching rule holds:
\begin{gather}\label{Eq_HS-branching}
\tilde{\varphi}^{(\mathrm{A}_{2n-1})}(s\vert x\vert q,t)
=\sum_{\theta} e_n(s;\theta;q,t)
\varphi^{(\mathrm{C}_n)}\big(s q^{-\phi}\big\vert x\big\vert q,t\big)
x^{-\phi},
\end{gather}
where $sq^{-\phi}:=\big(s_1q^{-\phi_1},\dots,s_nq^{-\phi_n}\big)$ and
$x^{-\phi}:=x_1^{-\phi_1}\cdots x_n^{-\phi_n}=
\prod_{1\leq i<j\leq n} (x_ix_j)^{-\theta_{ij}}$.

By \eqref{Eq_rank-one}, $e_1(s_1;\leeg;q,t)=1$.
In \cite[Conjectures~9.12 and 9.13]{HS18} Hoshino and Shiraishi made the
following two conjectures.

\begin{Conjecture}\label{Conj_HS}
We have
\begin{gather*}
e_2(s_1,s_2;\theta_{12};q,t)=
\frac{\big(t,t/s_1,t/s_2,q^{\theta_{12}+1}/ts_1s_2;q\big)_{\theta_{12}}}
{\big(q,q/s_1,q/s_2,q^{\theta_{12}}/s_1s_2;q\big)_{\theta_{12}}}
\left(\frac{q}{t}\right)^{\theta_{12}},
\end{gather*}
and
\begin{gather*}
 e_3(s_1,s_2,s_3;\theta_{12},\theta_{13},\theta_{23};q,t)
 = \prod_{i=1}^3 \frac{(t/s_i;q)_{\phi_i}}{(q/s_i;q)_{\phi_i}}\\
\qquad{}\times
\prod_{1\leq i<j\leq 3}
\frac{(ts_j/s_i,q^{1-\theta_{j,6-i-j}}s_j/ts_i;q)_{\theta_{1,5-j}}}
{\big(qs_j/s_i,q^{-\theta_{j,6-i-j}}s_j/s_i;q\big)_{\theta_{1,5-j}}}
\frac{\big(t,q^{1+\sum_{1\leq k<l\leq 3}\theta_{kl}}/ts_is_j;q\big)_{\theta_{ij}}}
{\big(q,q^{\sum_{1\leq k<l\leq 3}\theta_{kl}}/s_is_j;q\big)_{\theta_{ij}}}
\left(\frac{q}{t}\right)^{\theta_{ij}},
\end{gather*}
where $\theta_{ji}:=\theta_{ij}$ for $1\leq i<j\leq 3$.
\end{Conjecture}

To now answer the question of the referee, we first notice that
\eqref{Eq_HS-C} can be used to define
$P_{\la}^{(\mathrm{C}_n,\mathrm{C}_n)}(x;q,t,t)$ for integer sequences
$\la=(\la_1,\dots,\la_n)$ that are not necessarily partitions.
In~\eqref{Eq_HS-branching} we can then specialise $s_i=q^{\la_i}t^{n-i+1}$ for
$1\leq i\leq n$, multiply both sides by $x^{\la}$, and use~\eqref{Eq_HS-A}
and~\eqref{Eq_HS-C}, to obtain
\begin{gather*}
P_{\la}(x^{\pm};q,t)=\sum_{\theta}
e_n\big(q^{\la_1}t^n,\dots,q^{\la_n}t;\theta;q,t\big)
P^{(\mathrm{C}_n,\mathrm{C}_n)}_{(\la_1-\phi_1,\dots,\la_n-\phi_n)}(x;q,t,t).
\end{gather*}
Assuming Conjecture~\ref{Conj_HS},
\begin{gather*}
e_2\big(q^{\la_1}t^2,q^{\la_2}t;\theta_{12};q,t\big)=0
\end{gather*}
unless $\theta_{12}\leq\la_2$, and
\begin{gather*}
e_3\big(q^{\la_1}t^3,q^{\la_2}t^2,q^{\la_3}t; \theta_{12},\theta_{13},\theta_{23};q,t\big)=0
\end{gather*}
unless $\theta_{12}\leq\la_2-\la_3$, $\theta_{13}\leq\la_1-\la_2$ and
$\theta_{13}+\theta_{23}\leq\la_3$.
For $n=2$ this implies that $\la_1-\phi_1=\la_1-\theta_{12}\geq 0$, and for
$n=3$ it implies
\begin{gather*}
(\la_1-\phi_1)-(\la_2-\phi_2)=\la_1-\la_2-\theta_{13}+\theta_{23}\geq \theta_{23}\geq 0, \\
(\la_2-\phi_2)-(\la_3-\phi_3)=\la_2-\la_3-\theta_{12}+\theta_{13} \geq \theta_{13}\geq 0, \\
\la_3-\phi_3=\la_3-\theta_{13}-\theta_{23}\geq 0.
\end{gather*}
It is thus reasonable to conjecture that
\begin{gather*}
e_n\big(q^{\la_1}t^n,\dots,q^{\la_n}t;\theta;q,t\big)=0
\end{gather*}
if $(\la_1-\phi_1,\dots,\la_n-\phi_n)\notin P_{+}(n)$, so that
\begin{gather*}
P_{\la}(x^{\pm};q,t)=\sum_{\substack{\theta \\
(\la_1-\phi_1,\dots,\la_n-\phi_n)\in P_{+}(n)}}
e_n\big(q^{\la_1}t^n,\dots,q^{\la_n}t;\theta;q,t\big)
P^{(\mathrm{C}_n,\mathrm{C}_n)}_{(\la_1-\phi_1,\dots,\la_n-\phi_n)}(x;q,t,t).
\end{gather*}
If $\la=(m^r)$ for $0\leq r\leq n$ we may compare the above with the branching rule~\eqref{Eq_CC}.
Let $\rho:=\floor{r/2}$, $I_n:=\{(i,j)\colon 1\leq i<j\leq n\}$,
\begin{gather*}
J_r:=\{(r-2i+1,r-2i+2)\colon 1\leq i\leq \rho\} \qquad\text{and}\qquad \bar{J}_{r,n}:=I_n\setminus J_r.
\end{gather*}
It then follows from \eqref{Eq_CC} that
\begin{gather*}
e_n\big(q^mt^n,\dots,q^m t^{n-r+1},t^{n-r},\dots,t;\theta;q,t\big)
\end{gather*}
is non-vanishing if and only if
\begin{gather*}
\theta_{ij}=0 \quad \text{for $(i,j)\in\bar{J}_{n,r}$}, \\
m\geq\theta_{r-1,r}\geq\theta_{r-3,r-2}\geq\cdots\geq\theta_{r-2\rho+1,r-2\rho+2}\geq 0,
\end{gather*}
in which case
\begin{alignat*}{3}
& \phi_{r-2i+1}=\phi_{2-2i+2}=\theta_{r-2i+1,2i-2i+2} \qquad && \text{for $1\leq i\leq\rho$},& \\
& \phi_i=0 \qquad && \text{otherwise}. &
\end{alignat*}
Moreover, for such non-vanishing $\theta$ we have the following identification
of $e_n$ and the function~$c_{\la^2}$ defined in~\eqref{Eq_c-def}.
Let $\mu\subset (m^r)$ such that $\mu'$ is even (i.e., $\mu=\la^2$)
and let $\theta$ be fixed in terms of $\mu$ (or $\la$) as
\begin{alignat*}{3}
& \theta_{r-2i+1,r-2i+2}=\mu_{2i-1}=\la_i \qquad &&\text{for $1\leq i\leq\rho$}, & \\
& \theta_{ij}=0\qquad && \text{for $(i,j)\in\bar{J}_{n,r}$}. &
\end{alignat*}
Then
\begin{gather}\label{Eq_en-c}
e_n\big(q^mt^n,\dots,q^mt^{n-r+1},t^{n-r},\dots,t;\theta;q,t\big)
=c_{\la^2}\big(q^{-m},t^{-(n-r)};q,t\big).
\end{gather}
It is readily checked that the rational functions $e_2$ and $e_3$ of
Conjecture~\ref{Conj_HS} satisfy~\eqref{Eq_en-c}, showing that the
conjecture is consistent with the branching rule~\eqref{Eq_CC}.
In particular, we have the slightly more general
\begin{gather*}
e_2\big(t^2s,t;\theta_{12};q,t\big) =\delta_{\theta_{12},0}, \\
e_2\big(t^2s,ts;\theta_{12};q,t\big) =c_{(\theta_{12},\theta_{12})}(1/s,1;q,t)
\end{gather*}
and\label{page_PS-end}
\begin{gather*}
e_3\big(t^3s,t^2,t;\theta_{12},\theta_{13},\theta_{23};q,t\big)=
\delta_{\theta_{12},0} \delta_{\theta_{13},0} \delta_{\theta_{23},0}, \\
e_3\big(t^3s,t^2s,t;\theta_{12},\theta_{13},\theta_{23};q,t\big)=
c_{(\theta_{12},\theta_{12})}(1/s,1/t;q,t)
\delta_{\theta_{13},0} \delta_{\theta_{23},0},\\
e_3\big(t^3s,t^2s,ts;\theta_{12},\theta_{13},\theta_{23};q,t\big) =
c_{(\theta_{23},\theta_{23})}(1/s,1;q,t)
\delta_{\theta_{12},0} \delta_{\theta_{13},0}.
\end{gather*}

\subsection*{Acknowledgements}
We thank one of the referees of our paper for suggesting we compare the branching
rule~\eqref{Eq_CC} with \cite[Conjectures~9.12 and 9.13]{HS18} by
Hoshino and Shiraishi. This work was supported by the Australian Research Council Discovery Grant DP170102648
and a KIAS Individual Grant (MG067302) at Korea Institute for Advanced Study.

\pdfbookmark[1]{References}{ref}
\LastPageEnding

\end{document}